\documentclass[11pt,letterpaper]{amsart}

\usepackage[english]{babel}
\usepackage[OT2,T1]{fontenc} % for the Sha symbol

\usepackage{mathtools}
\usepackage{amssymb}
\usepackage{stmaryrd}
\usepackage{bm}
\usepackage{tikz-cd}

\usepackage{comment}
\usepackage{todonotes}
\usepackage{xcolor}
\usepackage{enumitem}
\usepackage{graphicx}
\usepackage{setspace}
\usepackage{url}
\usepackage{xr}

\usepackage[title]{appendix}
\usepackage{chngcntr}
\usepackage{apptools}
\usepackage{thmtools}

\usepackage[pagebackref]{hyperref}
\usepackage[nameinlink]{cleveref}
\usepackage{bookmark}

\definecolor{crimsonglory}{rgb}{0.75, 0.0, 0.2}
\definecolor{darkpowderblue}{rgb}{0.0, 0.2, 0.6}

\hypersetup{
	colorlinks = true,
	urlcolor   = darkpowderblue,
	linkcolor  = darkpowderblue,
	citecolor  = crimsonglory
}

\theoremstyle{plain}
\newtheorem{theorem}{Theorem}[section]
\newtheorem{definition}[theorem]{Definition}

\newtheorem{prop}[theorem]{Proposition}
\newtheorem{cor}[theorem]{Corollary}
\newtheorem{lemma}[theorem]{Lemma}

\newtheorem{remark}[theorem]{Remark}
\newtheorem{remarks}[theorem]{Remarks}
\newtheorem{conj}[theorem]{Conjecture}
\newtheorem{assumption}[theorem]{Assumption}

\DeclareMathOperator{\End}{End}
\DeclareMathOperator{\GL}{GL}

\DeclareMathOperator{\pr}{pr}
\DeclareMathOperator{\tr}{tr}

\DeclareMathOperator{\spec}{Spec}
\DeclareMathOperator{\gal}{Gal}
\DeclareMathOperator{\aut}{Aut}

\DeclareMathOperator{\Span}{Span}

\DeclareMathOperator{\Atyp}{Atyp}

\DeclareMathOperator{\zar}{Zar}

\DeclareMathOperator{\crys}{crys}

\DeclareMathOperator{\SP}{Sp}
\DeclareMathOperator{\spli}{split}
\DeclareMathOperator{\bad}{bad}
\DeclareMathOperator{\new}{new}
\DeclareMathOperator{\good}{good}
\DeclareMathOperator{\simord}{simord}
\DeclareMathOperator{\ssing}{ssing}
\DeclareMathOperator{\nssing}{nssing}
\DeclareMathOperator{\Frac}{Frac}
\DeclareMathOperator{\adj}{adj}
\DeclareMathOperator{\geomisog}{geomisog}
\newcommand{\im}[1]{Im(#1)}
\newcommand{\Sum}[2]{\displaystyle\sum_{#1}^{#2}}

\newcommand{\Z}{\mathbb{Z}}
\newcommand{\N}{\mathbb{N}}
\renewcommand{\C}{\mathbb{C}}
\newcommand{\Q}{\mathbb{Q}}
\newcommand{\B}{\mathbb{B}}

\newcommand{\F}{\mathbb{F}}
\newcommand{\G}{\mathbb{G}}

\newcommand{\CX}{\mathcal{X}}
\newcommand{\ld}{,\ldots,}
\newcommand{\CP}{\mathcal{P}}
\newcommand{\CO}{\mathcal{O}}
\newcommand{\DD}{\Delta\!\!\!\!\Delta}

\DeclareSymbolFont{cyrletters}{OT2}{wncyr}{m}{n}
\DeclareMathSymbol{\Sha}{\mathalpha}{cyrletters}{"58}

\title{On the $v$-adic values of $G$-functions}
\author{Georgios Papas}
\address{Faculty of Mathematics and Computer Science\\
	The Weizmann Institute of Science\\
	234 Herzl Street\\
	Rehovot 76100\\
	Israel, and\\
	Institute for Advanced Study\\
	1 Einstein Drive\\
	Princeton, N.J. 08540\\
	U.S.A.}

\email{georgios.papas@weizmann.ac.il}
\email{gpapas@ias.edu}

\begin{document}
	\maketitle
	
	\begin{abstract}
		We study $G$-functions centered at smooth fibers of one--parameter families of principally polarized abelian varieties. Building on work of Yves Andr\'e, we associate to such families a collection of $G$-functions whose local values admit period interpretations at archimedean places and at finite places of good reduction of the distinguished fiber, and whose radii of convergence control reduction-theoretic proximity at the latter places. We also discuss a conjectural extension to places of bad semistable reduction using Hyodo--Kato cohomology.
		
		Specializing to $\mathcal{A}_2$, we construct non-trivial algebraic relations among the values of these $G$-functions at $E^2$- and $E\times\mathrm{CM}$-points in both the archimedean and $p$-adic settings. At non-supersingular finite places of good reduction, we construct relations independent of the place. Finally, we obtain height bounds via Andr\'e-Bombieri's Hasse principle and discuss applications to large Galois orbit estimates and Zilber--Pink finiteness results for curves in $\mathcal{A}_2$.
	\end{abstract}
	
	%\part{Introduction}
	\section{Introduction}\label{section:intro}

Let $Z\subset\mathcal A_g$ be a Hodge-generic curve, that is, a curve not contained in any proper special subvariety. The Zilber--Pink conjecture predicts that only finitely many points of $Z$ lie on special subvarieties of codimension at least two. 
For further background on the Zilber--Pink conjecture and related methods, see \cite{pilabook,zilberpinkbook}.
%CLEAN

The Pila--Zannier strategy \cite{pilazannier} combines o-minimal counting with lower bounds for Galois orbits to address such finiteness problems. In the setting considered here, the latter can be obtained from height bounds and isogeny estimates. Such height bounds can be obtained using the arithmetic properties of $G$-functions. These power series were introduced and first studied by Siegel \cite{siegel}; Bombieri \cite{bombg} and Andr\'e \cite{andre1989g} subsequently developed methods for studying their values at algebraic points.
%CLEAN

The connection with height bounds is provided by the Andr\'e--Bombieri Hasse principle \cite{bombg,andre1989g}. In the form used here, one seeks a polynomial relation among the values of a family of $G$-functions at an algebraic point that holds at every place where the functions converge (i.e. is global) and does not arise as the specialization of a polynomial functional relation among the functions themselves (i.e. is non-trivial). Andr\'e--Bombieri's Hasse principle then bounds the height of the point in terms of the degree of such a relation. Daw and Orr \cite{daworr,daworr2,daworr3,daworr4,daworr5} applied this method to obtain cases of the Zilber--Pink conjecture for curves in Shimura varieties.
%CLEAN

In much of the work applying $G$-functions to unlikely intersections, the functions are centered at points where a family of abelian varieties degenerates. Under suitable hypotheses, the Gauss--Manin connection gives rise to $G$-functions centered at such points. Their archimedean values admit a cohomological interpretation through the relative de Rham--Betti comparison, following Andr\'e \cite{andre1989g}; corresponding non-archimedean interpretations were developed by Daw and Orr \cite{daworr4,daworr5}.
%CLEAN

These degenerations correspond to points on the boundary of a compactification of the ambient moduli space. However, for $g\geq 2$, $\mathcal A_g$ contains complete curves whose closures do not meet the boundary. This motivates the study of $G$-functions centered at smooth fibers of families of abelian varieties.
%CLEAN

Building on the work of Andr\'e \cite{andre1989g,andremots}, we construct $G$-functions centered at a distinguished smooth fiber of a one--parameter family of principally polarized abelian varieties and interpret their local values in terms of periods. This construction, together with the determination of the polynomial functional relations among the associated $G$-functions, is carried out in arbitrary relative dimension. 

For families of abelian surfaces, we construct non-trivial algebraic relations among these values from splittings of the fibers, with particular attention to their dependence on the place. These relations yield height bounds and applications to the Zilber--Pink conjecture in $\mathcal A_2$.
%CLEAN

A central feature of these constructions is the systematic use of the additional structure carried by the crystalline cohomology of abelian varieties to obtain polynomial relations with algebraic coefficients. Although the underlying tools are classical, their use to construct algebraic relations among values of $G$-functions is a principal methodological contribution of this work. The approach is not intrinsically restricted to abelian surfaces and suggests analogous constructions in $\mathcal A_g$; see \Cref{section:splittingsinag}.
%CLEAN

\subsection{$G$-functions centered at smooth fibers}

Let $f:\CX\rightarrow S$ be a one--parameter family of principally polarized abelian varieties of relative dimension $g$, defined over a number field $K$, and let $s_0\in S(K)$. After replacing $K$ by a finite extension and passing to a finite cover $c:S'\rightarrow S$, we choose a rational function $x\in K(S')$ whose zeros $\xi_1,\ld,\xi_l$ are simple and lie above $s_0$. We also fix a basis of sections of the relative de Rham cohomology $H^1_{dR}(\CX/S)$ on a neighborhood of $s_0$, adapted to the Hodge filtration and symplectic with respect to the polarization pairing. Write $f':\CX'\rightarrow S'$ for the pullback family, and fix one of the centers $\xi:=\xi_1$.
%CLEAN

For $v\in\Sigma_K$ and $r>0$, let $\Delta_v(0,r)\subset\C_v$ denote the open disc of radius $r$ centered at $0$, understood in the complex analytic or rigid analytic sense according as $v$ is archimedean or finite. The following is established in \Cref{section:gfunctions}.

\begin{theorem}\label{introperiods}
	With the notation above, there exists a matrix
	$$
	Y_\xi(x)\in M_{2g}(\bar{\Q}[[x]])
	$$
	satisfying the following properties:
	\begin{enumerate}
		\item The entries of $Y_\xi(x)$ are $G$-functions.
		\item Let $v\in\Sigma_K$ be either an archimedean place or a finite place of good reduction of $\CX'_\xi$. Set
		$$
		r_v:=\min\{1,R_v(Y_\xi)\}.
		$$
		Then, for every $s$ in the connected component $\Delta_{\xi,r_v}$ of $x^{-1}(\Delta_v(0,r_v))$ containing $\xi$, we have
		$$
		P_v(s)=\iota_v(Y_\xi(x(s)))\Pi_v(\CX'_\xi),
		$$
		where $P_v(s)$ and $\Pi_v(\CX'_\xi)$ denote the corresponding period matrices.
	\end{enumerate}
\end{theorem}
%CLEAN

The construction of $Y_{\xi}(x)$ and the connection of its values with archimedean periods follow the work of Andr\'e \cite{andre1989g}. The connection between non-archimedean values of $G$-functions and crystalline periods was studied by Andr\'e in \cite{andremots}, using the comparison theory of Berthelot--Ogus \cite{bertogus}. At finite places of good reduction, the proof of the preceding theorem establishes the corresponding relative period identity using Ogus's theory of convergent $F$-isocrystals \cite{ogus}, including at places of $K$ ramified over $\Q$, which are not covered in the form required here by the comparison results of Berthelot--Ogus cited above; see \Cref{remarkramification}.
%CLEAN

At finite places, there are two notions of proximity to the distinguished center: analytic proximity, defined using the radii of convergence of the associated $G$-functions, and reduction-theoretic proximity, given by equality of reductions of the corresponding points on the integral model. The latter ensures that the fibers have the same reduction, as required for interpreting the values of $G$-functions in terms of crystalline periods. In \Cref{section:vadicproximity}, we modify the aforementioned basis of sections of relative de Rham cohomology so that analytic proximity implies reduction-theoretic proximity. This modification preserves symplecticity of the basis and, unlike the constructions in \cite{daworr4,daworr5,daworrpap}, does not require adjoining auxiliary $G$-functions.
%CLEAN

In what follows, $v$-adic proximity refers to the analytic notion described above.

The construction at the centers of the auxiliary cover gives a family $\mathcal Y$ of $G$-functions, as described in \Cref{section:settinggfuns}. The symplecticity of the chosen bases gives polynomial functional relations among their entries. We prove in \Cref{section:trivial} that, when the image of $S$ in $\mathcal A_g$ is Hodge generic, these are the only such relations:
$$
\mathcal Y^{\mathrm{Zar}}=(\SP_{2g}^{\Lambda})_{\bar{\Q}(x)},
$$
where $\Lambda$ indexes the generic isogeny classes arising in the construction. The proof uses the geometric Andr\'e--Grothendieck period theorem of Ayoub \cite{ayoub}.
%CLEAN	
	\subsection{Splittings in $\mathcal A_2$}

We now specialize to abelian surfaces. As in \Cref{section:relsa2}, we work on the auxiliary cover, suppressing primes from the notation and writing $(f:\CX\rightarrow S,s_0)$ for $(f':\CX'\rightarrow S',\xi)$. We say that a fiber $\CX_s$ splits if it is isogenous over $\bar{\Q}$ to a product of elliptic curves
$$
\CX_s\sim E_s\times_{\bar{\Q}}E'_s.
$$

Of particular interest to the Zilber--Pink conjecture are $E^2$-points, for which $E_s$ and $E'_s$ are isogenous non-CM elliptic curves, and $E\times CM$-points, for which exactly one of the two factors has complex multiplication. The corresponding points in $\mathcal A_2$ lie on special curves of the respective types; see \cite{daworr,daworr2}.
%CLEAN

\begin{theorem}\label{naivelocalrelations}
	Let $f:\CX\rightarrow S$ be a one--parameter family of principally polarized abelian surfaces over a number field $K$. Assume that the induced morphism $i:S\rightarrow \mathcal{A}_2$ has image a Hodge-generic curve in $\mathcal A_2$. Assume that $s_0\in S(K)$ is either an $E^2$-point or an $E\times CM$-point, and let $Y_{s_0}(x)\in M_4(\bar{\Q}[[x]])$ be the associated matrix of $G$-functions.
	
	There exists an absolute constant $c>0$ such that for every $E^2$-point or $E\times CM$-point $s\in S(\bar{\Q})$, there is a finite extension $L_s/K(s)$, with $[L_s:K(s)]\leq c$, satisfying the following property.
	
	Let $v\in\Sigma_{L_s}$ be a place for which $s$ is $v$-adically close to $s_0$. Assume that $v$ is either archimedean or a finite place of good reduction of $\CX_{s_0}$. Then there exists a homogeneous polynomial
	$$
	R_{s,v}\in L_s[X_{i,j}:1\leq i,j\leq 4]
	$$
	of degree at most $4$ such that
	$$
	\iota_v(R_{s,v}(Y_{s_0}(x(s))))=0,\qquad R_{s,v}\notin I(\SP_4).
	$$
	Moreover, at the finite places under consideration where $\CX_{s_0}$ does not have supersingular reduction, the polynomials $R_{s,v}$ may be chosen independently of $v$.
\end{theorem}
%CLEAN

Let $s\in S(\bar{\Q})$ be a point whose fiber splits and assume that $s$ is $v$-adically close to $s_0$. Fix isogenies
$$
\theta_0:\CX_{s_0}\rightarrow E_0\times E'_0,\qquad \theta_s:\CX_s\rightarrow E_s\times E'_s.
$$
Using the period interpretation of the associated $G$-functions and the functoriality of the comparison isomorphisms under isogenies, we obtain
\begin{equation}\label{eq:relationgenerator}
	\iota_v(F_s(x(s)))=\begin{pmatrix}\Pi_v(E_s)&0\\0&\Pi_v(E'_s)\end{pmatrix}\Theta_v\begin{pmatrix}\Pi_v(E_0)^{-1}&0\\0&\Pi_v(E'_0)^{-1}\end{pmatrix}.
\end{equation}
Here $\Pi_v(E)$ denotes the period matrix of an elliptic curve $E$, while the entries of $F_s(x)$ are algebraic linear combinations of those of $Y_{s_0}(x)$, with coefficients independent of $v$. The matrix $\Theta_v\in\GL_4(\C_v)$ arises from the corresponding cohomological realizations of the isogenies.
%CLEAN

The right-hand side of \eqref{eq:relationgenerator} involves the potentially transcendental periods of elliptic curves and the entries of $\Theta_v$, which at finite places need not be algebraic. To obtain polynomial relations with algebraic coefficients among the entries of $Y_{s_0}(x(s))$, we eliminate these quantities using the additional cohomological structures available at the various places. The constructions differ according to the place $v$.
%CLEAN

At archimedean places, the matrix $\Theta_v$ in \eqref{eq:relationgenerator} has rational entries, since it arises from morphisms on Betti cohomology. Taking determinants of suitable blocks of this identity and using the Legendre relation for the periods of elliptic curves yields polynomial relations among the values of the associated $G$-functions. These relations have uniformly bounded degree, although the corresponding polynomials may depend on $v$.
%CLEAN

At a finite place $v$ of good reduction of the distinguished fiber, the $v$-adic proximity of $s$ to $s_0$ implies that $\CX_s$ and $\CX_{s_0}$ have the same reduction. Reducing the isogenies $\theta_0$ and $\theta_s$ modulo $v$, we obtain a morphism
$$
\widetilde{\theta}_s\circ\widetilde{\theta}_0^\vee:\widetilde E_{0,v}\times\widetilde E'_{0,v}\longrightarrow\widetilde E_{s,v}\times\widetilde E'_{s,v}.
$$
The matrix $\Theta_v$ in \eqref{eq:relationgenerator} represents the pullback induced by this morphism on crystalline cohomology. With respect to bases compatible with the product decompositions, $\Theta_v$ has four $2\times 2$ blocks, each arising from a morphism between the corresponding reduced elliptic curves and hence either zero or invertible.
%CLEAN

If the reductions $\widetilde E_{0,v}$ and $\widetilde E'_{0,v}$ are not geometrically isogenous, then $\Theta_v$ is either block diagonal or block antidiagonal. This gives a quadratic polynomial relation among the entries of $Y_{s_0}(x(s))$, independent of $v$.
%CLEAN

At the remaining non-supersingular finite places, the elliptic curves $E_0$ and $E_0'$ above have geometrically isogenous ordinary reductions. We choose bases of crystalline cohomology adapted to the slopes of Frobenius, as in \cite[\S 2.1.1]{papaspadicgfuns2}. Since morphisms between the reduced elliptic curves preserve the Frobenius slope decomposition, the corresponding $2\times 2$ blocks of $\Theta_v$ are diagonal. Using \eqref{eq:relationgenerator}, we obtain polynomial relations among the values of the $G$-functions that can be chosen independently of $v$.
%CLEAN

At supersingular places, Frobenius no longer has distinct slopes, so the ordinary construction does not apply. Here the relations arise from endomorphisms of the reduced elliptic curves. When CM factors occur, we use the $p$-adic Betti lattices introduced by Andr\'e \cite{andremots}; when both $s_0$ and $s$ are $E^2$-points, we use traces and determinants of endomorphisms. The resulting polynomials have uniformly bounded degree, but may depend on $v$.
%CLEAN

It remains to establish that the resulting polynomial relations are non-trivial. Using the description of functional relations in \Cref{section:trivial}, we show that the corresponding polynomials do not belong to $I(\SP_4)$. In several cases, this requires explicit coefficient computations performed using the accompanying Mathematica code; see \cite{mathematicacode}.
%CLEAN

Prior to the developments discussed here, the only known constructions of algebraic relations among $p$-adic values of $G$-functions centered at smooth fibers were those of Andr\'e \cite{andremots} and Beukers \cite{beukers}, for CM points in families of elliptic curves. The elliptic CM case is revisited and extended in \cite{papaspadicgfuns2}, while related constructions appear in \cite{daworrpap} for points corresponding to isogenous triples in $Y(1)^3$ and in \cite{papaspadicgfuns3} for quaternionic multiplication in $\mathcal A_2$. These three works were developed contemporaneously with the present paper.
%CLEAN

\subsection{Height bounds and unlikely intersections}

To apply the Andr\'e--Bombieri Hasse principle, we combine the preceding local relations into a polynomial relation among the values of $\mathcal Y$ that holds at every place of proximity. At non-supersingular finite places, only a bounded number of factors is needed, since the corresponding polynomials are independent of the place. The contribution from archimedean places is controlled by the degree of the point, while that from supersingular places depends on the number of such places of proximity. The resulting global relation is non-trivial by \Cref{section:trivial} and the primeness of the ideal of functional relations of $\mathcal Y$.
%CLEAN

For $s\in S(\bar{\Q})$, let $\Sigma(s,0)$ denote the set of places $v\in\Sigma_{K(s)}$ satisfying
$$
|x(s)|_v<\min\{1,R_v(\mathcal Y)\},
$$
as in \Cref{section:vadicproximity}. We set
$$
\Sigma_{K,\ssing}(s,0):=\left\{w\in\Sigma_{K,f}:
\begin{aligned}
	&\CX_{s_0}\text{ has good supersingular reduction at }w,\\
	&\exists v\in\Sigma(s,0),\ v|w
\end{aligned}
\right\}.
$$
The preceding construction gives the following height bound.
%CLEAN

\begin{theorem}\label{goodreductionmainhtbound}
	Let $f:\CX\rightarrow S$ be a one--parameter family of principally polarized abelian surfaces over a smooth irreducible curve $S$, defined over a number field $K$. Assume that the image of the induced morphism $i:S\rightarrow \mathcal A_2$ is a Hodge-generic curve. 
	
	Suppose that $s_0\in S(K)$ is an $E^2$- or $E\times CM$-point whose corresponding abelian surface has everywhere potentially good reduction. Let $h$ be a Weil height on $S$. Then there exist constants $c_1,c_2>0$ such that, for every $E^2$- or $E\times CM$-point $s\in S(\bar{\Q})$, we have
	$$
	h(s)\leq c_1\Bigl([K(s):\Q]\bigl(1+|\Sigma_{K,\ssing}(s,0)|\bigr)\Bigr)^{c_2}.
	$$
\end{theorem}
%CLEAN

\subsubsection{Applications to Zilber--Pink}

As recalled in \cite{daworr}, for a Hodge-generic curve $Z\subset\mathcal A_2$, writing $\Atyp(Z)$ for the non-special part of its atypical locus, we have
$$
\Atyp(Z)=\Atyp(Z)_{E\times CM}\sqcup\Atyp(Z)_{E^2}\sqcup\Atyp(Z)_{\mathrm{QM}},
$$
where $\Atyp(Z)_*$ denotes the set of points of $Z$ that are Hodge generic in a special curve of type $*$. To study any one of these loci, we may assume that it is nonempty, since otherwise its finiteness is immediate. We may thus fix a point $s_0$ in the corresponding locus and a number field $K$ over which $Z$ is defined such that $s_0\in Z(K)$.
%CLEAN

For points with $|\Sigma_{K,\ssing}(s,0)|$ bounded, \Cref{goodreductionmainhtbound}, together with the Masser--W\"ustholz isogeny estimates, gives lower bounds for the size of Galois orbits. The results of Daw and Orr \cite{daworr,daworr2} then yield the following finiteness statement.
%CLEAN

\begin{cor}\label{goodreductionmainzpapp}
	Let $Z\subset\mathcal A_2$ be a smooth irreducible Hodge-generic curve defined over $\bar{\Q}$. Assume that $Z$ contains an $E^2$- or $E\times CM$-point $s_0$ whose corresponding abelian surface has everywhere potentially good reduction.
	
	Let $K$ be a number field of definition of $Z$ such that $s_0\in Z(K)$ and over which the auxiliary cover in \Cref{section:applicationsgood} associated to the pair $(Z,s_0)$ is defined.
	
	Then, for every $N\in\N$, the set
	\begin{equation}
		\Sha_{ZP\text{-}\spli,N}(Z):=
		\left\{s\in Z(\C):
		\begin{aligned}
			&s\text{ is an }E^2\text{- or }E\times CM\text{-point},\\
			&|\Sigma_{K,\ssing}(s,0)|\leq N
		\end{aligned}
		\right\}
	\end{equation}
	is finite.
\end{cor}
%CLEAN

The everywhere potentially good reduction assumption on the distinguished fiber may be viewed as an integrality condition. Indeed, if $\CX_{s_0}\sim E_0\times E'_0$, this assumption is equivalent to the integrality of the $j$-invariants of $E_0$ and $E'_0$; see \cite[Chapter VII, Proposition 5.5]{silvermanell}.
%CLEAN

The dependence of the preceding results on $|\Sigma_{K,\ssing}(s,0)|$ raises the question of bounding this quantity in terms of the arithmetic complexity of $s$. We return to this in \Cref{section:supersingularprox}, where we formulate a conjectural bound for points of the same splitting type as $s_0$.
%CLEAN

\subsection{Further directions}

\subsubsection{Bad semistable reduction}

At places of bad semistable reduction of the distinguished fiber, we seek an analogue of the crystalline period interpretation of the values of the associated $G$-functions. Hyodo--Kato cohomology provides the natural candidate for such an interpretation.

In \Cref{section:badred}, we formulate two conjectures. The first concerns the functoriality of Hyodo--Kato cohomology under isogenies, while the second asserts the existence of a relative Hyodo--Kato comparison isomorphism. 

Assuming these conjectures, we exploit the $(\phi,N)$-module structure on Hyodo--Kato cohomology to construct algebraic relations analogous to those of \Cref{naivelocalrelations} among the local values of the associated $G$-functions. These relations yield the following extension of \Cref{goodreductionmainhtbound}.
%CLEAN

\begin{theorem}\label{badredmainhtbound}
	Assume that \Cref{conjhyodokato} and \Cref{conjrelhyodokato} hold. Then the conclusion of \Cref{goodreductionmainhtbound} holds without the assumption that the distinguished fiber $\CX_{s_0}$ has everywhere potentially good reduction.
\end{theorem}

The corresponding Galois orbit bounds and Zilber--Pink finiteness result also hold without the everywhere potentially good reduction hypothesis; see \Cref{badredmainzpapp}.
%CLEAN

\subsubsection{Splittings in higher dimensions}
The constructions of \Cref{section:relsa2} suggest analogous relations among the values of $G$-functions for splittings in $\mathcal A_g$. At finite places, invariants of the reductions may control the dependence of these relations on $v$, while establishing their non-triviality presents a separate difficulty. We discuss these questions in \Cref{section:splittingsinag}.
%CLEAN

\subsection{Notation}
Let $K$ be a number field and $v\in\Sigma_K$. For an abelian variety $X$ over $K$, we write $X_v:=X\times_K K_v$ for its base change to $K_v$. If $v$ is a finite place of good reduction for $X$, we denote its reduction modulo $v$ by $\widetilde{X}_v$.

For each $v\in\Sigma_K$, fix an embedding $\iota_v:K\hookrightarrow\C_v$, where $\C_v$ denotes $\C$ if $v$ is archimedean and $\C_p$ if $v$ lies above a prime $p$. For a power series $y(x)=\Sum{n=0}{\infty}a_nx^n\in K[[x]]$, we write
$$
\iota_v(y(x)):=\Sum{n=0}{\infty}\iota_v(a_n)x^n\in\C_v[[x]].
$$
Given a family $\mathcal{Y}:=(y_1,\ld,y_N)\in K[[x]]^N$, we write $R_v(y_j)$ for the radius of convergence of $\iota_v(y_j)$ and set
$$
R_v(\mathcal{Y}):=\min_{1\leq j\leq N}R_v(y_j).
$$

	%\part{General Background}
		\section{Period matrices and splittings}\label{section:background}

We record the relations among period matrices arising from splittings of abelian varieties of the form
$$
X=Y\times_K Y',
$$
where $Y$ and $Y'$ have smaller dimension than $X$. With respect to bases adapted to this decomposition, the corresponding period matrices are block diagonal up to a fixed permutation of the basis.

The relevant period matrices are expressed through the archimedean and $p$-adic comparison isomorphisms, which we recall first. %CLEAN

\subsection{Period matrices}\label{section:periods}

For the remainder of this subsection, let $X$ be a $g$-dimensional principally polarized abelian variety over a number field $K$, with everywhere semistable reduction. For $v\in\Sigma_{K,\infty}$, or for $v\in\Sigma_{K,f}$ at which $X$ has good reduction, define
\begin{equation}
	H^1_v(X):=\begin{cases}
		H^1_{\crys}(\widetilde{X}_v/W(k_v))\otimes W(k_v)[\frac{1}{p}] & v\in\Sigma_{K,f}\text{ of good reduction for }X,\\
		
		H^1_B(X_v^{an},\Q)& v\in\Sigma_{K,\infty},
	\end{cases}
\end{equation}
where $H^1_B(X_v^{an},\Q)$ denotes the Betti cohomology of $X_v^{an}$. %CLEAN

For $v\in\Sigma_{K,f}$ of good reduction for $X$, the comparison isomorphism of Berthelot--Ogus gives
\begin{equation}\label{eq:abeliancompiso}
	\rho_v(X):H^1_{dR}(X/K)\otimes_K K_v\longrightarrow H^1_v(X)\otimes_{K_{v,0}}K_v,
\end{equation}
where $K_{v,0}$ denotes the fraction field of $W(k_v)$.

For $v\in\Sigma_{K,\infty}$, the classical de Rham--Betti comparison gives
\begin{equation}\label{eq:abeliancanisoarch}
	\rho_v(X):H^1_{dR}(X/K)\otimes_K\C\longrightarrow H^1_v(X)\otimes_{\Q}\C.
\end{equation} %CLEAN

Assume now that $X=Y\times_K Y'$, where $\dim Y=h$ and $\dim Y'=h'$. Suppose that $Y$ and $Y'$ are principally polarized, and equip $X$ with the induced product polarization. For every place $v$ as above, functoriality of the comparison isomorphisms and the decompositions
$H^1_{dR}(X/K)=H^1_{dR}(Y/K)\oplus H^1_{dR}(Y'/K)$ and $H^1_v(X)=H^1_v(Y)\oplus H^1_v(Y')$
give
$$
\rho_v(X)=\rho_v(Y)\oplus\rho_v(Y').
$$ %CLEAN

We use the following class of bases for de Rham cohomology.
\begin{definition}\label{hodgebasis}
	Let $X$ be a $g$-dimensional principally polarized abelian variety over a number field $K$, and let $\langle\cdot,\cdot\rangle$ denote the alternating pairing on $H^1_{dR}(X/K)$ induced by the polarization. An ordered basis
	$$
	\Gamma_{dR}(X):=(\omega_1,\ldots,\omega_g,\eta_1,\ldots,\eta_g)
	$$
	of $H^1_{dR}(X/K)$ is called a \textbf{Hodge basis} if
	\begin{enumerate}
		\item $\omega_1,\ldots,\omega_g$ form a basis of
		$$
		F^1_X:=e^{*}\Omega_{X/K}\subset H^1_{dR}(X/K),
		$$
		the first step of the Hodge filtration; and
		
		\item $\Gamma_{dR}(X)$ is symplectic, that is,
		$$
		\langle\omega_i,\eta_j\rangle=\delta_{i,j},\qquad\langle\omega_i,\omega_j\rangle=\langle\eta_i,\eta_j\rangle=0.
		$$
	\end{enumerate}
\end{definition} 
%CLEAN

For a Hodge basis $\Gamma_{dR}(X)$, we write
$$
\underline{\omega}:=(\omega_1,\ldots,\omega_g),\qquad \underline{\eta}:=(\eta_1,\ldots,\eta_g),
$$
so that $\Gamma_{dR}(X)=(\underline{\omega},\underline{\eta})$. %CLEAN

\begin{remark}[Bases for products of abelian varieties]\label{remarkbases}	
	
	Assume that $X=Y\times_K Y'$, with $\dim Y=h$ and $\dim Y'=h'$. Given Hodge bases $\Gamma_{dR}(Y)=(\underline{\omega},\underline{\eta})$ and $\Gamma_{dR}(Y')=(\underline{\omega}',\underline{\eta}')$, we write
	$$
	\Gamma_{dR}(X)=(\underline{\omega},\underline{\omega}',\underline{\eta},\underline{\eta}')
	$$
	for the induced Hodge basis. %CLEAN
	
	Similarly, given symplectic bases $\Gamma_v(Y)=(\underline{\gamma}_v,\underline{\delta}_v)$ and $\Gamma_v(Y')=(\underline{\gamma}'_v,\underline{\delta}'_v)$ of $H^1_v(Y)$ and $H^1_v(Y')$, respectively, we write
	$$
	\Gamma_v(X)=(\underline{\gamma}_v,\underline{\gamma}'_v,\underline{\delta}_v,\underline{\delta}'_v)
	$$
	for the induced symplectic basis of $H^1_v(X)$. %CLEAN
\end{remark}

\begin{definition}\label{defperiodmatrix}
	1. Let $Z$ be a principally polarized abelian variety over $K$, let $v\in\Sigma_K$ be either archimedean or a finite place of good reduction for $Z$, and fix a Hodge basis $\Gamma_{dR}(Z)$ of $H^1_{dR}(Z/K)$ and a basis $\beta_v$ of $H^1_v(Z)$. We define the $v$-period matrix associated to these bases by
	\begin{equation}
		\Pi_v(\Gamma_{dR}(Z),\beta_v):=
		\begin{cases}
			[\rho_v(Z)]_{\Gamma_{dR}(Z)}^{\beta_v} & v\in\Sigma_{K,f},\\[2mm]
			\frac{1}{2\pi i}[\rho_v(Z)]_{\Gamma_{dR}(Z)}^{\beta_v} & v\in\Sigma_{K,\infty}.
		\end{cases}
	\end{equation}
	
	2. In the setting of \Cref{remarkbases}, for $*\in\{X,Y,Y'\}$ we write
	$$
	\Pi_v(*):=\Pi_v(\Gamma_{dR}(*),\Gamma_v(*)).
	$$
\end{definition} %CLEAN
\begin{remark}
	The factor $\frac{1}{2\pi i}$ in the archimedean definition agrees with the normalization of the period matrix used in \cite[Ch.~X]{andre1989g}, in which the period matrices satisfy the Riemann relations in the form used below. %CLEAN
\end{remark}
		\subsection{Period relations from splittings}\label{section:periodsurfaces}

The product decomposition $X=Y\times_K Y'$ is reflected in the corresponding period matrices as follows.

\begin{lemma}\label{lemmaperiodsplittings}
	Let $X=Y\times_K Y'$ be as above, with $\dim Y=h$ and $\dim Y'=h'$, and let $v\in\Sigma_K$ be either archimedean or a finite place of good reduction for $X$. With respect to the bases of \Cref{remarkbases},
	\begin{equation}\label{eq:e2periods}
		\Pi_v(X)=J_{h,g}\begin{pmatrix}\Pi_v(Y)&0\\0&\Pi_v(Y')\end{pmatrix}J_{h,g}^{-1},
	\end{equation}
	where $J_{h,g}$ is the permutation matrix sending the ordered basis $(\underline{\omega},\underline{\eta},\underline{\omega}',\underline{\eta}')$ to $(\underline{\omega},\underline{\omega}',\underline{\eta},\underline{\eta}')$.
\end{lemma} %CLEAN
\begin{remark}
	When $g=2h$, the permutation matrix $J_{h,2h}$ is an involution, and hence
	$$
	J_{h,2h}^{-1}=J_{h,2h}.
	$$ %CLEAN
\end{remark}

The following result is a higher-dimensional generalization of \cite[Prop.~4.4, Eq.~(13)]{daworr4}, formulated in the notation of \cite[Lem.~3.1]{papaszpy1}. %CLEAN

\begin{lemma}\label{isogenyperiods}
	Let $\theta:X\rightarrow X'$ be an isogeny between two $g$-dimensional principally polarized abelian varieties over a number field $K$, with $X'=Y\times_K Y'$ as above. Let $v\in\Sigma_K$ be either archimedean or a finite place of good reduction for $X$, and therefore also for $X'$, $Y$, and $Y'$. %CLEAN
	
	Fix a Hodge basis $\Gamma_{dR}(X)$ of $H^1_{dR}(X/K)$ and a basis $\beta_v$ of $H^1_v(X)$. Then
	\begin{equation}\label{eq:isogenyperiods}
		[\theta]_{dR}\Pi_v(\Gamma_{dR}(X),\beta_v)=\Pi_v(X')[\theta]_v,
	\end{equation}
	where $[\theta]_{dR}$ is the matrix of $\theta^*$ with respect to $\Gamma_{dR}(X')$ and $\Gamma_{dR}(X)$, while $[\theta]_v$ is the matrix of the induced pullback map on $H^1_v$ with respect to $\Gamma_v(X')$ and $\beta_v$. %CLEAN
	
	Moreover, there exist $A,C\in \GL_g(K)$ and $B\in M_g(K)$ such that
	$$
	[\theta]_{dR}=\begin{pmatrix}A&0\\B&C\end{pmatrix}.
	$$ %CLEAN
\end{lemma}
\begin{remark}
	Since the de Rham bases are fixed over $K$, the entries of $[\theta]_{dR}$ lie in $K$ and are independent of the place $v$. %CLEAN
\end{remark}

\subsubsection{Bases induced by an isogeny}\label{section:isogenyhodgebasis}

Assume now that $X$ and $X'=Y\times_K Y'$ are abelian surfaces, with $Y$ and $Y'$ elliptic curves, and let
$$
\theta:X\rightarrow X'
$$
be an isogeny. We describe a Hodge basis of $H^1_{dR}(X/K)$ adapted to $\theta$ and to the product decomposition of $X'$. %CLEAN

Let $(\omega,\eta)$ and $(\omega',\eta')$ be Hodge bases of $H^1_{dR}(Y/K)$ and $H^1_{dR}(Y'/K)$, respectively. Since $\theta^*F^1_{X'}=F^1_X$, we set
$$
\underline{\omega}:=(\omega_1,\omega_2):=(\theta^*\omega,\theta^*\omega')
$$
and extend $\underline{\omega}$ to a Hodge basis $(\underline{\omega},\underline{\eta})$ of $H^1_{dR}(X/K)$. We call any such extension $\Gamma_{dR}(X,\theta)$ a \textbf{Hodge basis induced by $\theta$}. %CLEAN

The construction above gives the following.
\begin{lemma}\label{lemmaisogenyhodgebasis}
	Let $\theta:X\rightarrow X'=Y\times_K Y'$ be an isogeny of abelian surfaces. With respect to the bases $\Gamma_{dR}(X')$ and $\Gamma_{dR}(X,\theta)$, the matrix of $\theta^*:H^1_{dR}(X'/K)\longrightarrow H^1_{dR}(X/K)$ has the form
	$$
	[\theta]_{dR}=\begin{pmatrix}I_2&0\\B&C\end{pmatrix},
	$$
	where $B\in M_2(K)$ and $C\in\GL_2(K)$.
\end{lemma} %CLEAN
		\section{G-functions centered at a smooth point}\label{section:backgroundgfuns}

In this section we associate to a one-parameter family of principally polarized abelian varieties $f:\CX\rightarrow S$ defined over a number field $K$, together with a distinguished point $s_0\in S(K)$, a family of G-functions centered at $s_0$. Set $\CX_0:=\CX_{s_0}$. 

After passing to a suitable finite cover $S'\rightarrow S$, the construction is organized around a Hodge basis of sections chosen on the original family and a rational function $x$ on $S'$ whose zeros lie above $s_0$.

The first feature of the construction is a cohomological interpretation of the values of these G-functions. Following Andr\'e, at archimedean places and at finite places of good reduction of the central fiber, their values describe the relative period matrices through the de Rham--Betti and de Rham--crystalline comparison isomorphisms. At finite places of $K$ ramified over $\Q$, we establish the de Rham--crystalline interpretation using the theory of convergent $F$-isocrystals developed by Ogus. %CLEAN

The second feature of the construction is that analytic $v$-adic proximity implies the reduction-theoretic proximity required for the period interpretation. To achieve this, we modify the choice of Hodge basis so that, for $v$ a finite place, points lying within the $v$-adic radius of convergence of the resulting G-functions have fibers with the same reduction as the central fiber. %CLEAN

The third feature of the construction is an explicit description of the functional relations already satisfied by the associated G-functions. After assembling the matrices arising from the different centers into a single family $\mathcal{Y}$, we determine its Zariski closure using the geometric Andr\'e--Grothendieck period theorem \cite{ayoub,bakkertsimermanandregroth} and the monodromy of the underlying family. This gives an explicit description of the trivial relations against which the relations constructed in the following sections will be compared. %CLEAN

\subsection{Relative comparison isomorphisms}

We use the relative de Rham--Betti and de Rham--crystalline comparison isomorphisms for the family $f$. The discussion follows \cite[\S 3]{andremots}, where Andr\'e relates the work of Berthelot--Ogus \cite{bertogus} to values of G-functions. %CLEAN

Associated with $f$ is the differential module $(H^1_{dR}(\CX/S),\nabla)$, where $\nabla$ denotes the Gauss--Manin connection. Let $v\in\Sigma_{K,\infty}$ be an archimedean place and let $f_v^{an}:\CX_v^{an}\rightarrow S_v^{an}$ be the corresponding analytification. Grothendieck's relative de Rham--Betti comparison gives an isomorphism $\CP_v:H^1_{dR}(\CX/S)\otimes\CO_{S_v^{an}}\rightarrow R^1f_{v,*}^{an}\Q\otimes\CO_{S_v^{an}}$. %CLEAN

Upon restriction of $\CP_v$ to a sufficiently small archimedean disc $\Delta_v\subset S_v^{an}$ centered at $s_0$, a choice of basis of sections of $H^1_{dR}(\CX/S)|_{\Delta_v}$ and a trivializing frame of the Betti local system $R^1f_{v,*}^{an}\Q|_{\Delta_v}$ give a matrix representation of $\CP_v$; see \cite[Ch.~IX]{andre1989g}. %CLEAN

At finite places an analogous comparison is provided by the work of Berthelot--Ogus \cite{bertogus}. Fix $v\in\Sigma_{K,f}$ and assume that $\CX_0$ has good reduction at $v$. Let $f_v^{an}:\CX_v^{an}\rightarrow S_v^{an}$ be the rigid analytification of $f$. After shrinking a rigid analytic open disc $\Delta_v\subset S_v^{an}$ centered at $s_0$, the fibers $\CX_s$, for $s\in S(K)\cap\Delta_v$, have the same reduction as $\CX_0$. We denote this common reduction by $\widetilde{\CX}_{0,v}$ and write $k=k(v)$ for the residue field of $K$ at $v$. %CLEAN

Assume moreover that $v$ is unramified in $K/\Q$. The results of \cite{bertogus} identify the horizontal sections $(H^1_{dR}(\CX/S)\otimes_{\CO_S}\CO_{\Delta_v})^{\nabla}$ with $H^1_{\crys}(\widetilde{\CX}_{0,v}/W(k))\otimes K_{v,0}$. In particular, one obtains a comparison isomorphism
$$
\CP_v:H^1_{dR}(\CX/S)\otimes_{\CO_S}\CO_{\Delta_v}\rightarrow\big(H^1_{\crys}(\widetilde{\CX}_{0,v}/W(k))\otimes K_{v,0}\big)\otimes\CO_{\Delta_v}.
$$ %CLEAN

If $v\in\Sigma_{K,f}$ is ramified in $K/\Q$, an analogous comparison is provided by the work of Ogus \cite{ogus}. We recall the form needed in the proof of \Cref{gfuns}. In what follows, we write $\CP_v(s)$ for the fiber at $s\in S(\bar{\Q})$ of the corresponding comparison isomorphism. %CLEAN
		\subsection{Auxiliary covers and models}\label{section:htboundsred}
% REVISIT: this subsection has a dual role: it prepares the geometric setting required by the G-function construction and records that these reductions preserve the height-bound problem. On the second pass, make both roles explicit and reconsider the title only after the reductions have been fully organized.

The height bounds obtained later are deduced from the relations among G-function values via the Andr\'e--Bombieri Hasse principle; see \cite[Ch.~VII, \S5]{andre1989g}. Its application requires several auxiliary geometric properties of the family $f:\CX\rightarrow S$. We first make reductions ensuring these properties while preserving the height problem under consideration.
% CLEAN

We say that a fiber $\CX_s$ has ``unlikely many'' endomorphisms if the corresponding point of the Hodge generic curve in $\mathcal{A}_g$ lies on a proper special subvariety parametrizing abelian varieties with these additional endomorphisms, so that this incidence gives an unlikely intersection. %CLEAN

\begin{conj}\label{conjhtbound}
	Let $S$ be a smooth irreducible curve defined over $\bar{\Q}$ and let $f:\CX\rightarrow S$ be an abelian scheme as above. Assume that the image of the induced morphism $S\rightarrow \mathcal{A}_g$ is not contained in a proper special subvariety.
	
	Then there exist positive constants $c_1$ and $c_2$, depending on $f$, such that for every $s\in S(\bar{\Q})$ such that the fiber $\CX_s$ has ``unlikely many'' endomorphisms,
	\begin{equation}
		h(s)\leq c_1\cdot([\Q(s):\Q])^{c_2},
	\end{equation}
	where $h$ is some Weil height on $S$.
\end{conj}
% REVISIT: the eventual placement of the conjecture should be reconsidered.

We now carry out the aforementioned reductions, placing the family in a setting with a suitable local parameter and compatible integral models. %CLEAN

\begin{lemma}\label{reduction}
	It suffices to establish \Cref{conjhtbound} under the following additional assumptions. After choosing a number field $K$ over which $S$ and $f:\CX\rightarrow S$ are defined, and replacing it by a finite extension if necessary, there exist a regular $\CO_K$-model $\mathfrak{S}$ of $S$, a semi-abelian scheme $\mathfrak{X}\rightarrow\mathfrak{S}$, and a rational function $x\in K(S)$ satisfying the following properties:
	\begin{enumerate}
		\item $\mathfrak{X}_K\simeq\CX$ as abelian schemes over $S$;
		\item the points $\{\xi_1\ld\xi_l\}:=\{s\in S(\bar{\Q}):x(s)=0\}$ are simple zeros of $x$;
		\item the induced morphism $x:S\rightarrow\mathbb{P}^1$ is Galois; equivalently, the group $\aut_x(S):=\{\sigma\in\aut(S):x=x\circ\sigma\}$ acts transitively on the $\bar{\Q}$-fibers of $x$;
		\item the fibers $\CX_{\xi_1},\ldots,\CX_{\xi_l}$ are mutually $K$-isomorphic;
		\item\label{l1it1} for every $\xi\in\{\xi_1\ld\xi_l\}$, the fiber $\CX_{\xi}$ has everywhere semistable reduction;
		\item\label{integralx} the morphism $x$ extends to a morphism $\mathfrak{S}\rightarrow\mathbb{P}^1_{\CO_K}$, and there exists a regular proper $\CO_K$-model $\mathfrak{S}'$ of $S$ such that every element of $\aut_x(S)$ extends to a morphism $\mathfrak{S}'\rightarrow\mathfrak{S}$.
	\end{enumerate}
\end{lemma} %CLEAN
\begin{proof}
	The proof is the same as that of \cite[Lemma~2.16]{papaseffbrsieg}, with condition \ref{integralx} supplied by \cite[Lemma~6.2]{daworr5}. %CLEAN
\end{proof}

\begin{remarks}\label{remarkonmodels}
1. In the applications below, the points $\xi_j$ will themselves be points of interest. For instance, in the settings of \Cref{goodreductionmainzpapp} or \Cref{badredmainzpapp}, the fibers $\CX_{\xi_j}$ are abelian surfaces admitting an isogeny decomposition $\CX_{\xi_j}\sim B\times B'$.\\ %CLEAN
	
	2. Strictly speaking, the regular scheme $\mathfrak{S}$, which is projective over $\CO_K$, is a model of a fixed smooth projective curve $S'$ containing $S$. By abuse of terminology, we refer to $\mathfrak{S}$ as a ``model of $S$ over $\CO_K$''. %CLEAN
\end{remarks}

		\subsection{G-functions}\label{section:gfunctions}

Henceforth we work in the setting of \Cref{reduction}, with $s_0\in\{\xi_1\ld\xi_l\}$ one of the simple zeros of $x$. Set $g:=\dim\CX_0$.
% CLEAN

Let $v\in\Sigma_K$ and let $\iota_v:K\hookrightarrow\C_v$ be the associated embedding. Consider the analytification of $f:\CX\rightarrow S$ with respect to $v$, in the complex or rigid analytic category according as $v$ is archimedean or finite. If $\Delta_r\subset\C_v$ denotes the open disc of radius $r$ centered at $0$, let $\Delta_{s_0,r}$ be the connected component of $x^{-1}(\Delta_r)$ containing $s_0$. By abuse of terminology, we refer to $\Delta_{s_0,r}$ as the $v$-adic disc centered at $s_0$ with radius $r$.
% CLEAN

Let $\Gamma_{dR}(\CX_0):=(\omega_{1,0},\ldots,\omega_{g,0},\eta_{1,0},\ldots,\eta_{g,0})$ be a Hodge basis of $H^1_{dR}(\CX_0/K)$. After possibly removing finitely many points from $S$ and replacing $K$ by a finite extension, we may assume that there exists a global basis of sections
$$
\Gamma_{dR}(\CX):=(\omega_1,\ldots,\omega_g,\eta_1,\ldots,\eta_g)\subset H^1_{dR}(\CX/S)(S)
$$
specializing at $s_0$ to $\Gamma_{dR}(\CX_0)$. We may moreover assume that $\omega_1,\ldots,\omega_g$ form a basis of sections of the first step $\mathcal{F}^1(S)$ of the Hodge filtration on $H^1_{dR}(\CX/S)$. We refer to $\Gamma_{dR}(\CX)$ as a Hodge basis of $\CX$, in the sense of \Cref{hodgebasis}.
% CLEAN

\begin{theorem}\label{gfuns}
	Let $f$, $s_0$, and $x$ be as above, and let $\Gamma_{dR}(\CX)$ be a Hodge basis of $\CX$ specializing at $s_0$ to $\Gamma_{dR}(\CX_0)$. There exists a matrix $Y_G\in M_{2g}(\bar{\Q}[[x]])$ satisfying the following properties:
	\begin{enumerate}
		\item the entries of $Y_G$ are G-functions;
		\item for every $v\in\Sigma_{K,\infty}$, and every $v\in\Sigma_{K,f}$ at which $\CX_0$ has good reduction, writing $r_v:=\min\{1,R_v(Y_G)\}$, we have
		$$
		\CP_v(s)=\iota_v(Y_G(x(s)))\cdot\Pi_v(\CX_0)
		$$
		for every $s\in\Delta_{s_0,r_v}$, where $\Pi_v(\CX_0)$ denotes the period matrix of $\CX_0$ defined in \Cref{defperiodmatrix}.
	\end{enumerate}
\end{theorem}
% CLEAN

		\begin{proof}	
	With respect to the basis of sections $\Gamma_{dR}(\CX)$ of $H^1_{dR}(\CX/S)(U)$, where $U$ is an affine open neighborhood of $s_0$, the Gauss--Manin connection determines a differential system
	\begin{equation}\label{eq:diffsys} 
	\frac{d}{d x} Y=A \cdot Y,
	\end{equation} 
	where $A\in M_{2g}(K(S))$.
	% CLEAN
	
	By \cite[Ch.~III]{andre1989g}, there exists a matrix solution $Y_G\in M_{2g}(\bar{\Q}[[x]])$ of \eqref{eq:diffsys} satisfying $Y_G(0)=I_{2g}$. Following Andr\'e, we refer to $Y_G$ as the normalized uniform solution of the system.
	% CLEAN
	
	That the entries of $Y_G$ are G-functions follows from the work of Andr\'e \cite{andre1989g}. For a concise account, together with the archimedean case of part~(2), see \cite[Theorem~2.7]{papaseffbrsieg}.
	% CLEAN
	
	By \Cref{lemgfunsnew} below, after replacing $\Gamma_{dR}(\CX)$ by a Hodge basis with the same specialization at $s_0$ and $Y_G$ by the corresponding normalized uniform solution, we may arrange that the following property holds; we continue to denote the resulting matrix by $Y_G$.
	
	\begin{assumption}\label{assumptionpadicproximity}
		Let $v\in\Sigma_{K,f}$ be a finite place at which $\CX_0$ has good reduction, and write
		$$
		r_v:=\min\{1,R_v(Y_G)\}.
		$$
		Then the entries of $Y_G$ converge $v$-adically on $\Delta_{s_0,r_v}$, and every $s\in S(\bar{\Q})\cap\Delta_{s_0,r_v}$ has the same reduction modulo $v$ as $s_0$, in the sense of \Cref{section:vadicproximity}.
	\end{assumption}
	% CLEAN
	
	Fix $v\in\Sigma_{K,f}$ at which $\CX_0$ has good reduction, and set $L:=K_v$. Let $V:=\CO_L$ be its ring of integers, let $k:=k(v)$ be its residue field of characteristic $p>0$, and set $W:=W(k)$ and $K_0:=\Frac(W)$. For notational simplicity, for the remainder of the proof we write $f:\CX\rightarrow S$ in place of $f_v:\CX_v\rightarrow S_v$. We also write $\mathfrak{f}:\mathfrak{X}\rightarrow\mathfrak{S}$ for the base change of the corresponding objects in \Cref{reduction} along $\spec(V)\rightarrow\spec(\CO_K)$.
	% CLEAN
	
	Since $\mathfrak{S}$ is proper over $V$, the point $s_0$ extends to a section $\widetilde{s}_0:\spec(V)\rightarrow\mathfrak{S}$. Using the extension of $x$ to the integral model, we may choose a $p$-adic formal disc $\DD\rightarrow\mathfrak{S}$ centered along $\widetilde{s}_0$ whose rigid generic fiber $\Delta=\DD_L$ is the rigid analytic disc considered above, on which the entries of $Y_G$ converge. The relation between these objects is summarized schematically by the following diagram.\footnote{The left-hand square is taken in the rigid analytic category, while the right-hand square is taken in the formal category.}
	\begin{center}
		\begin{tikzcd}
			\CX_{\Delta}\arrow[dd, "f_{\Delta}"] \arrow[rr] \arrow[rd] & & \mathfrak{X}_{\DD} \arrow[dd, "\mathfrak{f}_{\DD}"] \arrow[rd] & \\
			& \CX \arrow[dd, "f"] \arrow[rr] & & \mathfrak{X} \arrow[dd, "\mathfrak{f}"] \\
			\Delta \arrow[rr] \arrow[rd] \arrow[rdd] & {} \arrow[r] & \DD \arrow[rd] \arrow[rdd] & \\
			& S \arrow[rr] \arrow[d] & & \mathfrak{S} \arrow[d] \\
			& \spec(L) \arrow[rr] & & \spec(V)
		\end{tikzcd}
	\end{center}
	% CLEAN
	
	By the proof of \cite[Proposition~2.16]{ogus}, applied with $S_K=\Delta$ and $S=\DD$ in the notation of loc.~cit., the differential module $(H^1_{dR}(\CX_\Delta/\Delta),\nabla)$ gives rise to a crystal of $L\otimes\CO_{\DD/V}$-modules, which we denote by $\mathcal{E}:=\sigma(H^1_{dR}(\CX_\Delta/\Delta),\nabla)$.
	% CLEAN
	
	The morphism $\mathfrak{f}_{\DD}:\mathfrak{X}_{\DD}\rightarrow\DD$ is smooth and proper. Smoothness follows from the fact that the semi-abelian scheme $\mathfrak{X}\rightarrow\mathfrak{S}$ is smooth. For properness, recall that $\mathfrak{X}\rightarrow\mathfrak{S}$ is constructed using Gabber's lemma so that the pullback along the section $\widetilde{s}_0:\spec(V)\rightarrow\mathfrak{S}$ is the connected Néron model of $\CX_0$.\footnote{See the proof of \cite[Lemma~3.4, p.~213]{andre1989g}; see also \cite{delignegabber} for the original result.} Since $\CX_0$ has good reduction, this connected Néron model is an abelian scheme, and hence its special fiber is proper. Properness of the adic morphism $\mathfrak{f}_{\DD}$ then follows by the corresponding criterion for locally Noetherian formal schemes; see \cite[Tag~0AM5]{stacks-project}.
	% CLEAN

	By \cite[\S 2]{ogus}, $\mathcal{E}$ is a convergent isocrystal on $\DD/V$. On the other hand, \cite[Theorems~3.1 and~3.7]{ogus} associate with $\mathfrak{f}_{\DD}$ a convergent $F$-isocrystal
	$$R^1(\mathfrak{f}_{\DD})_*\CO_{\mathfrak{X}_{\DD}/L}$$
	on $\DD/V$.
	% CLEAN
	
	Working under \Cref{assumptionpadicproximity}, fix two points $s,t\in\Delta\cap S(\bar{L})$ such that the corresponding sections $\widetilde{s},\widetilde{t}:\spec(\CO_{\bar{L}})\rightarrow\mathfrak{S}$ satisfy
	$$
	\widetilde{s}|_{\spec(\bar{k})}=\widetilde{t}|_{\spec(\bar{k})}=\widetilde{s}_0|_{\spec(\bar{k})},
	$$
	where $\bar{k}$ denotes an algebraic closure of $k$ and $\widetilde{s}_0$ is understood after base change to $\CO_{\bar{L}}$.
	% CLEAN
	
	The discussion in \cite[\S\S 3.8--3.10]{ogus} identifies $\mathcal{E}$ with $R^1(\mathfrak{f}_{\DD})_*\CO_{\mathfrak{X}_{\DD}/L}$ and thereby equips $\mathcal{E}$ with an $F$-structure. The proof of \cite[Corollary~5.9]{ogus} then yields, for $P,Q\in\{s,t,s_0\}$, isomorphisms
	$$
	\epsilon(P,Q):H^1_{dR}(\CX_P/L(P,Q))\rightarrow H^1_{dR}(\CX_Q/L(P,Q)).
	$$
	% CLEAN
	
	As explained in \cite[Remark~5.14.3]{ogus}, these isomorphisms fit into a commutative diagram
	\begin{center}
		\begin{tikzcd}
			H^1_{dR}(\CX_P/L(P,Q)) \arrow[d, "\sigma_{\crys,\widetilde{P}}"'] \arrow[r,"{\epsilon(P,Q)}"] & H^1_{dR}(\CX_Q/L(P,Q)) \arrow[d, "\sigma_{\crys,\widetilde{Q}}"] \\
			L(P,Q)\otimes H^1_{\crys}(\widetilde{\mathfrak{X}}_P/W) \arrow[r, "{\alpha(P,Q)}^*"] & H^1_{\crys}(\widetilde{\mathfrak{X}}_Q/W)\otimes L(P,Q)
		\end{tikzcd}
	\end{center}
	where $\widetilde{\mathfrak{X}}_P$ denotes the special fiber of $\mathfrak{X}_P$ and $\alpha(P,Q):\widetilde{\mathfrak{X}}_Q\rightarrow\widetilde{\mathfrak{X}}_P$ is a uniquely determined isogeny. The comparison isomorphisms $\sigma_{crys,\widetilde{P}}$ agree with those of Berthelot--Ogus \cite{bertogus}; see also \cite[Remark~3.9.2]{ogus}.
	% CLEAN
	
	Taking $Q=s_0$ above, and writing $\epsilon(P,0)$ and $\sigma_{\crys,0}$ for simplicity, we obtain for every $P\in\Delta$ as above a canonical isomorphism
	$$
	\delta(P,0):=\sigma_{\crys,0}\circ\epsilon(P,0):H^1_{dR}(\CX_P/L(P))\rightarrow H^1_{\crys}(\widetilde{\mathfrak{X}}_0/W)\otimes L(P).
	$$
	% CLEAN
	
	The flatness of the Gauss--Manin connection induces cocycle relations on the isomorphisms $\epsilon(\cdot,\cdot)$. After tensoring with $\mathbb{C}_p$, the maps $\delta(P,0)^{-1}$ give parallel transport and identify
	$$
	H^1_{dR}(\CX_\Delta/\Delta)^{\nabla}\simeq H^1_{\crys}(\widetilde{\mathfrak{X}}_0/W)\otimes\mathbb{C}_p.
	$$
	% CLEAN
	
	Fix a basis $\Gamma_v(\CX_0)=(\gamma_1,\ldots,\gamma_{2g})$ of $H^1_{\crys}(\widetilde{\mathfrak{X}}_0/W)$, and write $(e_1,\ldots,e_{2g}):=\Gamma_{dR}(\CX)$. Writing
	$$
	\delta(P,0)(e_i)=\sum_j\varpi_{ij}(P)\gamma_j
	$$
	defines a matrix $(\varpi_{ij}(P))\in\GL_{2g}(\CO_\Delta)$ satisfying the differential system \eqref{eq:diffsys}. Since $Y_G$ is the normalized solution, with $Y_G(0)=I_{2g}$, and $\iota_v(Y_G)\in M_{2g}(\CO_\Delta)$ by assumption, we obtain
	$$
	(\varpi_{ij}(P))=\iota_v(Y_G(x(P)))\cdot(\varpi_{ij}(0))
	$$
	for every $P\in\Delta$.
	% CLEAN
\end{proof}

\begin{remark}\label{remarkramification}
	In \cite[\S3]{andremots}, Andr\'e invokes \cite{bertogus} to identify the horizontal sections of $H^1_{dR}(\CX/S)$ with $H^1_{\crys}(\widetilde{\CX}_{0,v}/W(k(v)))\otimes\mathbb{C}_p$ in the notation used here. The results of \cite{bertogus}, however, do not appear to cover this identification at places $v\in\Sigma_{K,f}$ ramified over $\Q$; see the introduction of \cite{ogus}. Indeed, when $v$ is ramified, one no longer has $\CO_{K_v}=W(k(v))$, and the theory of convergent $F$-isocrystals developed by Ogus is needed in the argument above.
\end{remark}
% CLEAN
		\subsubsection{The associated family of G-functions}\label{section:settinggfuns}

Having constructed the G-functions associated with a single center, we now assemble the corresponding matrices at the points $\xi_j$ into a single associated family. We retain the setting and notation of \Cref{reduction}, and write $\{\xi_1,\ld,\xi_l\}$ for the simple zeros of $x$.
% CLEAN

By the Galois property of $x:S\rightarrow\mathbb{P}^1$ in \Cref{reduction}, for each $\xi_j$ there exists $\sigma_j\in\aut_x(S)$ such that $\sigma_j(\xi_1)=\xi_j$. Let $f_j:\CX_j\rightarrow S$ be the pullback of $f:\CX\rightarrow S$ along $\sigma_j$. Applying \Cref{gfuns} to $f_j$ with center $s_0=\xi_1$ yields a matrix of G-functions $Y_{G,j}$. We denote by $\mathcal{Y}_j$ the set of entries of $Y_{G,j}$.
% CLEAN

Following Daw and Orr \cite{daworr4}, some of the abelian schemes $f_j:\CX_j\rightarrow S$ may be generically isogenous. They therefore define an equivalence relation on $\{1,\ld,l\}$ by
$$
j\sim j'\qquad\Longleftrightarrow\qquad \CX_j\text{ and }\CX_{j'}\text{ are generically isogenous}.
$$
Let $\Lambda$ denote the set of equivalence classes, and identify each $\lambda\in\Lambda$ with the smallest index in its class. We then set $\mathcal{Y}:=\bigsqcup_{\lambda\in\Lambda}\mathcal{Y}_{\lambda}$. For the interaction of the families $\CX_\lambda$ with the integral models, see \cite[\S6.E]{daworr5}.
% CLEAN

\begin{definition}\label{deffamilygfuns}
	We call $\mathcal{Y}=\bigsqcup_{\lambda\in\Lambda}\mathcal{Y}_\lambda$, equivalently the family consisting of the entries of the matrices $Y_{G,\lambda}$ for $\lambda\in\Lambda$, the \textbf{family of G-functions associated to} $(f:\CX\rightarrow S,x)$ \textbf{centered at} $s_0:=\xi_1$.
\end{definition}
% CLEAN

\begin{remark}\label{remark g functions and covers}
	Let $f:\CX\rightarrow S$ be a one-parameter family of principally polarized $g$-dimensional abelian varieties defined over a number field $K$, and let $s_0\in S(K)$. After replacing $K$ by a finite extension and passing to a finite cover $S'\rightarrow S$, the pullback family
	$$
	f':\CX\times_S S'\longrightarrow S'
	$$
	may be placed in the setting of \Cref{reduction}, with the points $\{\xi_1,\ld,\xi_l\}$ lying above $s_0$. 
	
	The passage to the auxiliary cover is part of the construction of the family $\mathcal{Y}$ from the original pair $(f,s_0)$. Thus, even when $f$ itself does not satisfy the hypotheses of \Cref{reduction}, \Cref{deffamilygfuns} associates a family of G-functions to $(f,s_0)$ after passing to such auxiliary data.
	
	The essential de Rham datum is already present downstairs on $S$. The Hodge basis of sections may be chosen on the original family, independently of the auxiliary cover and of the function $x$. Indeed, a Hodge basis of sections chosen on an affine neighborhood of $s_0$ pulls back to the corresponding neighborhood of the cover, which contains all the points $\xi_j$. The Galois automorphisms $\sigma_j$ then allow the resulting families of G-functions to be expressed with respect to a single center $\xi_1$. In particular, no independent choice of de Rham basis is made at the different points $\xi_j$.
\end{remark}
% CLEAN

\subsubsection{$v$-adic proximity}\label{section:vadicproximity}

Consider the $\CO_K$-model $\mathfrak{S}'$ of $S$ introduced in \Cref{reduction}. For $s\in S(\bar{\Q})$, let
$$
\widetilde{s}:\spec\CO_{K(s)}\rightarrow\mathfrak{S}'\times_{\spec\CO_K}\spec\CO_{K(s)}
$$
denote the induced section.
% CLEAN

The following observation of Andr\'e relates $v$-adic closeness measured by $x$ to equality of reduction; see \cite[Ch.~X, \S3.1]{andre1989g}. There exist constants $\kappa_v>0$, for $v\in\Sigma_{K,f}$, with $\kappa_v=1$ for all but finitely many $v$, such that if $s\in S(\bar{\Q})$ and $w\in\Sigma_{K(s)}$ satisfies $w\mid v$ and
$$
|x(s)|_w<\kappa_v^{[K(s):\Q]},
$$
then $\widetilde{s}$ and $\widetilde{\xi}_t$ have the same image in $\mathfrak{S}'(\mathbb{F}_{p(w)})$ for some $1\leq t\leq l$.
% CLEAN

Choose $\zeta\in K^\times$ such that $|\zeta|_v\leq\kappa_v$ for every $v\in\Sigma_{K,f}$, as in the discussion following \cite[Ch.~X, \S3.1]{andre1989g}. We may moreover choose $\zeta$ so that $|\zeta|_v=1$ whenever $\kappa_v=1$, by finiteness of the ideal class group. We modify Andr\'e's construction so as to retain a symplectic Hodge basis while obtaining the same control on $v$-adic proximity.
% CLEAN

Let $\Gamma_{dR}(\CX)=\left(\underline{\omega},\underline{\eta}\right)$ be a Hodge basis of $H^1_{dR}(\CX/S)$ over an affine neighborhood $U_0$ of $s_0$. Set $H(x):=\frac{\zeta}{\zeta-x}$ and define
$$
\Gamma_{dR,\new}(\CX):=\left(H(x)\underline{\omega},H(x)^{-1}\underline{\eta}\right).
$$

After removing from $U_0$ the finitely many points at which $x=\zeta$, the function $H(x)$ is invertible, and $\Gamma_{dR,\new}(\CX)$ is again a Hodge basis of $H^1_{dR}(\CX/S)(U_0)$.
% CLEAN

We record the effect of this modification on the associated matrix of G-functions.

\begin{lemma}\label{lemgfunsnew}
	Let $\Gamma_{dR}(\CX)$ and $\Gamma_{dR,\new}(\CX)$ be as above. Let $Y_G$ and $Y_{G,\new}$ denote the matrices of G-functions associated with $\Gamma_{dR}(\CX)$ and $\Gamma_{dR,\new}(\CX)$, respectively, by the archimedean construction in the proof of \Cref{gfuns}. Then
	\begin{equation}
		Y_{G,\new}=
		\begin{pmatrix}
			\operatorname{diag}(H) & 0\\
			0 & \operatorname{diag}(H^{-1})
		\end{pmatrix}Y_G.
	\end{equation}
	Moreover, after possibly decreasing the constants $\kappa_v$, we may choose $\zeta\in K^\times$ such that
	$$
	\min\{1,R_v(Y_{G,\new})\}\leq\min\{1,R_v(Y_G),\kappa_v\}
	$$
	for every $v\in\Sigma_{K,f}$.
\end{lemma}
% CLEAN
\begin{proof} 
	For the first assertion, let $A_{\new}$ denote the matrix of the differential system associated with $\Gamma_{dR,\new}(\CX)$. Set $D:=\operatorname{diag}(H I_g,H^{-1}I_g)$. Since $H'/H=H/\zeta$, the change of basis gives
	\begin{equation}
		A_{\new}=D'D^{-1}+DAD^{-1}
		=\operatorname{diag}\left(\frac{H}{\zeta}I_g,-\frac{H}{\zeta}I_g\right)+DAD^{-1},
	\end{equation}
	where $A$ is the matrix of the differential system \eqref{eq:diffsys} associated with $\Gamma_{dR}(\CX)$.
	% CLEAN
	
	Since $H(0)=1$, we have $D(0)=I_{2g}$. Hence $DY_G$ is a normalized uniform solution of the differential system with matrix $A_{\new}$, and therefore
	$$
	Y_{G,\new}=DY_G
	$$
	by uniqueness of the normalized uniform solution.
	% CLEAN
	
	For the moreover part, let $\Sigma_G:=\{v\in\Sigma_{K,f}:\kappa_v\neq1\}$, which is finite, and write $Y_G=(y_{ij}(x))$. For each $v\in\Sigma_G$, set
	$$
	r_v:=\min\{|\xi|_v:\xi\neq0,\ \iota_v(y_{ij})(\xi)=0\text{ for some }i,j\text{ with }y_{ij}\neq0\},
	$$
	with the convention $r_v=+\infty$ if the set on the right is empty.
	% CLEAN
	
	For each $v\in\Sigma_G$, we have $r_v>0$, since each nonzero power series $\iota_v(y_{ij})$ has no nonzero zeros in a sufficiently small neighborhood of $0$. Hence, setting
	$$
	\kappa_v':=\frac{1}{2}\min\{1,R_v(Y_G),\kappa_v,r_v\},
	$$
	we have $\kappa_v'>0$.
	% CLEAN
	
	Replace $\kappa_v$ by $\kappa_v'$ in the choice of $\zeta$. Then $|\zeta|_v<r_v$ and $|\zeta|_v<R_v(Y_G)$ for every $v\in\Sigma_G$. Since $y_{11}(0)=1$, the definition of $r_v$ gives $\iota_v(y_{11})(\zeta)\neq0$. The $(1,1)$-entry of $Y_{G,\new}$ is $H\,y_{11}$, and therefore has a pole at $x=\zeta$. Consequently,
	$$
	R_v(Y_{G,\new})\leq|\zeta|_v\leq\kappa_v'
	\leq\min\{1,R_v(Y_G),\kappa_v\},
	$$
	which gives the required bound for $v\in\Sigma_G$. If $v\notin\Sigma_G$, then $\kappa_v=|\zeta|_v=1$, so both $D$ and $D^{-1}$ are analytic on the open unit disc. Since $Y_{G,\new}=DY_G$, it follows that
	$$
	\min\{1,R_v(Y_{G,\new})\}=\min\{1,R_v(Y_G)\},
	$$
	and the required bound follows in this case as well.
	% CLEAN
	\end{proof}
\begin{remark}
	In what follows, we work implicitly with the family of G-functions given by the entries of $Y_{G,\new}$. The distinction from the construction used by Andr\'e and, more recently, by Daw--Orr is that the auxiliary function $H$ is here absorbed into the choice of Hodge basis rather than adjoined to the family of G-functions. In particular, the same family of G-functions both controls the relevant $v$-adic radii of convergence and retains the cohomological interpretation and the trivial relations described in \Cref{trivialrelations}.
	
	This avoids the need to enlarge the family by an auxiliary G-function which does not occur in the relations themselves. Such an enlargement requires additional care in the effective application of the Andr\'e--Bombieri height bound; see \cite[Remark~2.10, \S5.8, Remark~5.14]{daworr4}.
\end{remark}
% CLEAN / substantive explanation of the basis modification

Return to the family $\mathcal{Y}=\bigsqcup_{\lambda\in\Lambda}\mathcal{Y}_\lambda$ of \Cref{section:settinggfuns}. Since $\Lambda$ is finite, we may choose a single $\zeta\in K^\times$ for which the preceding construction applies simultaneously to every $Y_{G,\lambda}$. Applying the corresponding modification to each $Y_{G,\lambda}$ gives matrices $Y_{G,\new,\lambda}$. Replacing $Y_{G,\lambda}$ by $Y_{G,\new,\lambda}$ in the definition of $\mathcal{Y}$, the preceding construction gives the following property.
% CLEAN
\begin{lemma}\label{propertyvadicprox}
	Let $\xi:=x(s)$. If $|\xi|_w<\min\{1,R_w(\mathcal{Y})\}$ for some $w\in\Sigma_{K(s),f}$, then $\widetilde{s}$ and $\widetilde{\xi}_t$ have the same image in $\mathfrak{S}'(\mathbb{F}_{p(w)})$ for some $1\leq t\leq l$.
\end{lemma}
% CLEAN

Let $v\in \Sigma_{K,f}$. There are two natural notions of $v$-adic proximity to a center $\xi_j$: lying within the $v$-adic radius of convergence of the associated G-functions, and having the same reduction modulo $v$ as $\xi_j$, in which case the period-theoretic interpretation of \Cref{gfuns} applies. The modification above is arranged so that the first implies the second.

\begin{definition}
	Let $s\in S(\bar{\Q})$ and $w\in\Sigma_{K(s)}$. We say that $s$ is $w$-adically close to $0$ if
	$$
	|x(s)|_w<\min\{1,R_w(\mathcal{Y})\}.
	$$
	
	For $1\leq t\leq l$, we say that $s$ is $w$-adically close to $\xi_t$ if $s$ is $w$-adically close to $0$ and belongs to the connected component of
	$$x^{-1}\big(\Delta(0,\min\{1,R_w(\mathcal{Y})\})\big)$$
	containing $\xi_t$.
\end{definition}
% CLEAN

We close this construction by recording the dependence of the associated family of G-functions on the choices entering its construction.
\begin{remark}
	In the setting of \Cref{remark g functions and covers}, fix the auxiliary geometric data of \Cref{reduction} and a Hodge basis of sections on an affine neighborhood of $s_0$ in the original family. As explained there, this determines the corresponding de Rham data on the auxiliary cover. The matrices of G-functions are then uniquely determined by their respective differential systems together with the normalization $Y_G(0)=I_{2g}$.
	
	The only further auxiliary choice entering the proximity modification is the element $\zeta$ in $H(x)=\zeta/(\zeta-x)$. Thus, once the auxiliary geometric data and the initial Hodge basis of sections are fixed, the only remaining choice in the construction is $\zeta$; after this choice, no further independent choices of de Rham bases or normalized G-function matrices are made.
\end{remark}
% CLEAN
		\subsection{Trivial relations}\label{section:trivial}
% KEEP
Having constructed the associated family of G-functions and established its local period-theoretic and $v$-adic proximity properties, we now determine the polynomial functional relations among the G-functions. This will suffice to prove that the relations constructed in the sequel are non-trivial.
% CLEAN

By a trivial relation among the values of the G-functions at an algebraic number $\xi$ we mean a polynomial relation obtained by specializing at $\xi$ a polynomial functional relation among the G-functions themselves; see \cite[Ch.~VII, \S5]{andre1989g}.
% CLEAN
\begin{prop}\label{trivialrelations}
	Let $\mathcal{Y}$ be the family of G-functions associated with $f:\CX\rightarrow S$ as in \Cref{deffamilygfuns}, and assume that the image of the induced morphism $S\rightarrow\mathcal{A}_g$ is Hodge generic. Then
	$$\mathcal{Y}^{\zar}=(\SP_{2g}^{\Lambda})_{\bar{\Q}(x)}\subset(M_{2g}^{\Lambda})_{\bar{\Q}(x)}.$$
	Equivalently, $\mathcal{Y}^{\zar}$ is cut out by the ideal
	\begin{equation}\label{eq:trivialideal}
		I_{\Lambda}:=\sum_{\lambda\in\Lambda}I(\SP_{2g})_{\lambda},
	\end{equation}
	where $I(\SP_{2g})_{\lambda}$ denotes the extension to $\Q[X_{i,j,\mu}:1\leq i,j\leq2g,\ \mu\in\Lambda]$ of the ideal defining
	$$\SP_{2g}\subset M_{2g}=\spec\Q[X_{i,j,\lambda}:1\leq i,j\leq2g]$$
	in the $\lambda$-th factor.
\end{prop}
% CLEAN
\begin{proof}
		Fix an archimedean embedding $\iota:K\hookrightarrow\C$ and a sufficiently small disc $\Delta$ centered at $s_0=\xi_1$. For each $\lambda\in\Lambda$, let $\Gamma_{dR}(\CX_\lambda)$ be the Hodge basis used to define $Y_{G,\lambda}$, and choose a symplectic trivializing frame of $R^1f_{\lambda,*}\Q|_\Delta$. Let $\CP_\lambda$ be the corresponding relative period matrix, and set $f_\Lambda:\prod_{S,\lambda\in\Lambda}\CX_\lambda\rightarrow S$ and $\CP_\Lambda:=(\CP_\lambda)_{\lambda\in\Lambda}$. Thus $\CP_\Lambda$ is the relative period matrix of $f_\Lambda$ with respect to the product bases. 
		
		The relative motivic Galois group associated with $f_\Lambda$ is $\SP_{2g}^{\Lambda}$. Indeed, by \cite[Corollary~4.17]{bakkertsimermanandregroth} it is the Zariski closure of the monodromy of $R^1f_{\Lambda,*}\Q$, while \cite[Proposition~6.8 and Lemma~6.11]{daworr5} identifies this closure with $\SP_{2g}^{\Lambda}$.
		
		By the geometric Andr\'e--Grothendieck period theorem of Ayoub \cite{ayoub}, and since $\C(S)/\C(x)$ is algebraic, we have
		$$
		\operatorname{trdeg}_{\C(x)}
		\C(x)\bigl(\text{entries of }\CP_\Lambda\bigr)
		\geq \dim \SP_{2g}^{\Lambda}.
		$$
		
		On $\Delta$, \Cref{gfuns} gives $\CP_\lambda=Y_{G,\lambda}\Pi_\lambda$, where $\Pi_\lambda$ is the period matrix of the fiber of $\CX_\lambda$ at $s_0$. Set
		$$
		\Pi_\Lambda:=(\Pi_\lambda)_{\lambda\in\Lambda},
		\qquad
		Y_G:=(Y_{G,\lambda})_{\lambda\in\Lambda}.
		$$
		Then $\CP_\Lambda=Y_G\Pi_\Lambda$. Since $\Pi_\Lambda$ is constant and invertible,
		$$
		\operatorname{trdeg}_{\C(x)}
		\C(x)\bigl(\text{entries of }Y_G\bigr)
		=
		\operatorname{trdeg}_{\C(x)}
		\C(x)\bigl(\text{entries of }\CP_\Lambda\bigr)
		\geq \dim\SP_{2g}^{\Lambda}.
		$$
		
		The values of $\CP_\lambda$ and the matrix $\Pi_\lambda$ lie in the $\SP_{2g,\C}$-torsor
		$$
		\mathcal{T}_{2g}:=
		\left\{
		M\in\GL_{2g}(\C):
		M^TJ_{2g}M=\frac{1}{2\pi i}J_{2g}
		\right\},
		\qquad
		J_{2g}:=
		\begin{pmatrix}
			0&I_g\\
			-I_g&0
		\end{pmatrix},
		$$
		determined by the archimedean normalization. Thus their quotient
		$Y_{G,\lambda}=\CP_\lambda\Pi_\lambda^{-1}$ satisfies the symplectic relations.
		
		Consequently, the graph of $Y_G$ is contained in
		$$
		\mathbb{A}^1_\C\times\SP_{2g,\C}^{\Lambda}.
		$$
		The preceding transcendence-degree bound therefore shows that this graph is Zariski dense in $\mathbb{A}^1_\C\times\SP_{2g,\C}^{\Lambda}$. Since the entries of $Y_G$ lie in $\bar{\Q}[[x]]$, it follows that
		$$
		\mathcal{Y}^{\zar}=(\SP_{2g}^{\Lambda})_{\bar{\Q}(x)},
		$$
		as claimed.
		% CLEAN
\end{proof}

We record for later use the following.

\begin{lemma}\label{idealprime}
	The ideal
	$$
	I_\Lambda\subset\bar{\Q}[X_{i,j,\lambda}:1\leq i,j\leq 2g,\ \lambda\in\Lambda]
	$$
	of \Cref{trivialrelations} is prime.
\end{lemma}
% CLEAN
\begin{proof}
	The assertion follows from the argument of \cite[Lemma~5.10]{daworr5}, using \cite[Proposition~5.11]{daworr5}.
\end{proof}
% CLEAN

	%\part{relations at finite places}
		\section{Splittings in $\mathcal{A}_2$}\label{section:relsa2}

In this section, let $f:\CX\rightarrow S$ be a one-parameter family of principally polarized abelian surfaces defined over a number field $K$, and assume throughout that the hypotheses of \Cref{reduction} hold. We construct polynomial relations among the values of the associated $G$-functions at points $s\in S(\bar{\Q})$ for which $\CX_s$ is isogenous to a product of elliptic curves. The cases relevant to the Zilber--Pink conjecture in $\mathcal{A}_2$ will arise as particular instances of this splitting setting. Whenever $\CX_s$ is an $E\times CM$-abelian surface, we label the factors so that $E_s'$ is the CM elliptic curve.
%CLEAN

Fix a center $s_0:=\xi_1\in\{\xi_1\ld\xi_l\}$ and restrict to points $s\in S(\bar{\Q})$ that are $v$-adically close to $s_0$. Although the construction of \Cref{section:settinggfuns} is multicentered, throughout this section we work with the single-center matrix $Y_G$ of \Cref{gfuns} associated with the chosen $s_0$.
%CLEAN

Throughout this section, assume moreover that the distinguished fiber $\CX_0$ admits an isogeny
$\theta_0:\CX_0\rightarrow\CX_0'$,
where $\CX_0':=E_0\times_K E_0'$ and $E_0,E_0'$ are everywhere semistable elliptic curves defined over $K$, as in \Cref{reduction}.
%CLEAN

Fix Hodge bases $\Gamma_{dR}(E_0)$ and $\Gamma_{dR}(E_0')$, and let $\Gamma_{dR}(\CX_0')$ be the induced Hodge basis of \Cref{remarkbases}. We also fix a Hodge basis of sections $\Gamma_{dR}(\CX/S):=(\omega_1,\omega_2,\eta_1,\eta_2)$ whose specialization at $s_0$ is of the form $\Gamma_{dR}(\CX_0,\theta_0)$ in the notation of \Cref{section:isogenyhodgebasis}. By \Cref{gfuns}, this basis determines a matrix $Y_G(x)\in M_4(\bar{\Q}[[x]])$ whose entries are $G$-functions.
%CLEAN

\subsection{The relation-generating identity}\label{section:splittingrelations}
%CLEAN

Let $v\in\Sigma_K$ be either an archimedean place or a finite place of good reduction of $\CX_0$. Fix symplectic bases $\Gamma_v(E_0):=(\gamma_0,\delta_0)$ and $\Gamma_v(E_0'):=(\gamma_0',\delta_0')$ of $H^1_v(E_0)$ and $H^1_v(E_0')$, respectively, and let $\Gamma_v(\CX_0'):=(\gamma_0,\gamma_0',\delta_0,\delta_0')$ be the induced ordered symplectic basis of $H^1_v(\CX_0')$.
%CLEAN

\begin{lemma}\label{lemisogfin}
	Let $s\in S(K)$ be $v$-adically close to $s_0$, and assume that there exists an isogeny
	$$
	\theta_s:\CX_s\longrightarrow \CX_s':=E_s\times_K E_s',
	$$
	where $E_s$ and $E_s'$ are elliptic curves. 
	
	Choose Hodge bases $\Gamma_{dR}(E_s)$ and $\Gamma_{dR}(E_s')$, and let $\Gamma_{dR}(\CX_s')$ be the induced Hodge basis of \Cref{remarkbases}. Choose moreover symplectic bases $\Gamma_v(E_s)$ and $\Gamma_v(E_s')$, and let $\Gamma_v(\CX_s')$ be the corresponding ordered symplectic basis.
	
	Set
	$$
	J_{2,3}:=\begin{pmatrix}1&0&0&0\\0&0&1&0\\0&1&0&0\\0&0&0&1\end{pmatrix}.
	$$
	Then there exists $\Theta_v\in\GL_4(\C_v)$ such that
	\begin{equation}\label{eq:e2fmain}
		\begin{gathered}
			\iota_v\!\left(J_{2,3}\cdot[\theta_s]_{dR}\cdot Y_G(x(s))\cdot[\theta_0]_{dR}^{-1}\cdot J_{2,3}\right)=\\
			=\frac{1}{\deg(\theta_0)}
			\begin{pmatrix}\Pi_v(E_s)&0\\0&\Pi_v(E_s')\end{pmatrix}\cdot\Theta_v\cdot\begin{pmatrix}\Pi_v(E_0)^{-1}&0\\0&\Pi_v(E_0')^{-1}\end{pmatrix}.
		\end{gathered}
	\end{equation}
	
	Here $[\theta_s]_{dR}$ denotes the matrix of $\theta_s^*:H^1_{dR}(\CX_s'/K)\rightarrow H^1_{dR}(\CX_s/K)$ with respect to $\Gamma_{dR}(\CX_s')$ and the specialization of $\Gamma_{dR}(\CX/S)$ at $s$. Similarly, $[\theta_0]_{dR}^{-1}$ denotes the inverse of the matrix of $\theta_0^*:H^1_{dR}(\CX_0'/K)\rightarrow H^1_{dR}(\CX_0/K)$ with respect to the corresponding Hodge bases fixed above.
\end{lemma}
%CLEAN

\begin{proof}
For $P\in\{s,0\}$, functoriality of the de Rham--crystalline and de Rham--Betti comparison isomorphisms gives the commutative diagram
\begin{center}
	\begin{tikzcd}
		H^1_{dR}(\CX_P/K)\otimes_K\C_v \arrow[r, "\CP_v(\CX_P)"] & H^1_v(\CX_P)\otimes_{\Q_v}\C_v \\
		H^1_{dR}(\CX'_P/K)\otimes_K\C_v \arrow[u, "\theta^*_{P,dR}"] \arrow[r, "\CP_v(\CX'_P)"] & H^1_v(\CX'_P)\otimes_{\Q_v}\C_v \arrow[u, "\theta^*_{P,v}"'],
	\end{tikzcd}
\end{center}
where $\Q_v$ denotes $\Q$ if $v\in\Sigma_{K,\infty}$ and $K_{v,0}:=\operatorname{Frac}(W(\mathbb{F}_{p(v)}))$ if $v\in\Sigma_{K,f}$.
%CLEAN

With respect to the chosen bases, the preceding diagram gives
\begin{equation}\label{eq:lemisogmat1}
	[\theta_P]_{dR}\cdot\CP_v(P)=J_{2,3}\cdot\begin{pmatrix}\Pi_v(E_P)&0\\0&\Pi_v(E_P')\end{pmatrix}\cdot J_{2,3}\cdot[\theta_P]_v.
\end{equation}
%CLEAN
By \Cref{gfuns},
$$
\CP_v(s)=\iota_v(Y_G(x(s)))\cdot\CP_v(0).
$$
Substituting \eqref{eq:lemisogmat1} for $P=0$ gives
\begin{equation}
	\CP_v(s)=\iota_v(Y_G(x(s)))\cdot[\theta_0]_{dR}^{-1}\cdot J_{2,3}\cdot\begin{pmatrix}\Pi_v(E_0)&0\\0&\Pi_v(E_0')\end{pmatrix}\cdot J_{2,3}\cdot[\theta_0]_v.
\end{equation}
Using $J_{2,3}^{-1}=J_{2,3}$, we may therefore rewrite \eqref{eq:lemisogmat1} for $P=s$ as
\begin{equation}\label{eq:lemisogmat2}
	\begin{gathered}
		J_{2,3}\cdot[\theta_s]_{dR}\cdot\iota_v(Y_G(x(s)))\cdot[\theta_0]_{dR}^{-1}\cdot J_{2,3}=\\
		=\begin{pmatrix}\Pi_v(E_s)&0\\0&\Pi_v(E_s')\end{pmatrix}\cdot J_{2,3}\cdot[\theta_s]_v\cdot[\theta_0]_v^{-1}\cdot J_{2,3}\cdot\begin{pmatrix}\Pi_v(E_0)^{-1}&0\\0&\Pi_v(E_0')^{-1}\end{pmatrix}.
	\end{gathered}
\end{equation}
%CLEAN

Let $\theta_0^\vee:\CX_0'\rightarrow\CX_0$ be the dual isogeny of $\theta_0$. Then
$$
[\theta_0^\vee]_v\cdot[\theta_0]_v=[\theta_0]_v\cdot[\theta_0^\vee]_v=\deg(\theta_0)I_4,
$$
and therefore
$$
[\theta_0]_v^{-1}=\frac{1}{\deg(\theta_0)}[\theta_0^\vee]_v.
$$
%CLEAN

Substituting this identity into \eqref{eq:lemisogmat2} gives \eqref{eq:e2fmain} with
\begin{equation}\label{eq:thetadesc}
	\Theta_v=J_{2,3}\cdot[\theta_s]_v\cdot[\theta_0^\vee]_v\cdot J_{2,3}\in\GL_4(\C_v).
\end{equation}
\end{proof}
%CLEAN

If $v\in\Sigma_{K,\infty}$, then $\Theta_v\in\GL_4(\Q)$. Indeed, by \eqref{eq:thetadesc}, it is determined by the matrices $[\theta_P]_v\in\GL_4(\Q)$ of the pullbacks
$\theta_P^*:H^1(\CX_P',\Q)\rightarrow H^1(\CX_P,\Q)$
for $P\in\{s,0\}$. At finite places of good reduction, $\Theta_v$ need not be defined over $\Q$; instead, its block structure is governed by morphisms between the reductions of the elliptic curves appearing in \Cref{lemisogfin}.
%CLEAN

\begin{lemma}\label{lemisogfin2}
	In the setting of \Cref{lemisogfin}, assume moreover that $v\in\Sigma_{K,f}$. Then
	$$
	\Theta_v=\begin{pmatrix}\Theta_{1,1}&\Theta_{1,2}\\\Theta_{2,1}&\Theta_{2,2}\end{pmatrix},
	$$
	where each $\Theta_{i,j}\in M_2(\C_v)$ is induced by a morphism of elliptic curves between one of the reductions of $E_0,E_0'$ and one of the reductions of $E_s,E_s'$. In particular, each $\Theta_{i,j}$ is either zero or invertible.
\end{lemma}
%CLEAN
\begin{proof}
	Set
	$$
	\widetilde{\Theta}_v:=[\theta_s]_v\cdot[\theta_0^\vee]_v,
	$$
	so that, by \eqref{eq:thetadesc},
	$$
	\Theta_v=J_{2,3}\cdot\widetilde{\Theta}_v\cdot J_{2,3}.
	$$
	%CLEAN
	
	Reducing the isogenies modulo $v$ gives
	$$
	\widetilde{\theta}_s:\widetilde{\CX}_{s,v}\longrightarrow \widetilde{E}_{s,v}\times_{\mathbb{F}_{p(v)}}\widetilde{E}_{s,v}'
	\qquad\text{and}\qquad
	\widetilde{\theta}_0^\vee:\widetilde{E}_{0,v}\times_{\mathbb{F}_{p(v)}}\widetilde{E}_{0,v}'\longrightarrow\widetilde{\CX}_{0,v}.
	$$
	Since $s$ is $v$-adically close to $s_0$, we have $\widetilde{\CX}_{0,v}=\widetilde{\CX}_{s,v}$. Composing the two isogenies therefore gives an isogeny
	\begin{equation}
		\phi_v:\widetilde{E}_{0,v}\times_{\mathbb{F}_{p(v)}}\widetilde{E}_{0,v}'\xrightarrow{\widetilde{\theta}_0^\vee}\widetilde{\CX}_{0,v}=\widetilde{\CX}_{s,v}\xrightarrow{\widetilde{\theta}_s}\widetilde{E}_{s,v}\times_{\mathbb{F}_{p(v)}}\widetilde{E}_{s,v}'.
	\end{equation}
	%CLEAN
	
	The induced pullback on crystalline cohomology is
	$$
	\phi_{v,\crys}^*=(\widetilde{\theta}_0^\vee)^*_{\crys}\circ\widetilde{\theta}_s^*{}_{\crys}.
	$$
	With respect to the bases fixed above, its matrix is $\widetilde{\Theta}_v$. Hence, by \eqref{eq:thetadesc}, $\Theta_v$ is the matrix of $\phi_{v,\crys}^*$ with respect to the ordered product bases determined by $\Gamma_v(E_0),\Gamma_v(E_0')$ and $\Gamma_v(E_s),\Gamma_v(E_s')$.
	%CLEAN
	
	The composition
	$$
	\varphi_{1,1}: \widetilde{E}_{0,v}\times_{\mathbb{F}_{p(v)}}\{0\}\rightarrow \widetilde{E}_{0,v}\times_{\mathbb{F}_{p(v)}}\widetilde{E}_{0,v}'\xrightarrow{\phi_v}\widetilde{E}_{s,v}\times_{\mathbb{F}_{p(v)}}\widetilde{E}_{s,v}'\xrightarrow{\pr_1}\widetilde{E}_{s,v}
	$$
	is a morphism of elliptic curves. Similarly, define the morphisms
	$$
	\varphi_{1,2}:\{0\}\times_{\mathbb{F}_{p(v)}}\widetilde{E}_{0,v}'\rightarrow \widetilde{E}_{0,v}\times_{\mathbb{F}_{p(v)}}\widetilde{E}_{0,v}'\xrightarrow{\phi_v}\widetilde{E}_{s,v}\times_{\mathbb{F}_{p(v)}}\widetilde{E}_{s,v}'\xrightarrow{\pr_1}\widetilde{E}_{s,v},
	$$
	$$
	\varphi_{2,1}: \widetilde{E}_{0,v}\times_{\mathbb{F}_{p(v)}}\{0\}\rightarrow \widetilde{E}_{0,v}\times_{\mathbb{F}_{p(v)}}\widetilde{E}_{0,v}'\xrightarrow{\phi_v}\widetilde{E}_{s,v}\times_{\mathbb{F}_{p(v)}}\widetilde{E}_{s,v}'\xrightarrow{\pr_2}\widetilde{E}_{s,v}',
	$$
	and
	$$
	\varphi_{2,2}: \{0\}\times_{\mathbb{F}_{p(v)}}\widetilde{E}_{0,v}'\rightarrow \widetilde{E}_{0,v}\times_{\mathbb{F}_{p(v)}}\widetilde{E}_{0,v}'\xrightarrow{\phi_v}\widetilde{E}_{s,v}\times_{\mathbb{F}_{p(v)}}\widetilde{E}_{s,v}'\xrightarrow{\pr_2}\widetilde{E}_{s,v}'.
	$$
	%CLEAN
	For $1\leq i,j\leq 2$, let $\Theta_{i,j}:=[(\varphi_{i,j})_{\crys}]$ denote the matrix of the induced morphism on crystalline cohomology with respect to the bases fixed above. By construction,
	$$
	\Theta_v=\begin{pmatrix}\Theta_{1,1}&\Theta_{1,2}\\\Theta_{2,1}&\Theta_{2,2}\end{pmatrix}.
	$$
	%CLEAN
\end{proof}

When $\widetilde{E}_{0,v}$ and $\widetilde{E}_{0,v}'$ are not geometrically isogenous, the preceding lemma forces a particularly simple block structure for $\Theta_v$.
%CLEAN
\begin{cor}\label{corollarynonisogenous}
	Let $s\in S(K)$ be as in \Cref{lemisogfin}, and assume that $E_0$ and $E_0'$ are not geometrically isogenous. Let $v\in\Sigma_{K,f}$ be a place for which $s$ is $v$-adically close to $s_0$ and such that $\widetilde{E}_{0,v}\nsim\widetilde{E}_{0,v}'$. Then the matrix $\Theta_v$ in \eqref{eq:e2fmain} is of one of the forms
	$$
	\begin{pmatrix}\Theta_{1,1}&0\\0&\Theta_{2,2}\end{pmatrix}
	\qquad\text{or}\qquad
	\begin{pmatrix}0&\Theta_{1,2}\\\Theta_{2,1}&0\end{pmatrix},
	$$
	where the nonzero blocks belong to $\GL_2(\C_v)$.
\end{cor}
%CLEAN

\begin{proof}
	The morphisms $\varphi_{i,j}$ of the preceding proof are either isogenies or zero; see, for example, \cite[Proposition II.6.8]{hhorne}, together with the fact that $\varphi_{i,j}(0)=0$. If both $\varphi_{1,1}$ and $\varphi_{1,2}$ were nonzero, then $\widetilde{E}_{0,v}$ and $\widetilde{E}_{0,v}'$ would both be isogenous to $\widetilde{E}_{s,v}$, contradicting the assumption on $v$.
	%CLEAN
	
	Assume first that $\varphi_{1,2}=0$ and $\varphi_{1,1}$ is an isogeny. Then $\Theta_{1,2}=0$ and $\Theta_{1,1}$ is invertible. Since $\Theta_v$ is invertible, $\Theta_{2,2}$ is also invertible, and hence $\varphi_{2,2}$ is an isogeny. If $\varphi_{2,1}\neq 0$, then both $\widetilde{E}_{0,v}$ and $\widetilde{E}_{0,v}'$ are isogenous to $\widetilde{E}_{s,v}'$, again contradicting the assumption on $v$. Thus
	$$
	\Theta_v=\begin{pmatrix}\Theta_{1,1}&0\\0&\Theta_{2,2}\end{pmatrix}.
	$$
	%CLEAN
	
	If $\varphi_{1,1}=0$ and $\varphi_{1,2}$ is an isogeny, the same argument gives
	$$
	\Theta_v=\begin{pmatrix}0&\Theta_{1,2}\\\Theta_{2,1}&0\end{pmatrix}.
	$$
	These are the only possibilities, since the invertibility of $\Theta_v$ excludes $\varphi_{1,1}=\varphi_{1,2}=0$.
\end{proof}
%CLEAN

		\subsection{Relations at finite places}\label{section:splitfinite}

We retain the notation introduced before \Cref{lemisogfin}.
%CLEAN
\begin{prop}\label{propisoga2fin}
	Let $s\in S(\bar{\Q})$ be such that $\CX_s$ is isogenous to a product of elliptic curves. Assume that $E_0$ and $E_0'$ are not geometrically isogenous, and set
	$$
	\Sigma(E_0,E_0')_{ngi}:=
	\left\{v\in\Sigma_{K,f}:
	\begin{aligned}
		&\CX_0\text{ has good reduction at }v,\\
		&\widetilde{E}_{0,v}\text{ and }\widetilde{E}_{0,v}'\text{ are not geometrically isogenous}
	\end{aligned}
	\right\}.
	$$
	Then there exists a polynomial $R_{s,f}\in\bar{\Q}[X_{i,j}:1\leq i,j\leq 4]$ satisfying the following:
	\begin{enumerate}
		\item $R_{s,f}$ has coefficients in a finite extension $L_s/K(s)$ with $[L_s:K(s)]$ bounded by an absolute constant $c_f$;
		
		\item $R_{s,f}$ is homogeneous of degree $2$;
		
		\item for every $w\in\Sigma_{L_s,f}$ lying above some $v\in\Sigma(E_0,E_0')_{ngi}$ such that $s$ is $w$-adically close to $s_0$, we have
		$$
		\iota_w(R_{s,f}(Y_G(x(s))))=0;
		$$
		
		\item $R_{s,f}\notin I(\SP_4)$.
	\end{enumerate}
\end{prop}
%CLEAN
\begin{proof}
	Let $L_s$ be the field $L$ supplied by \cite[Theorem 4.2]{silverberg} for $A=\CX_s$ and $B=\CX_s'$. In the notation of loc.\ cit.,
	$$
	[L_s:K(s)]\leq H(4)^2,
	$$
	while \cite[Corollary 3.3]{silverberg} gives $H(4)\leq 2\cdot 36^8$. Hence
	$$
	[L_s:K(s)]\leq 4\cdot 36^{16}.
	$$
	%CLEAN
	
	After base change to $L_s$, the isogeny $\theta_s:\CX_s\rightarrow\CX_s'=E_s\times_{L_s}E_s'$ is defined over $L_s$. Fix $w\in\Sigma_{L_s,f}$ lying above some $v\in\Sigma(E_0,E_0')_{ngi}$, and assume that $s$ is $w$-adically close to $s_0$.
	%CLEAN
	
	Set, in the notation of \Cref{lemisogfin},
	$$
	(F_{i,j}(x)):=J_{2,3}\cdot[\theta_s]_{dR}\cdot Y_G(x)\cdot[\theta_0]_{dR}^{-1}\cdot J_{2,3}.
	$$
	Each $F_{i,j}(x)\in\bar{\Q}[[x]]$ is a linear combination of the entries of $Y_G(x)$. Moreover, these linear combinations are independent of the finite place, since they depend only on the de Rham matrices $[\theta_s]_{dR}$ and $[\theta_0]_{dR}^{-1}$.
	%CLEAN
	
	Define $R_{s,f}\in\bar{\Q}[X_{i,j}:1\leq i,j\leq 4]$ by
	$$
	R_{s,f}(Y_G(x))=F_{1,2}(x)F_{2,4}(x).
	$$
	Then $R_{s,f}$ is homogeneous of degree $2$ and has coefficients in $L_s$. 
	%REVISIT
	
	In this notation, \eqref{eq:e2fmain} becomes
	\begin{equation}\label{eq:e2fprep}
		\iota_w\bigl((F_{i,j}(x(s)))\bigr)=\frac{1}{\deg(\theta_0)}
		\begin{pmatrix}\Pi_w(E_s)&0\\0&\Pi_w(E_s')\end{pmatrix}\cdot\Theta_w\cdot\begin{pmatrix}\Pi_w(E_0)^{-1}&0\\0&\Pi_w(E_0')^{-1}\end{pmatrix}.
	\end{equation}
	%CLEAN
	
	Since the de Rham bases are Hodge bases, we may write
	\begin{equation}\label{eq:thematrices}
		[\theta_s]_{dR}=\begin{pmatrix}A_s&0\\B_s&C_s\end{pmatrix},\qquad
		[\theta_0]_{dR}^{-1}=\begin{pmatrix}A_0&0\\B_0&C_0\end{pmatrix}.
	\end{equation}
	Since $\theta_s$ and $\theta_0$ are isogenies, we have $A_P,C_P\in\GL_2(L_s)$ for $P\in\{s,0\}$.
	%CLEAN
	
	Set
	$$
	\widetilde{F}(x):=J_{2,3}\cdot(F_{i,j}(x))\cdot J_{2,3},
	$$
	and write
	$$
	Y_G(x)=\begin{pmatrix}Y_{1,1}(x)&Y_{1,2}(x)\\Y_{2,1}(x)&Y_{2,2}(x)\end{pmatrix},
	\qquad Y_{i,j}(x)\in M_2(\bar{\Q}[[x]]).
	$$
	Then, by \eqref{eq:thematrices}, $\widetilde{F}(x)$ is equal to
	\begin{equation*}
		\begin{pmatrix}
			A_s(Y_{1,1}(x)A_0+Y_{1,2}(x)B_0)&A_sY_{1,2}(x)C_0\\
			(B_sY_{1,1}(x)+C_sY_{2,1}(x))A_0+(B_sY_{1,2}(x)+C_sY_{2,2}(x))B_0&(B_sY_{1,2}(x)+C_sY_{2,2}(x))C_0
		\end{pmatrix}.
	\end{equation*}
	By construction,
	$$
	\widetilde{F}_{3,4}(x)=F_{2,4}(x),\qquad F_{1,2}(x)=\widetilde{F}_{1,3}(x).
	$$
	%CLEAN
	
	Write $Y_G(x)=(y_{i,j}(x))_{1\leq i,j\leq 4}$, and set $B_s:=(f_{i,j})$, $C_s:=(e_{i,j})$, and $C_0:=(c_{i,j})$. For example,
	\begin{equation}
		\begin{aligned}
			F_{2,4}(x)={}&c_{1,2}\bigl(f_{1,1}y_{1,3}(x)+f_{1,2}y_{2,3}(x)+e_{1,1}y_{3,3}(x)+e_{1,2}y_{4,3}(x)\bigr)\\
			&+c_{2,2}\bigl(f_{1,1}y_{1,4}(x)+f_{1,2}y_{2,4}(x)+e_{1,1}y_{3,4}(x)+e_{1,2}y_{4,4}(x)\bigr).
		\end{aligned}
	\end{equation}
	Thus $R_{s,f}$ is obtained from $F_{1,2}(x)F_{2,4}(x)$ by replacing each $y_{i,j}(x)$ with $X_{i,j}$.
	%CLEAN
	
	By \Cref{corollarynonisogenous}, the matrix $\Theta_w$ is either block diagonal or block antidiagonal. Hence \eqref{eq:e2fprep} gives
	\begin{equation}
		\iota_w((F_{i,j}(x(s))))=
		\frac{1}{\deg(\theta_0)}
		\begin{pmatrix}\Pi_w(E_s)\Phi_1\Pi_w(E_0)^{-1}&0\\0&\Pi_w(E_s')\Phi_2\Pi_w(E_0')^{-1}\end{pmatrix}
	\end{equation}
	or
	\begin{equation}
		\iota_w((F_{i,j}(x(s))))=
		\frac{1}{\deg(\theta_0)}
		\begin{pmatrix}0&\Pi_w(E_s)\Phi_1\Pi_w(E_0')^{-1}\\\Pi_w(E_s')\Phi_2\Pi_w(E_0)^{-1}&0\end{pmatrix}.
	\end{equation}
	In the first case $\iota_w(F_{2,4}(x(s)))=0$, while in the second $\iota_w(F_{1,2}(x(s)))=0$. Since
	$$
	R_{s,f}(Y_G(x))=F_{1,2}(x)F_{2,4}(x),
	$$
	we obtain
	$$
	\iota_w(R_{s,f}(Y_G(x(s))))=0.
	$$
	%CLEAN
	
	To prove that $R_{s,f}\notin I(\SP_4)$, suppose that $R_{s,f}\in I(\SP_4)$. Since $I(\SP_4)$ is prime by \Cref{idealprime}, one of the two linear factors corresponding to $F_{1,2}(x)$ and $F_{2,4}(x)$ must belong to $I(\SP_4)$.
	%CLEAN
	
	Let $R_1$ be the factor corresponding to $F_{2,4}(x)$, and suppose that $R_1\in I(\SP_4)$. For every $n\geq 1$, set
	$$
	T_n:=\begin{pmatrix}n&0\\0&\frac{1}{n}\end{pmatrix},\qquad
	S_n:=\begin{pmatrix}T_n&0\\0&T_n^{-1}\end{pmatrix}\in\SP_4.
	$$
	Then $R_1(S_n)=0$ for every $n\geq 1$, and hence
	$$
	e_{1,2}c_{2,2}n^2+e_{1,1}c_{1,2}=0
	$$
	for every $n\geq 1$. Therefore
	$$
	e_{1,2}c_{2,2}=e_{1,1}c_{1,2}=0.
	$$
	%CLEAN
	
	Similarly, for
	$$
	U_n:=\begin{pmatrix}0&\frac{1}{n}\\-n&0\end{pmatrix},\qquad
	S_n':=\begin{pmatrix}U_n&0\\0&(U_n^T)^{-1}\end{pmatrix}\in\SP_4,
	$$
	the equality $R_1(S_n')=0$ gives
	$$
	c_{1,2}e_{1,2}-n^2c_{2,2}e_{1,1}=0
	$$
	for every $n\geq1$. Hence
	$$
	c_{1,2}e_{1,2}=c_{2,2}e_{1,1}=0.
	$$
	Together with the preceding equalities, and since $(c_{i,j})$ is invertible, this forces $e_{1,1}=e_{1,2}=0$, contradicting the invertibility of $C_s=(e_{i,j})$.
	%CLEAN
	
	The same test-matrix argument shows that the factor corresponding to $F_{1,2}(x)$ does not belong to $I(\SP_4)$. Therefore $R_{s,f}\notin I(\SP_4)$.
	%CLEAN
	\end{proof}
		\subsubsection{Ordinary places}\label{section:ordinary}

At finite places of simultaneous good ordinary reduction of $E_0$ and $E_0'$, the constructions below use bases adapted to the Frobenius action on $H^1_{\crys}$, following \cite{papaspadicgfuns2}. We restrict to the two cases relevant to the Zilber--Pink setting in $\mathcal{A}_2$, beginning with the case where $s_0$ is an $E\times CM$-point.
%CLEAN

In the $E\times CM$ case, we make the choices at the beginning of this section subject to the following additional condition: We take
$$
\Gamma_{dR}(E_0')=(\omega_0',\eta_0')
$$
to consist of eigenvectors for the action of the CM field $\End^0(E_0')$; see \cite[\S 2.1.1]{papaspadicgfuns2}. The induced Hodge basis of $\CX_0'$ is then used in the choice of the basis $\Gamma_{dR}(\CX_0,\theta_0)$, and hence of the Hodge basis of sections $\Gamma_{dR}(\CX/S)$ fixed at the beginning of the section.
%CLEAN

\begin{prop}\label{propordinaryexcm}
	Assume that $s_0$ is an $E\times CM$-point of $f:\CX\rightarrow S$, so that
	$$
	\CX_0\sim E_0\times_K E_0',
	$$
	where $E_0$ is a non-CM elliptic curve and $E_0'$ is a CM elliptic curve.
	
	Let $s\in S(\bar{\Q})$ be such that $\CX_s$ is either an $E\times CM$-abelian surface or an $E^2$-abelian surface. Then there exists a polynomial $R_{s,\simord}\in\bar{\Q}[X_{i,j}:1\leq i,j\leq 4]$ satisfying the following:
	\begin{enumerate}
		\item $R_{s,\simord}$ has coefficients in a finite extension $L_s/K(s)$ with $[L_s:K(s)]$ bounded by an absolute constant $c_f$;
		
		\item $R_{s,\simord}$ is homogeneous of degree at most $2$;
		
		\item for every $w\in\Sigma_{L_s,f}$ such that $s$ is $w$-adically close to $s_0$, $\CX_0$ has good reduction at $w$, $E_0,E_0'$ both have ordinary reduction at $w$, and $\widetilde{E}_{0,w}\sim\widetilde{E}_{0,w}'$, we have
		$$
		\iota_w(R_{s,\simord}(Y_G(x(s))))=0;
		$$
		
		\item $R_{s,\simord}\notin I(\SP_4)$.
	\end{enumerate}
\end{prop}
%CLEAN
\begin{proof}
	The field $L_s/K(s)$ is the one introduced in \Cref{propisoga2fin}. Fix $v\in\Sigma_{L_s,f}$ such that $s$ is $v$-adically close to $s_0$ and:
	\begin{enumerate}
		\item $\CX_0$ has good reduction at $v$;
		\item $\widetilde{E}_{0,v}\sim\widetilde{E}_{0,v}'$; and
		\item $\widetilde{E}_{0,v}$ is ordinary.
	\end{enumerate}
	%CLEAN
	
	Since $\widetilde{E}_{0,v}\times_{\mathbb{F}_{p(v)}}\widetilde{E}_{0,v}'$ is isogenous to $\widetilde{E}_{s,v}\times_{\mathbb{F}_{p(v)}}\widetilde{E}_{s,v}'$, both $\widetilde{E}_{s,v}$ and $\widetilde{E}_{s,v}'$ are ordinary; see \cite[Theorem V.3.1]{silvermanell}. Hence Frobenius-adapted bases may be chosen for all four elliptic curves.
	%CLEAN
	
	The symplectic bases appearing in \Cref{lemisogfin} were arbitrary. We now choose $(\gamma_P,\delta_P)$ and $(\gamma_P',\delta_P')$, for $P\in\{0,s\}$, to be adapted to the Frobenius action on $H_v^1(E_P)$ and $H_v^1(E_P')$, respectively, as in \cite[\S 2.1.1]{papaspadicgfuns2}. Together with the choice of $(\omega_0',\eta_0')$ above, this gives
	\begin{equation}\label{eq:periodcmell}
		\Pi_v(E_0')=\begin{pmatrix}\varpi_0&0\\0&\varpi_0^{-1}\end{pmatrix}
	\end{equation}
	for some $\varpi_0\in\C_v^*$; see \cite[Lemma 2.6]{papaspadicgfuns2}.
	%CLEAN
	
	Recall the morphisms $\varphi_{i,j}$ introduced in the proof of \Cref{lemisogfin2}. Their pullbacks induce morphisms of $\phi$-modules on crystalline cohomology. By the choice of Frobenius-adapted bases, \cite[Lemma 2.4]{papaspadicgfuns2} gives
	\begin{equation}\label{eq:excmperiodfin}
		\Theta_{i,j}=[(\varphi_{i,j})_{\crys}]=\begin{pmatrix}\alpha_{i,j}&0\\0&\beta_{i,j}\end{pmatrix},
	\end{equation}
	where $\alpha_{i,j},\beta_{i,j}\in\C_v$ for $1\leq i,j\leq2$.
	%CLEAN
	
	Write
	$$
	\iota_v\bigl((F_{i,j}(x(s)))\bigr)=:(F_{i,j})=\begin{pmatrix}F_1&F_2\\F_3&F_4\end{pmatrix},
	$$
	where $F_1,\ldots,F_4\in M_2(\C_v)$ are the $2\times2$ blocks of the matrix on the left-hand side of \eqref{eq:e2fmain}.\\
	%CLEAN
	
	\textit{Case 1:} $\CX_s$ is an $E\times CM$-point.\\
	%CLEAN
	
	Choose the Hodge basis of $H^1_{dR}(E_s'/L_s)$ as in \cite[Lemma 2.6]{papaspadicgfuns2}, so that
	$$
	\Pi_v(E_s')=\begin{pmatrix}\varpi_s&0\\0&\varpi_s^{-1}\end{pmatrix}
	$$
	for some $\varpi_s\in\C_v^*$.
	%CLEAN
	
	Combining \eqref{eq:excmperiodfin} and \eqref{eq:e2fmain} gives
	$$
	F_4=\frac{1}{\deg(\theta_0)}
	\begin{pmatrix}\varpi_s\alpha_{2,2}\varpi_0^{-1}&0\\0&\varpi_s^{-1}\beta_{2,2}\varpi_0\end{pmatrix}.
	$$
	Hence
	$$
	F_{3,4}=F_{4,3}=0.
	$$
	%CLEAN
	
	Since $F_{3,4}(x)$ is a linear combination of the entries of $Y_G(x)$, let $R_{s,\simord}$ be the corresponding linear polynomial, obtained by replacing each entry $y_{i,j}(x)$ of $Y_G(x)$ by the variable $X_{i,j}$. Equivalently,
	$$
	R_{s,\simord}(Y_G(x))=F_{3,4}(x).
	$$
	Writing $C_0:=(c_{i,j})$ and $A_s:=(d_{i,j})$, in the notation of \eqref{eq:thematrices}, we have
	$$
	R_{s,\simord}=c_{1,2}d_{2,1}X_{1,3}+c_{2,2}d_{2,1}X_{1,4}+c_{1,2}d_{2,2}X_{2,3}+c_{2,2}d_{2,2}X_{2,4}.
	$$
	Since $F_{3,4}=0$, we obtain
	$$
	\iota_v(R_{s,\simord}(Y_G(x(s))))=0.
	$$
	%CLEAN
	
	It remains to prove that $R_{s,\simord}\notin I(\SP_4)$. Suppose, to the contrary, that $R_{s,\simord}\in I(\SP_4)$. Set
	$$
	N:=\begin{pmatrix}I_2&I_2\\0&I_2\end{pmatrix}\in\SP_4,\qquad \overline{S}_n:=S_nN,\qquad \overline{S}_n':=S_n'N,
	$$
	where $S_n$ and $S_n'$ are the matrices used in the proof of \Cref{propisoga2fin}. Then
	$$
	R_{s,\simord}(\overline{S}_n)=R_{s,\simord}(\overline{S}_n')=0
	$$
	for every $n\geq1$. Arguing as in that proof, these equalities force all coefficients of $R_{s,\simord}$ to vanish, contradicting the invertibility of $A_s$ and $C_0$. Hence $R_{s,\simord}\notin I(\SP_4)$.\\
	%CLEAN

\textit{Case 2:} $\CX_s$ is an $E^2$-point.\\
%CLEAN

Let $\varphi_s:E_s\rightarrow E_s'$ be an isogeny. As in \Cref{propisoga2fin}, we may assume that $\varphi_s$ is defined over $L_s$.
%CLEAN

By functoriality of the Berthelot--Ogus comparison isomorphism, as in \cite[Lemma 3.1]{papaszpy1}, we have
\begin{equation}\label{eq:isogpairsperiods}
	[\varphi_s]_{dR}\cdot\Pi_v(E_s)=\Pi_v(E_s')\cdot[\varphi_s]_v,
\end{equation}
where $[\varphi_s]_{dR}$ denotes the matrix of the pullback induced by $\varphi_s$ with respect to the chosen Hodge bases of $E_s$ and $E_s'$, and $[\varphi_s]_v$ denotes the matrix of the induced map on crystalline cohomology with respect to $\Gamma_v(E_s)$ and $\Gamma_v(E_s')$.
%CLEAN
	
Using \eqref{eq:isogpairsperiods}, we may rewrite \eqref{eq:e2fmain} as
\begin{equation}
	\begin{gathered}
		\iota_v\!\left(\begin{pmatrix}I_2&0\\0&[\varphi_s]_{dR}^{-1}\end{pmatrix}\cdot(F_{i,j}(x(s)))\right)=\\
		=\frac{1}{\deg(\theta_0)}
		\begin{pmatrix}\Pi_v(E_s)&0\\0&\Pi_v(E_s)\end{pmatrix}
		\begin{pmatrix}I_2&0\\0&[\varphi_s]_v^{-1}\end{pmatrix}
		\cdot\Theta_v\cdot\begin{pmatrix}\Pi_v(E_0)^{-1}&0\\0&\Pi_v(E_0')^{-1}\end{pmatrix}.
	\end{gathered}
\end{equation}
%CLEAN

Set
$$
(G_{i,j}(x)):=\begin{pmatrix}I_2&0\\0&[\varphi_s]_{dR}^{-1}\end{pmatrix}\cdot(F_{i,j}(x)),
\qquad
\widetilde{\Theta}_v:=\begin{pmatrix}I_2&0\\0&[\varphi_s]_v^{-1}\end{pmatrix}\cdot\Theta_v.
$$
Then the preceding identity becomes
\begin{equation}\label{eq:excme2main}
	\begin{gathered}
		\iota_v\bigl((G_{i,j}(x(s)))\bigr)=\\
		=\frac{1}{\deg(\theta_0)}
		\begin{pmatrix}\Pi_v(E_s)&0\\0&\Pi_v(E_s)\end{pmatrix}
		\begin{pmatrix}\widetilde{\Theta}_{1,1}&\widetilde{\Theta}_{1,2}\\\widetilde{\Theta}_{2,1}&\widetilde{\Theta}_{2,2}\end{pmatrix}
		\begin{pmatrix}\Pi_v(E_0)^{-1}&0\\0&\Pi_v(E_0')^{-1}\end{pmatrix}.
	\end{gathered}
\end{equation}
By \eqref{eq:excmperiodfin} and the choice of Frobenius-adapted bases, each block has the form
$$
\widetilde{\Theta}_{i,j}=\begin{pmatrix}\alpha_{i,j}&0\\0&\beta_{i,j}\end{pmatrix},
$$
where $\alpha_{i,j},\beta_{i,j}\in\C_v$ for $1\leq i,j\leq2$.
%CLEAN

Write
$$
\iota_v\bigl((G_{i,j}(x(s)))\bigr)=:\begin{pmatrix}G_1&G_2\\G_3&G_4\end{pmatrix}.
$$
Then \eqref{eq:excme2main} gives
$$
G_2=\frac{1}{\deg(\theta_0)}\Pi_v(E_s)\cdot\widetilde{\Theta}_{1,2}\cdot\Pi_v(E_0')^{-1},\qquad
G_4=\frac{1}{\deg(\theta_0)}\Pi_v(E_s)\cdot\widetilde{\Theta}_{2,2}\cdot\Pi_v(E_0')^{-1}.
$$
%CLEAN

Since the chosen bases are symplectic, we have $\det\Pi_v(E_s)=1$; see \cite[Chapter 5]{berthbreenmessing}. Write
$$
\Pi_v(E_s)=(\pi_{i,j})=\begin{pmatrix}\pi_{1,1}&\pi_{1,2}\\\pi_{2,1}&\pi_{2,2}\end{pmatrix}.
$$
Then \eqref{eq:excme2main} and \eqref{eq:periodcmell} give
$$
G_2=\frac{1}{\deg(\theta_0)}\begin{pmatrix}\pi_{1,1}&\pi_{1,2}\\\pi_{2,1}&\pi_{2,2}\end{pmatrix}\cdot
\begin{pmatrix}\alpha_{1,2}\varpi_0^{-1}&0\\0&\beta_{1,2}\varpi_0\end{pmatrix}.
$$
%CLEAN

Since $\det\Pi_v(E_s)=1$, we have
$$
\begin{pmatrix}\pi_{2,2}&-\pi_{1,2}\end{pmatrix}\Pi_v(E_s)=\begin{pmatrix}1&0\end{pmatrix}.
$$
Multiplying the expression for $G_2$ on the left by $\begin{pmatrix}\pi_{2,2}&-\pi_{1,2}\end{pmatrix}$ therefore gives
\begin{equation}\label{eq:g32pi22}
G_{1,4}\pi_{2,2}-G_{2,4}\pi_{1,2}=0.
\end{equation}
%CLEAN

Applying the same argument to $G_4$ gives
$$
G_{3,4}\pi_{2,2}-G_{4,4}\pi_{1,2}=0.
$$
Since $(\pi_{1,2},\pi_{2,2})\neq(0,0)$, the two identities imply
$$
G_{1,4}G_{4,4}-G_{2,4}G_{3,4}=0.
$$
%CLEAN

Define $R_{s,\simord}\in L_s[X_{i,j}:1\leq i,j\leq4]$ by
$$
R_{s,\simord}(Y_G(x))=G_{1,4}(x)G_{4,4}(x)-G_{2,4}(x)G_{3,4}(x).
$$
Then $R_{s,\simord}$ is homogeneous of degree $2$, and hence
$$
\iota_v(R_{s,\simord}(Y_G(x(s))))=0.
$$
It remains to prove that $R_{s,\simord}\notin I(\SP_4)$. Suppose, to the contrary, that $R_{s,\simord}\in I(\SP_4)$.
%CLEAN

The accompanying Mathematica code \cite{mathematicacode} computes, for $R_{s,\simord}$, denoted by ``Rexcme2'' in the code, the remainder upon division by a Gr\"obner basis of $I(\SP_4)$. Write
$$
[\varphi_s]_{dR}^{-1}=\begin{pmatrix}a_s&0\\b_s&c_s\end{pmatrix}.
$$
Since this matrix is invertible, $a_sc_s\neq0$. Likewise, writing $C_0=(c_{i,j})$, $D:=(d_{i,j})$ and $E:=(e_{i,j})$ as in \eqref{eq:thematrices}, we have $\det(C_0)\det(D)\det(E)\neq0$.
%CLEAN

For a monomial $M$ appearing in the remainder, write $c(M)$ for its coefficient. Set
$$
H:=\begin{pmatrix}0&-c_s\\a_s&0\end{pmatrix}.
$$
For $k\in\{3,4\}$, the four coefficients $c(X_{i,k}X_{j+2,k})$, with $1\leq i,j\leq2$, give
$$
\bigl(c(X_{i,k}X_{j+2,k})\bigr)_{1\leq i,j\leq2}=-c_{k-2,2}^{\,2}D^{\mathsf T}HE.
$$
%CLEAN

Under our assumption, these coefficients vanish. Since $C_0$ is invertible, $c_{1,2}$ and $c_{2,2}$ cannot both vanish. Hence
$$
D^{\mathsf T}HE=0.
$$
But
$$
\det(D^{\mathsf T}HE)=a_sc_s\det(D)\det(E)\neq0,
$$
a contradiction. Thus $R_{s,\simord}\notin I(\SP_4)$.
%CLEAN
\end{proof}

\begin{remark}
Let $s\in S(\bar{\Q})$ be an $E\times CM$-point that is $v$-adically close to the $E\times CM$-point $s_0$, and write
$$
F_P:=\End^0_{\bar{\Q}}(E_P')
$$
for the CM field of $E_P'$, where $P\in\{s,0\}$. Suppose that $v$ is a place of simultaneous ordinary reduction at which $\widetilde{E}_{0,v}\sim\widetilde{E}_{0,v}'$, so that \Cref{propisoga2fin} does not apply. Then $F_0=F_s$.

Indeed, in this case the reductions of $E_0'$ and $E_s'$ are geometrically isogenous ordinary elliptic curves, and the equality follows from that of their endomorphism algebras over $\bar{\mathbb{F}}_{p(v)}$. Thus, if $F_0\neq F_s$, the ordinary places of proximity are already covered by \Cref{propisoga2fin}.
%CLEAN
\end{remark}
%CLEAN
		We next consider the case in which the distinguished fiber $\CX_0$ is an $E^2$-abelian surface.
%CLEAN
\begin{prop}\label{propordinarye2}
	Assume that $s_0$ is an $E^2$-point of $f:\CX\rightarrow S$, so that
	$$
	\CX_0\sim E_0\times_K E_0',
	$$
	where $E_0$ and $E_0'$ are isogenous elliptic curves.
	
	Let $s\in S(\bar{\Q})$ be such that $\CX_s$ is either an $E\times CM$-abelian surface or an $E^2$-abelian surface. Then there exists a polynomial $R_{s,\simord}\in\bar{\Q}[X_{i,j}:1\leq i,j\leq4]$ satisfying the following:
	\begin{enumerate}
		\item $R_{s,\simord}$ has coefficients in a finite extension $L_s/K(s)$ with $[L_s:K(s)]$ bounded by an absolute constant $c_f$;
		
		\item $R_{s,\simord}$ is homogeneous of degree at most $4$;
		
		\item for every $w\in\Sigma_{L_s,f}$ such that $s$ is $w$-adically close to $s_0$, $\CX_0$ has good reduction at $w$, and $E_0,E_0'$ both have ordinary reduction at $w$, we have
		$$
		\iota_w(R_{s,\simord}(Y_G(x(s))))=0;
		$$
		
		\item $R_{s,\simord}\notin I(\SP_4)$.
	\end{enumerate}
\end{prop}
	%CLEAN
\begin{proof}
	Let $L_s/K(s)$ be the extension obtained as in the proof of \Cref{propisoga2fin}, using \cite[Theorem 4.2]{silverberg}. Fix $v\in\Sigma_{L_s,f}$ such that $s$ is $v$-adically close to $s_0$, $\CX_0$ has good reduction at $v$, and $E_0,E_0'$ have ordinary reduction at $v$.
	
	As in the proof of \Cref{propordinaryexcm}, all four elliptic curves $E_0,E_0',E_s,E_s'$ have ordinary reduction at $v$. We therefore choose the Frobenius-adapted bases $(\gamma_P,\delta_P)$ and $(\gamma_P',\delta_P')$, for $P\in\{0,s\}$, as in that proof.
	%CLEAN
	
	Let $\varphi_0:E_0\rightarrow E_0'$ be an isogeny. Applying \eqref{eq:isogpairsperiods} to $\varphi_0$, we may rewrite \eqref{eq:e2fmain} as
	\begin{equation}\label{eq:centere2main}
		\begin{gathered}
			\iota_v\!\left((F_{i,j}(x(s)))\begin{pmatrix}I_2&0\\0&[\varphi_0]_{dR}\end{pmatrix}\right)=\\
			=\frac{1}{\deg(\theta_0)}
			\begin{pmatrix}\Pi_v(E_s)&0\\0&\Pi_v(E_s')\end{pmatrix}
			\Theta_v\begin{pmatrix}I_2&0\\0&[\varphi_0]_v\end{pmatrix}
			\begin{pmatrix}\Pi_v(E_0)^{-1}&0\\0&\Pi_v(E_0)^{-1}\end{pmatrix}.
		\end{gathered}
	\end{equation}
	%CLEAN
	
	\textit{Case 1:} $\CX_s$ is an $E^2$-abelian surface.\\
	%CLEAN
	
	Let $\varphi_s:E_s\rightarrow E_s'$ be an isogeny. As in \Cref{propisoga2fin}, we may assume that $\varphi_s$ is defined over $L_s$.
	%CLEAN
	
	Applying \eqref{eq:isogpairsperiods} to $\varphi_s$, set
	$$
	(G_{i,j}(x)):=\begin{pmatrix}I_2&0\\0&[\varphi_s]_{dR}^{-1}\end{pmatrix}\cdot(F_{i,j}(x))\cdot\begin{pmatrix}I_2&0\\0&[\varphi_0]_{dR}\end{pmatrix}
	$$
	and
	$$
	\widetilde{\Theta}_v:=\begin{pmatrix}I_2&0\\0&[\varphi_s]_v^{-1}\end{pmatrix}\cdot\Theta_v\cdot\begin{pmatrix}I_2&0\\0&[\varphi_0]_v\end{pmatrix}.
	$$
	Then \eqref{eq:centere2main} becomes
	\begin{equation}\label{eq:e2e2main}
		\begin{gathered}
			\iota_v((G_{i,j}(x(s))))=\\
			=\frac{1}{\deg(\theta_0)}
			\begin{pmatrix}\Pi_v(E_s)&0\\0&\Pi_v(E_s)\end{pmatrix}
			\begin{pmatrix}\widetilde{\Theta}_{1,1}&\widetilde{\Theta}_{1,2}\\\widetilde{\Theta}_{2,1}&\widetilde{\Theta}_{2,2}\end{pmatrix}
			\begin{pmatrix}\Pi_v(E_0)^{-1}&0\\0&\Pi_v(E_0)^{-1}\end{pmatrix}.
		\end{gathered}
	\end{equation}
	%CLEAN
	
	With respect to the Frobenius-adapted bases, each block has the form
	$$
	\widetilde{\Theta}_{i,j}=\begin{pmatrix}\alpha_{i,j}&0\\0&\beta_{i,j}\end{pmatrix},
	$$
	where $\alpha_{i,j},\beta_{i,j}\in\C_v$ for $1\leq i,j\leq2$.
	%CLEAN
	
	Write
	$$
	\Pi_v(E_0)^{-1}=\begin{pmatrix}\rho_{1,1}&\rho_{1,2}\\\rho_{2,1}&\rho_{2,2}\end{pmatrix},
	\qquad
	\Pi_v(E_s)=\begin{pmatrix}\pi_{1,1}&\pi_{1,2}\\\pi_{2,1}&\pi_{2,2}\end{pmatrix}.
	$$
	By symplecticity, $\det\Pi_v(E_P)=1$ for $P\in\{s,0\}$; see \cite[Chapter 5]{berthbreenmessing}. Set
	$$
	G_{i,j}:=\iota_v(G_{i,j}(x(s))),\qquad
	(G_{i,j})=\begin{pmatrix}G_1&G_2\\G_3&G_4\end{pmatrix},
	$$
	where $G_1,\ldots,G_4\in M_2(\C_v)$ are the corresponding $2\times2$ blocks.
	%CLEAN
	
	From \eqref{eq:e2e2main},
	\begin{equation}\label{eq:e2e2g1}
		G_1=\frac{1}{\deg(\theta_0)}\Pi_v(E_s)\begin{pmatrix}\alpha_{1,1}&0\\0&\beta_{1,1}\end{pmatrix}\Pi_v(E_0)^{-1}.
	\end{equation}
	Since $\det\Pi_v(E_0)^{-1}=1$, we have
	$$
	\Pi_v(E_0)^{-1}\begin{pmatrix}\rho_{2,2}\\-\rho_{2,1}\end{pmatrix}=\begin{pmatrix}1\\0\end{pmatrix},
	\qquad
	\Pi_v(E_0)^{-1}\begin{pmatrix}-\rho_{1,2}\\\rho_{1,1}\end{pmatrix}=\begin{pmatrix}0\\1\end{pmatrix}.
	$$
	Multiplying \eqref{eq:e2e2g1} on the right by $\begin{pmatrix}\rho_{2,2}&-\rho_{2,1}\end{pmatrix}^{T}$ gives
	$$
	\begin{pmatrix}
		G_{1,1}\rho_{2,2}-G_{1,2}\rho_{2,1}\\
		G_{2,1}\rho_{2,2}-G_{2,2}\rho_{2,1}
	\end{pmatrix}
	=
	\frac{1}{\deg(\theta_0)}
	\begin{pmatrix}
		\alpha_{1,1}\pi_{1,1}\\
		\alpha_{1,1}\pi_{2,1}
	\end{pmatrix}.
	$$
	%CLEAN
	
	Similarly, the expression for $G_2$ gives
	$$
	\begin{pmatrix}
		G_{1,3}\rho_{2,2}-G_{1,4}\rho_{2,1}\\G_{2,3}\rho_{2,2}-G_{2,4}\rho_{2,1}\end{pmatrix}=\frac{1}{\deg(\theta_0)}	\begin{pmatrix}	\alpha_{1,2}\pi_{1,1}\\	\alpha_{1,2}\pi_{2,1}
	\end{pmatrix}.
	$$
	The two vectors on the left therefore have determinant $0$, giving
	\begin{equation}\label{e2e2g1g2}
		\begin{aligned}
			0={}&\bigl(G_{1,1}\rho_{2,2}-G_{1,2}\rho_{2,1}\bigr)
			\bigl(G_{2,3}\rho_{2,2}-G_{2,4}\rho_{2,1}\bigr)-\\
			&-\bigl(G_{1,3}\rho_{2,2}-G_{1,4}\rho_{2,1}\bigr)
			\bigl(G_{2,1}\rho_{2,2}-G_{2,2}\rho_{2,1}\bigr).
		\end{aligned}
	\end{equation}
	%CLEAN
	
	Assume first that $\rho_{1,1}\rho_{2,1}\neq0$, and set
	$$
	\rho_1:=\frac{\rho_{2,2}}{\rho_{2,1}},\qquad
	\rho_2:=\frac{\rho_{1,2}}{\rho_{1,1}}.
	$$
	Expanding \eqref{e2e2g1g2}, we may write
	$$
	A_1\rho_{2,2}^2-B_1\rho_{2,2}\rho_{2,1}+C_1\rho_{2,1}^2=0,
	$$
	where $A_1,B_1,C_1$ are homogeneous polynomials of degree $2$ in the $G_{i,j}$. Dividing by $\rho_{2,1}^2$ gives
	$$
	A_1\rho_1^2-B_1\rho_1+C_1=0.
	$$
	Multiplying \eqref{eq:e2e2g1} on the right by $\begin{pmatrix}-\rho_{1,2}\\\rho_{1,1}\end{pmatrix}$ and applying the same argument gives
	$$
	A_1\rho_{1,2}^2-B_1\rho_{1,2}\rho_{1,1}+C_1\rho_{1,1}^2=0.
	$$
	Dividing by $\rho_{1,1}^2$ gives
	$$
	A_1\rho_2^2-B_1\rho_2+C_1=0.
	$$
	%CLEAN

	Applying the same argument to the pair $(G_3,G_2)$ yields homogeneous polynomials $A_2,B_2,C_2$ of degree $2$ in the entries of $G_3$ and $G_2$ satisfying
	$$
	A_2\rho_{2,2}^2-B_2\rho_{2,2}\rho_{2,1}+C_2\rho_{2,1}^2=0,
	$$
	and
	$$
	A_2\rho_{1,2}^2-B_2\rho_{1,2}\rho_{1,1}+C_2\rho_{1,1}^2=0.
	$$
	Hence
	$$
	A_2\rho_j^2-B_2\rho_j+C_2=0,\qquad j=1,2.
	$$
	%CLEAN
	
	Since $(\rho_{i,j})$ is invertible, we have $\rho_1\neq\rho_2$. If one of the polynomials
	$$
	A_iX^2-B_iX+C_i,\qquad i=1,2,
	$$
	vanishes identically, then
	\begin{equation}\label{eq:e2e2ordinaryfinal}
		A_1C_2-A_2C_1=0
	\end{equation}
	holds immediately. Otherwise, these polynomials have the same two distinct roots $\rho_1$ and $\rho_2$, and are therefore proportional. Thus \eqref{eq:e2e2ordinaryfinal} holds in either case.
	%CLEAN
	
	It remains to consider the two cases excluded above. If $\rho_{2,1}=0$, then $\rho_{2,2}\neq0$, and the first pair of undivided identities gives $A_1=A_2=0$. Similarly, if $\rho_{1,1}=0$, then $\rho_{1,2}\neq0$, and the second pair gives $A_1=A_2=0$. Hence \eqref{eq:e2e2ordinaryfinal} also holds in these cases.
	%CLEAN
	
	Since each $G_{i,j}(x)$ is an $L_s$-linear combination of the entries of $Y_G(x)$, there exist power series $A_i(x),C_i(x)$ such that
	$$
	\iota_v(A_i(x(s)))=A_i,\qquad \iota_v(C_i(x(s)))=C_i,
	$$
	and each $A_i(x)$ and $C_i(x)$ is of the form $Q(Y_G(x))$ for some homogeneous polynomial $Q\in L_s[X_{i,j}:1\leq i,j\leq4]$ of degree $2$.
	%CLEAN
	
	Define $R_{s,\simord}\in L_s[X_{i,j}:1\leq i,j\leq4]$ by
	$$
	R_{s,\simord}(Y_G(x))=A_1(x)C_2(x)-A_2(x)C_1(x).
	$$
	Then $R_{s,\simord}$ is homogeneous of degree $4$, and hence
	$$
	\iota_v(R_{s,\simord}(Y_G(x(s))))=0.
	$$
	Finally, $R_{s,\simord}\notin I(\SP_4)$ by \Cref{app:ordinarye2e2}.
	%CLEAN

\textit{Case 2:} $\CX_s$ is an $E\times CM$-abelian surface.\\
%CLEAN

Set
$$
(G_{i,j}(x)):=(F_{i,j}(x))\cdot\begin{pmatrix}I_2&0\\0&[\varphi_0]_{dR}\end{pmatrix},
\qquad
\widetilde{\Theta}_v:=\Theta_v\cdot\begin{pmatrix}I_2&0\\0&[\varphi_0]_v\end{pmatrix}.
$$
Then \eqref{eq:centere2main} becomes
\begin{equation}\label{eq:e2excmeq1}
	\begin{gathered}
		\iota_v\bigl((G_{i,j}(x(s)))\bigr)\\
		=\frac{1}{\deg(\theta_0)}
		\begin{pmatrix}\Pi_v(E_s)&0\\0&\Pi_v(E_s')\end{pmatrix}
		\begin{pmatrix}\widetilde{\Theta}_{1,1}&\widetilde{\Theta}_{1,2}\\\widetilde{\Theta}_{2,1}&\widetilde{\Theta}_{2,2}\end{pmatrix}
		\begin{pmatrix}\Pi_v(E_0)^{-1}&0\\0&\Pi_v(E_0)^{-1}\end{pmatrix}.
	\end{gathered}
\end{equation}
%CLEAN

Since $\widetilde{E}_{s,v}'$ is ordinary, choose the Hodge basis of $H^1_{dR}(E_s'/L_s)$ as in \cite[Lemma 2.6]{papaspadicgfuns2}. Then
$$
\Pi_v(E_s')=\begin{pmatrix}\pi_s&0\\0&\pi_s^{-1}\end{pmatrix}
$$
for some $\pi_s\in\C_v^*$. Write
$$
\iota_v\bigl((G_{i,j}(x(s)))\bigr)=:\begin{pmatrix}G_1&G_2\\G_3&G_4\end{pmatrix}.
$$
%CLEAN

By \eqref{eq:e2excmeq1}, for $j=3,4$ we have
$$
G_j=\begin{pmatrix}\pi_s\alpha_j&0\\0&\pi_s^{-1}\beta_j\end{pmatrix}
\begin{pmatrix}\rho_{1,1}&\rho_{1,2}\\\rho_{2,1}&\rho_{2,2}\end{pmatrix}
$$
for some $\alpha_j,\beta_j\in\C_v$. Multiplying the expressions for $G_3$ and $G_4$ on the right by $\begin{pmatrix}\rho_{2,2}&-\rho_{2,1}\end{pmatrix}^T$ and taking their second rows gives
$$
G_{4,1}\rho_{2,2}-G_{4,2}\rho_{2,1}=0,\qquad
G_{4,3}\rho_{2,2}-G_{4,4}\rho_{2,1}=0.
$$
%CLEAN

Since $(\rho_{2,1},\rho_{2,2})\neq(0,0)$, these identities imply
$$
G_{4,1}G_{4,4}-G_{4,2}G_{4,3}=0.
$$
%CLEAN

Let $R_{s,\simord}\in L_s[X_{i,j}:1\leq i,j\leq4]$ be the corresponding polynomial, defined by
$$
R_{s,\simord}(Y_G(x))=G_{4,1}(x)G_{4,4}(x)-G_{4,2}(x)G_{4,3}(x).
$$
Then $R_{s,\simord}$ is homogeneous of degree $2$, and hence
$$
\iota_v(R_{s,\simord}(Y_G(x(s))))=0.
$$
It remains to prove that $R_{s,\simord}\notin I(\SP_4)$.
%CLEAN

Suppose, to the contrary, that $R_{s,\simord}\in I(\SP_4)$. The polynomial denoted by ``Qe2excm'' in the accompanying Mathematica code \cite{mathematicacode} specializes to $R_{s,\simord}$ upon setting $a_s=c_s=1$ and $b_s=0$. For this polynomial, the code computes the remainder upon division by a Gr\"obner basis of $I(\SP_4)$. For a monomial $\prod X_{i,j}$ appearing in this remainder, write $c\!\left(\prod X_{i,j}\right)$ for its coefficient.
%CLEAN

Since $a_0c_0\neq0$, the coefficient identities
$$
c(X_{4,1}X_{4,3})=c(X_{4,1}X_{4,4})=0
$$
give
$$
c_{1,2}e_{2,2}^2=c_{2,2}e_{2,2}^2=0.
$$
Since $(c_{i,j})$ is invertible, $e_{2,2}=0$. The same argument applied to
$$
c(X_{3,1}X_{3,3})=c(X_{3,1}X_{3,4})=0
$$
gives $e_{2,1}=0$, contradicting the invertibility of $(e_{i,j})$.
%CLEAN
\end{proof}
		\subsubsection{Places of supersingular reduction}

The remaining finite-place case is that in which both $E_0$ and $E_0'$ have supersingular reduction at $v$. Over a finite field, any two supersingular elliptic curves are geometrically isogenous; see, for example, \cite[Lemma 42.1.11]{voight}. Thus \Cref{propisoga2fin} does not apply.
%CLEAN

Unlike the relations constructed at ordinary places, the relations obtained at supersingular places may depend on $v$. We treat separately the cases in which the distinguished fiber $\CX_0$ is an $E\times CM$-abelian surface and an $E^2$-abelian surface. In each case, the construction applies to both $E\times CM$-points and $E^2$-points.
%CLEAN

\begin{prop}\label{propsupersingularexcm}
	Assume that $s_0$ is an $E\times CM$-point of $f:\CX\rightarrow S$, so that
	$$
	\CX_0\sim E_0\times_K E_0',
	$$
	where $E_0$ is a non-CM elliptic curve and $E_0'$ is a CM elliptic curve.
	
	Let $s\in S(\bar{\Q})$ be such that $\CX_s$ is either an $E\times CM$-abelian surface or an $E^2$-abelian surface. Let $w\in\Sigma_{K(s),f}$ be a place for which $s$ is $w$-adically close to $s_0$ and such that $E_0$ and $E_0'$ have supersingular reduction at $w$. Let $L_s/K(s)$ be the extension introduced in \Cref{propisoga2fin}. Then, for every $v\in\Sigma_{L_s,f}$ with $v\mid w$, there exists
	$$
	R_{s,v}\in L_s[X_{i,j}:1\leq i,j\leq4]
	$$
	satisfying the following:
	\begin{enumerate}
		\item $R_{s,v}$ is homogeneous of degree at most $4$;
		\item
		$$
		\iota_v(R_{s,v}(Y_G(x(s))))=0;
		$$
		\item $R_{s,v}\notin I(\SP_4)$.
	\end{enumerate}
\end{prop}
\begin{proof}
	Since $\widetilde{E}_{0,v}$ and $\widetilde{E}_{0,v}'$ are supersingular, the same holds for $\widetilde{E}_{s,v}$ and $\widetilde{E}_{s,v}'$.
	%CLEAN
	
	After replacing $K(s)$ by $L_s$, we may assume that all relevant isogenies are defined over $K(s)$ and that $\End_{\bar{\Q}}^0(E_0')\subset K(s)$. Fix $v\in\Sigma_{K(s),f}$ as in the statement.
	%CLEAN
	
	Since the Frobenius slopes are not distinct in the supersingular case, we cannot choose bases as in the ordinary case. We retain a symplectic basis $(\gamma_0,\delta_0)$ of $H_v^1(E_0)$ and choose $(\gamma_0',\delta_0')$ using Andr\'e's $p$-adic Betti lattice for $E_0'$; see \cite{andremots} and \cite[\S~2.1.2]{papaspadicgfuns2}. With $(\omega_0',\eta_0')$ chosen as in \Cref{propordinaryexcm}, this gives
	\begin{equation}\label{eq:periodcmell2}
		\Pi_v(E_0')=\begin{pmatrix}\varpi_0&0\\0&\varpi_0^{-1}\end{pmatrix}
	\end{equation}
	for some $\varpi_0\in\C_v^*$; see \cite[Lemma 2.6]{papaspadicgfuns2}.\\
	%CLEAN
	
	\textit{Case 1:} $\CX_s$ is an $E\times CM$-abelian surface.\\
	%CLEAN
	
	In the notation of the proof of \Cref{propisoga2fin}, write
	$$
	\iota_v\bigl((F_{i,j}(x(s)))\bigr)=:\begin{pmatrix}F_1&F_2\\F_3&F_4\end{pmatrix},
	$$
	where $F_i\in M_2(\C_v)$ and $\Theta_{i,j}=[\varphi_{i,j}]_v$ for the morphisms introduced in the proof of \Cref{lemisogfin2}. Then \eqref{eq:e2fmain} becomes
	\begin{equation}\label{eq:supersingmain}
		\begin{gathered}
			\begin{pmatrix}F_1&F_2\\F_3&F_4\end{pmatrix}=\\
			=
			\frac{1}{\deg{\theta_0}}\begin{pmatrix}\Pi_v(E_s)&0\\0&\Pi_v(E_s')\end{pmatrix}
			\begin{pmatrix}\Theta_{1,1}&\Theta_{1,2}\\\Theta_{2,1}&\Theta_{2,2}\end{pmatrix}
			\begin{pmatrix}\Pi_v(E_0)^{-1}&0\\0&\Pi_v(E_0')^{-1}\end{pmatrix}.
		\end{gathered}
	\end{equation}
	%CLEAN
	
	If $\varphi_{i,j}=0$ for some $1\leq i,j\leq2$, then the corresponding block $\Theta_{i,j}$ in \eqref{eq:supersingmain} vanishes. As in the proof of \Cref{propisoga2fin}, this yields a non-trivial linear relation among the entries of $Y_G(x(s))$. We may therefore assume that all four morphisms $\varphi_{i,j}$ are nonzero, and hence are isogenies.
	%CLEAN
	
	Consider the endomorphism
	$$
	\alpha:=\varphi_{1,2}^{\vee}\circ\varphi_{1,1}\circ\varphi_{2,1}^{\vee}\circ\varphi_{2,2}
	\in\End(\widetilde{E}_{0,v}').
	$$
	By the construction of Andr\'e's $p$-adic Betti lattice,
	$$
	[\alpha]_v\in\GL_2(K(s)).
	$$
	%CLEAN
	
	Since, for any isogeny $\varphi$, we have $\varphi\circ\varphi^\vee=[\deg(\varphi)]$, it follows that
	$$
	[\varphi^\vee]_v=\deg(\varphi)[\varphi]_v^{-1}.
	$$
	Together with the proof of \Cref{lemisogfin2}, this gives
	$$
	[\alpha]_v=
	\deg(\varphi_{1,2})\deg(\varphi_{2,1})
	\Theta_{1,2}^{-1}\Theta_{1,1}\Theta_{2,1}^{-1}\Theta_{2,2}.
	$$
	%CLEAN
	
	Combining this with \eqref{eq:supersingmain} gives
	\begin{equation}\label{eq:supersingularexcmexcm1}
		F_2^{-1}F_1F_3^{-1}F_4=\Pi_v(E_0')\frac{1}{\deg(\varphi_{1,2})\deg(\varphi_{2,1})}[\alpha]_v\Pi_v(E_0')^{-1}.
	\end{equation}
	
	Write
	$$
	\begin{pmatrix}a&b\\c&d\end{pmatrix}:=\frac{1}{\deg(\varphi_{1,2})\deg(\varphi_{2,1})}[\alpha]_v\in\GL_2(K(s)).
	$$
	By \eqref{eq:periodcmell2}, the right-hand side of \eqref{eq:supersingularexcmexcm1} is
	$$
	\begin{pmatrix}
		a&b\varpi_0^2\\c\varpi_0^{-2}&d
	\end{pmatrix}.
	$$
	
	Since $F_j^{-1}=\det(F_j)^{-1}F_j^{\adj}$, equation \eqref{eq:supersingularexcmexcm1} becomes
	\begin{equation}\label{eq:supersingularexcmexcm2}
		F_2^{\adj}F_1F_3^{\adj}F_4=
		\begin{pmatrix}
			a\det(F_2F_3)&b\varpi_0^2\det(F_2F_3)\\
			c\varpi_0^{-2}\det(F_2F_3)&d\det(F_2F_3)
		\end{pmatrix}.
	\end{equation}
	
	Let $P_{1,1}$ denote the top-left entry of the matrix on the left-hand side of \eqref{eq:supersingularexcmexcm2}, and let $R_{s,v}$ be the polynomial corresponding to the relation
	$$
	P_{1,1}-a\det(F_2F_3)=0.
	$$
	By construction, $R_{s,v}$ is homogeneous of degree $4$, has coefficients in $K(s)$, and satisfies $\iota_v(R_{s,v}(Y_G(x(s))))=0$. It remains to show that $R_{s,v}\notin I(\SP_4)$. The required coefficient check, using the remainder upon division by a Gr\"obner basis of $I(\SP_4)$ computed by the accompanying Mathematica code, is given in \Cref{section:app-supersingular-nontriviality}.\\
	%CLEAN
	
	\textit{Case 2:} $\CX_s$ is an $E^2$-abelian surface.\\
	%CLEAN
	
	Let $\varphi_s:E_s\rightarrow E_s'$ be the isogeny introduced in Case~2 of the proof of \Cref{propordinaryexcm}, and let $\widetilde{\varphi}_{s,v}$ denote its reduction at $v$. With the notation of that proof, define $(G_{i,j}(x))$ and $\widetilde{\Theta}_v$ as there. This recovers \eqref{eq:excme2main}.
	%CLEAN
	
	For $j\in\{1,2\}$, set
	$$
	\widetilde{\varphi}_{1,j}:=\varphi_{1,j},
	\qquad
	\widetilde{\varphi}_{2,j}:=\widetilde{\varphi}_{s,v}^{\vee}\circ\varphi_{2,j}.
	$$
	Then
	$$
	\widetilde{\Theta}_{1,j}=[\widetilde{\varphi}_{1,j}]_v,
	\qquad
	\widetilde{\Theta}_{2,j}=\frac{1}{\deg(\varphi_s)}[\widetilde{\varphi}_{2,j}]_v.
	$$
	%CLEAN
	
	Consider the endomorphism
	$$
	\beta:=\widetilde{\varphi}_{1,2}^{\vee}\circ\widetilde{\varphi}_{2,2}\in\End(\widetilde{E}_{0,v}').
	$$
	As in the previous case, the construction of Andr\'e's $p$-adic Betti lattice gives
	$$
	[\beta]_v\in\GL_2(K(s)).
	$$
	Moreover,
	$$
	[\beta]_v=\deg(\widetilde{\varphi}_{1,2})\deg(\varphi_s)	\widetilde{\Theta}_{1,2}^{-1}\widetilde{\Theta}_{2,2}.
	$$
	%CLEAN
	
	As in the previous case, set
	$$
	\begin{pmatrix}a&b\\c&d\end{pmatrix}:=\frac{1}{\deg(\widetilde{\varphi}_{1,2})\deg(\varphi_s)}[\beta]_v\in\GL_2(K(s)).
	$$
	Then \eqref{eq:excme2main} and \eqref{eq:periodcmell2} give
	$$
	G_2^{-1}G_4=\begin{pmatrix}a&b\varpi_0^2\\c\varpi_0^{-2}&d\end{pmatrix},
	$$
	and hence
	$$
	G_2^{\adj}G_4=\begin{pmatrix}a\det(G_2)&b\varpi_0^2\det(G_2)\\c\varpi_0^{-2}\det(G_2)&d\det(G_2)\end{pmatrix}.
	$$
	%CLEAN
	
	Let $P_{1,1}$ denote the top-left entry of $G_2^{\adj}G_4$, and let $R_{s,v}$ be the polynomial corresponding to the relation
	$$
	P_{1,1}-a\det(G_2)=0.
	$$
	Since the entries of $(G_{i,j}(x))$ are $K(s)$-linear combinations of the entries of $(F_{i,j}(x))$, the polynomial $R_{s,v}$ has coefficients in $K(s)$. By construction, it is homogeneous of degree $2$ and satisfies $\iota_v(R_{s,v}(Y_G(x(s))))=0$. It remains to show that $R_{s,v}\notin I(\SP_4)$.
	%CLEAN
	
	Suppose, to the contrary, that $R_{s,v}\in I(\SP_4)$. We retain the notation from the non-triviality check in Case~2 of the proof of \Cref{propordinaryexcm}. Write $C:=(c_{i,j})$, $D:=(d_{i,j})$, and $E:=(e_{i,j})$. Then
	$$
	a_sc_s\det(C)\det(D)\det(E)\neq0.
	$$
	
	The accompanying Mathematica code \cite{mathematicacode} computes the remainder of the division of $R_{s,v}$, denoted by ``Rss2'' in the code, by a Gr\"obner basis of $I(\SP_4)$. Writing $c(\prod X_{i,j})$ for the coefficient of a monomial $\prod X_{i,j}$ in this remainder, the coefficients $c(X_{1,2}X_{4,3})$, $c(X_{1,2}X_{4,4})$, $c(X_{1,3}X_{3,2})$, $c(X_{1,4}X_{3,2})$, $c(X_{1,3}X_{4,2})$, and $c(X_{1,4}X_{4,2})$ give
	\begin{equation}\label{eq:nontriviality-supersingular-case2-excm1}
		\begin{aligned}
			c_{1,2}(d_{1,1}a-a_sd_{2,1})e_{1,2}
			&=c_{2,2}(d_{1,1}a-a_sd_{2,1})e_{1,2}=0,\\
			c_{1,2}(a_sd_{2,1}e_{1,1}-c_sd_{1,1}e_{2,1})
			&=c_{2,2}(a_sd_{2,1}e_{1,1}-c_sd_{1,1}e_{2,1})=0,\\
			c_{1,2}d_{1,1}(ae_{1,2}-c_se_{2,2})
			&=c_{2,2}d_{1,1}(ae_{1,2}-c_se_{2,2})=0.
		\end{aligned}
	\end{equation}
	%CLEAN
	
	Since $C$ is invertible, $c_{1,2}$ and $c_{2,2}$ cannot both vanish. The three pairs in \eqref{eq:nontriviality-supersingular-case2-excm1} therefore give
	\begin{gather*}
		(d_{1,1}a-a_sd_{2,1})e_{1,2}=0,\\
		a_sd_{2,1}e_{1,1}-c_sd_{1,1}e_{2,1}=0,\\
		d_{1,1}(ae_{1,2}-c_se_{2,2})=0.
	\end{gather*}
	Eliminating $d_{1,1}ae_{1,2}$ from the first and third identities gives
	$$
	a_sd_{2,1}e_{1,2}-c_sd_{1,1}e_{2,2}=0.
	$$
	Together with the second identity, this yields
	$$
	(a_sd_{2,1},-c_sd_{1,1})E=(0,0).
	$$
	Since $E$ is invertible and $a_sc_s\neq0$, we obtain $d_{1,1}=d_{2,1}=0$, contradicting $\det(D)\neq0$. Hence $R_{s,v}\notin I(\SP_4)$.
	%CLEAN
\end{proof}
		
\begin{prop}\label{propsupersingulare2}
	Assume that $s_0$ is an $E^2$-point of $f:\CX\rightarrow S$, so that
	$$
	\CX_0\sim E_0\times_K E_0',
	$$
	where $E_0$ and $E_0'$ are isogenous elliptic curves.
	
	Let $s\in S(\bar{\Q})$ be such that $\CX_s$ is either an $E\times CM$-abelian surface or an $E^2$-abelian surface. Let $w\in\Sigma_{K(s),f}$ be a place for which $s$ is $w$-adically close to $s_0$ and such that $E_0$ and $E_0'$ have supersingular reduction at $w$. Let $L_s/K(s)$ be the extension introduced in \Cref{propisoga2fin}. Then, for every $v\in\Sigma_{L_s,f}$ with $v\mid w$, there exists
	$$
	R_{s,v}\in L_s[X_{i,j}:1\leq i,j\leq4]
	$$
	satisfying the following:
	\begin{enumerate}
		\item $R_{s,v}$ is homogeneous of degree at most $2$;
		\item
		$$
		\iota_v(R_{s,v}(Y_G(x(s))))=0;
		$$
		\item $R_{s,v}\notin I(\SP_4)$.
	\end{enumerate}
\end{prop}

\begin{proof}
	As in the proof of \Cref{propsupersingularexcm}, after replacing $K(s)$ by $L_s$, we may assume that all relevant isogenies are defined over $K(s)$ and, when $s$ is an $E\times CM$-point, that $\End_{\bar{\Q}}^0(E_s')\subset K(s)$. Fix $v\in\Sigma_{K(s),f}$ as in the statement.
	%CLEAN
	
	Let $\varphi_0:E_0\rightarrow E_0'$ be the isogeny introduced in the proof of \Cref{propordinarye2}, and let $\widetilde{\varphi}_{0,v}:\widetilde{E}_{0,v}\rightarrow\widetilde{E}_{0,v}'$ denote its reduction at $v$.
	
	\textit{Case 1:} $\CX_s$ is an $E^2$-abelian surface.\\
	%CLEAN
	
	Let $\varphi_s:E_s\rightarrow E_s'$ be the isogeny introduced in Case~1 of the proof of \Cref{propordinarye2}, and let $\widetilde{\varphi}_{s,v}$ denote its reduction at $v$. With the notation of that proof, define $(G_{i,j}(x))$ and $\widetilde{\Theta}_v$ as there. This recovers \eqref{eq:e2e2main}. Write
	$$
	\widetilde{\Theta}_v=\begin{pmatrix}\widetilde{\Theta}_{1,1}&\widetilde{\Theta}_{1,2}\\	\widetilde{\Theta}_{2,1}&\widetilde{\Theta}_{2,2}
	\end{pmatrix},
	$$
	where $\widetilde{\Theta}_{i,j}\in M_2(\C_v)$.
	%CLEAN
	
	Write
	$$
	\iota_v\bigl((G_{i,j}(x(s)))\bigr)=:\begin{pmatrix}G_1&G_2\\G_3&G_4\end{pmatrix},
	$$
	where $G_1,\ldots,G_4\in M_2(\C_v)$ are the corresponding $2\times2$ blocks.
	%CLEAN
	
	We retain the notation $\varphi_{i,j}$ for the morphisms introduced in the proof of \Cref{lemisogfin2}. Set
	$$
	\widetilde{\varphi}_{1,1}:=\varphi_{1,1},
	\qquad
	\widetilde{\varphi}_{1,2}:=\varphi_{1,2}\circ\widetilde{\varphi}_{0,v},
	$$
	and
	$$
	\widetilde{\varphi}_{2,1}:=\widetilde{\varphi}_{s,v}^{\vee}\circ\varphi_{2,1},
	\qquad
	\widetilde{\varphi}_{2,2}:=\widetilde{\varphi}_{s,v}^{\vee}\circ\varphi_{2,2}\circ\widetilde{\varphi}_{0,v}.
	$$
	Then
	$$
	\widetilde{\Theta}_{1,j}=[\widetilde{\varphi}_{1,j}]_v,
	\qquad
	\widetilde{\Theta}_{2,j}
	=\frac{1}{\deg(\varphi_s)}[\widetilde{\varphi}_{2,j}]_v,
	\qquad j\in\{1,2\}.
	$$
	%CLEAN
	
	If $\widetilde{\varphi}_{i,j}=0$ for some $1\leq i,j\leq2$, then the corresponding block $\widetilde{\Theta}_{i,j}$ in \eqref{eq:e2e2main} vanishes. As in the proof of \Cref{propisoga2fin}, this yields a non-trivial linear relation among the entries of $Y_G(x(s))$. We may therefore assume that all four morphisms $\widetilde{\varphi}_{i,j}$ are nonzero, and hence are isogenies.
	%CLEAN
	
	Consider the endomorphism
	$$
	\alpha:=\widetilde{\varphi}_{2,2}\circ\widetilde{\varphi}_{1,2}^{\vee}
	\in\End(\widetilde{E}_{s,v}).
	$$
	In the above notation,
	$$
	[\alpha]_v=\deg(\widetilde{\varphi}_{1,2})\deg(\varphi_s)
	\widetilde{\Theta}_{2,2}\widetilde{\Theta}_{1,2}^{-1}.
	$$
	%CLEAN
	
	From \eqref{eq:e2e2main},
	\begin{equation}\label{eq:e2e2g2}
		G_2=\frac{1}{\deg(\theta_0)}\Pi_v(E_s)\widetilde{\Theta}_{1,2}\Pi_v(E_0)^{-1},
	\end{equation}
	and
	\begin{equation}\label{eq:e2e2g4}
		G_4=\frac{1}{\deg(\theta_0)}\Pi_v(E_s)\widetilde{\Theta}_{2,2}\Pi_v(E_0)^{-1}.
	\end{equation}
	%CLEAN
	
	Set	$m:=\deg(\widetilde{\varphi}_{1,2})\deg(\varphi_s)$. Using the expressions for $G_2$ and $G_4$ above, we obtain
	\begin{equation}\label{eq:e2e2supersing1}
		G_4G_2^{\adj}
		=
		\det(G_2)\Pi_v(E_s)\frac{1}{m}[\alpha]_v\Pi_v(E_s)^{-1}.
	\end{equation}
	%CLEAN

	Since $\alpha$ is an endomorphism of the elliptic curve $\widetilde{E}_{s,v}$, we have
	$$
	\tr([\alpha]_v),\det([\alpha]_v)\in\Z;
	$$
	see \cite[Chapter~V, \S~2, Corollary]{demazure}.
	%CLEAN
	
	Taking traces in \eqref{eq:e2e2supersing1} gives
	\begin{equation}\label{eq:e2e2supersingtrace}
		\tr(G_4G_2^{\adj})=\frac{\tr([\alpha]_v)}{m}\det(G_2).
	\end{equation}
	Similarly, taking determinants gives
	\begin{equation}\label{eq:e2e2supersingdet}
		\det(G_4)=\frac{\det([\alpha]_v)}{m^2}\det(G_2).
	\end{equation}
	%CLEAN
	
	Let $R_{s,v,1}$ and $R_{s,v,2}$ be the polynomials corresponding to \eqref{eq:e2e2supersingtrace} and \eqref{eq:e2e2supersingdet}, respectively. By construction, both have coefficients in $K(s)$, are homogeneous of degree $2$, and satisfy the required vanishing condition. The non-triviality check in \Cref{section:app-supersingular-e2-nontriviality} shows that at least one of them does not belong to $I(\SP_4)$.\\
	%CLEAN
	
	\textit{Case 2:} $\CX_s$ is an $E\times CM$-abelian surface.\\
	%CLEAN
	
	With the notation of Case~2 in the proof of \Cref{propordinarye2}, the same argument recovers \eqref{eq:e2excmeq1}.
	%CLEAN
	
	Since $E_s'$ is a CM elliptic curve with supersingular reduction at $v$, we choose $(\gamma_s',\delta_s')$ using Andr\'e's $p$-adic Betti lattice, as in \cite{andremots}; see also \cite[\S~2.1.2]{papaspadicgfuns2}. With $(\omega_s',\eta_s')$ chosen as in \cite[Lemma 2.6]{papaspadicgfuns2}, this gives
	$$
	\Pi_v(E_s')=\begin{pmatrix}\varpi_s&0\\0&\varpi_s^{-1}\end{pmatrix}
	$$
	for some $\varpi_s\in\C_v^*$.
	%CLEAN
	
	Write
	$$
	\iota_v\bigl((G_{i,j}(x(s)))\bigr)=:\begin{pmatrix}G_1&G_2\\G_3&G_4\end{pmatrix}.
	$$
	%CLEAN
	
	We retain the notation $\varphi_{i,j}$ for the morphisms of elliptic curves introduced in the proof of \Cref{lemisogfin2}. Write
	$$
	\widetilde{\Theta}_v=\begin{pmatrix}\widetilde{\Theta}_{1,1}&\widetilde{\Theta}_{1,2}\\	\widetilde{\Theta}_{2,1}&\widetilde{\Theta}_{2,2}
	\end{pmatrix}.
	$$
	For $i\in\{1,2\}$, set
	$$
	\widetilde{\varphi}_{i,1}:=\varphi_{i,1},\qquad	\widetilde{\varphi}_{i,2}:=\varphi_{i,2}\circ\widetilde{\varphi}_{0,v}.
	$$
	Then
	$$
	\widetilde{\Theta}_{i,j}=[\widetilde{\varphi}_{i,j}]_v,	\qquad 1\leq i,j\leq2.
	$$
	%CLEAN
	
	As above, if $\widetilde{\varphi}_{i,j}=0$ for some $1\leq i,j\leq2$, then $\widetilde{\Theta}_{i,j}=0$ yields a non-trivial linear relation; hence we may assume that all four $\widetilde{\varphi}_{i,j}$ are isogenies.
	%CLEAN
	
	Consider the endomorphism
	$$
	\beta:=\widetilde{\varphi}_{2,2}\circ\widetilde{\varphi}_{2,1}^{\vee}\in\End(\widetilde{E}_{s,v}').
	$$
	By the construction of Andr\'e's $p$-adic Betti lattice, $[\beta]_v\in\GL_2(K(s))$. Moreover,
	$$
	[\beta]_v=	\deg(\widetilde{\varphi}_{2,1})\widetilde{\Theta}_{2,2}\widetilde{\Theta}_{2,1}^{-1}.
	$$
	%CLEAN
	
	As in the proof of \Cref{propsupersingularexcm}, write
	$$
	\begin{pmatrix}a&b\\c&d\end{pmatrix}:=\frac{1}{\deg(\widetilde{\varphi}_{2,1})}[\beta]_v\in\GL_2(K(s)).
	$$
	Arguing as in that proof, \eqref{eq:e2excmeq1} and the above description of $\Pi_v(E_s')$ give
	$$
	G_4G_3^{\adj}=	\det(G_3) \begin{pmatrix}a&b\varpi_s^2\\c\varpi_s^{-2}&d\end{pmatrix}.
	$$
	%CLEAN
	
	Let $P_{1,1}$ denote the top-left entry of $G_4G_3^{\adj}$, and let $R_{s,v}$ be the polynomial corresponding to the relation
	$$
	P_{1,1}-a\det(G_3)=0.
	$$
	Since the entries of $(G_{i,j}(x))$ are $K(s)$-linear combinations of the entries of $(F_{i,j}(x))$, the polynomial $R_{s,v}$ has coefficients in $K(s)$. By construction, it is homogeneous of degree $2$ and satisfies $\iota_v(R_{s,v}(Y_G(x(s))))=0$. It remains to show that $R_{s,v}\notin I(\SP_4)$.
	%CLEAN
	
	Suppose, to the contrary, that $R_{s,v}\in I(\SP_4)$. We retain the notation from the non-triviality check in Case~2 of the proof of \Cref{propordinarye2}. Write $C:=(c_{i,j})$, $D:=(d_{i,j})$, and $E:=(e_{i,j})$. Then
	$$
	a_0c_0\det(C)\det(D)\det(E)\neq0.
	$$

	The accompanying Mathematica code \cite{mathematicacode} computes the remainder of the division of $R_{s,v}$, denoted by ``Rss3'' in the code, by a Gröbner basis of $I(\SP_4)$. Writing $c(\prod X_{i,j})$ for the coefficient of a monomial $\prod X_{i,j}$ in this remainder, the coefficients $c(X_{1,3}X_{3,1})$, $c(X_{1,3}X_{4,1})$, $c(X_{2,3}X_{3,1})$, and $c(X_{2,3}X_{4,1})$ give
	$$
	c_{1,2}d_{2,1}e_{2,1}
	=c_{1,2}d_{2,1}e_{2,2}
	=c_{1,2}d_{2,2}e_{2,1}
	=c_{1,2}d_{2,2}e_{2,2}=0.
	$$
	Since the second rows of $D$ and $E$ are nonzero, we obtain $c_{1,2}=0$, and hence $c_{1,1}\neq0$.
	%CLEAN
	
	The coefficients $c(X_{1,3}X_{3,2})$, $c(X_{1,2}X_{4,3})$, $c(X_{2,2}X_{3,3})$, and $c(X_{2,2}X_{4,3})$ now give
	\begin{gather*}
		d_{2,1}e_{2,1}=d_{2,1}e_{2,2}=0,\\
		d_{2,2}e_{2,1}=0,\qquad
		d_{2,1}e_{2,1}-d_{2,2}e_{2,2}=0.
	\end{gather*}
	Since the second row of $E$ is nonzero, the first pair gives $d_{2,1}=0$. Hence $\det(D)\neq0$ gives $d_{2,2}\neq0$, and the remaining two identities imply $e_{2,1}=e_{2,2}=0$, contradicting $\det(E)\neq0$. Thus $R_{s,v}\notin I(\SP_4)$.
	%CLEAN
\end{proof}
		\subsection{Archimedean relations}\label{section:archrels}

We return to the general splitting setting of \Cref{section:splittingrelations}, without imposing the additional hypotheses used in \Cref{section:splitfinite}. Fix an archimedean place $v\in\Sigma_{K,\infty}$.
%CLEAN

\begin{prop}\label{archrelexcm}
	Let $s\in S(\bar{\Q})$ be such that
	$$
	\CX_s\sim E_s\times_{\bar{\Q}}E_s',
	$$
	where $E_s$ and $E_s'$ are elliptic curves. Let $w\in\Sigma_{K(s),\infty}$ be a place above $v$, and assume that $s$ is $w$-adically close to $s_0$.
	
	Then there exist a finite extension $L_s/K(s)$ and a polynomial
	$$
	R_{s,w}\in L_s[X_{i,j}:1\leq i,j\leq4]
	$$
	satisfying:
	\begin{enumerate}
		\item $[L_s:K(s)]$ is bounded by an absolute constant;
		\item for every $w'\in\Sigma_{L_s,\infty}$ with $w'\mid w$,
		$$
		\iota_{w'}(R_{s,w}(Y_G(x(s))))=0;
		$$
		\item $R_{s,w}$ is homogeneous of degree at most $c_\infty$, where $c_\infty>0$ is an absolute constant;
		\item $R_{s,w}\notin I(\SP_4)$.
	\end{enumerate}
\end{prop}
%CLEAN

\begin{proof}
	Let $L_s/K(s)$ be the extension introduced in \Cref{propisoga2fin}, and fix an isogeny
	$$
	\theta_s:\CX_s\longrightarrow\CX_s'
	$$
	defined over $L_s$.
	%CLEAN
	
	Fix $w'\in\Sigma_{L_s,\infty}$ with $w'\mid w$. After base change to $L_s$, choose Hodge bases $\Gamma_{dR}(E_s)$ and $\Gamma_{dR}(E_s')$, together with symplectic bases of $H^1_{w'}(E_s)$ and $H^1_{w'}(E_s')$, for the de Rham--Betti comparison associated with $w'$. To simplify notation, write $\Pi_v(E_s)$ and $\Pi_v(E_s')$ for the resulting period matrices, suppressing their dependence on $w'$.
	%CLEAN
	
	Set
	$$
	(F_{i,j}(x)):=J_{2,3}[\theta_s]_{dR}Y_G(x)[\theta_0]_{dR}^{-1}J_{2,3},
	$$
	as in the proof of \Cref{propisoga2fin}. By \Cref{lemisogfin},
	\begin{equation}\label{eq:mainarchrel}
		\begin{gathered}
			\iota_{w'}\bigl((F_{i,j}(x(s)))\bigr)=\\
			=\frac{1}{\deg(\theta_0)}
			\begin{pmatrix}\Pi_v(E_s)&0\\0&\Pi_v(E_s')\end{pmatrix}
			\Theta_{w'}
			\begin{pmatrix}\Pi_v(E_0)^{-1}&0\\0&\Pi_v(E_0')^{-1}\end{pmatrix},
		\end{gathered}
	\end{equation}
	where $\Theta_{w'}\in\GL_4(\Q)$.
	%CLEAN
	
	Write
	$$
	\Theta_{w'}=\begin{pmatrix}\Theta_{1,1}&\Theta_{1,2}\\\Theta_{2,1}&\Theta_{2,2}\end{pmatrix},
	\qquad
	(F_{i,j}(x))=:\begin{pmatrix}F_1(x)&F_2(x)\\F_3(x)&F_4(x)\end{pmatrix},
	$$
	where $\Theta_{i,j}\in M_2(\Q)$ and $F_i(x)\in M_2(\bar{\Q}[[x]])$. Thus \eqref{eq:mainarchrel} becomes
	\begin{equation}
		\begin{gathered}
			\iota_{w'}\bigl((F_{i,j}(x(s)))\bigr)=\\
			=\frac{1}{\deg(\theta_0)}
			\begin{pmatrix}
				\Pi_v(E_s)\Theta_{1,1}\Pi_v(E_0)^{-1}&\Pi_v(E_s)\Theta_{1,2}\Pi_v(E_0')^{-1}\\
				\Pi_v(E_s')\Theta_{2,1}\Pi_v(E_0)^{-1}&\Pi_v(E_s')\Theta_{2,2}\Pi_v(E_0')^{-1}
			\end{pmatrix}.
		\end{gathered}
	\end{equation}
	%CLEAN
	
	By the Legendre relation, see \cite[Chapter 18, \S 1]{langellfuns}, we have\footnote{Our normalization differs from that of \cite{langellfuns} by the factor $\frac{1}{2\pi i}I_{2g}$, and hence gives the inverse determinant.}
	$$
	\det(\Pi_v(E_P))=\frac{1}{2\pi i},\qquad P\in\{s,0\}.
	$$
	Taking determinants in the upper-left block of the preceding identity gives
	\begin{equation}\label{eq:archrelsplit0}
		\iota_{w'}\!\left(\det(F_1(x(s)))\right)=\frac{\det(\Theta_{1,1})}{\deg(\theta_0)^2}\in\Q.
	\end{equation}
	%CLEAN
	
	Let $R_0\in L_s[X_{i,j}:1\leq i,j\leq4]$ be defined by
	$$
	R_0(Y_G(x))=\det(F_1(x)),
	$$
	and set
	$$
	d_{w'}:=\frac{\det(\Theta_{1,1})}{\deg(\theta_0)^2}\in\Q.
	$$
	Define
	\begin{equation}\label{eq:archrelsplit}
		R_{s,w'}:=R_0-d_{w'}\bigl(X_{1,1}X_{3,3}+X_{2,1}X_{4,3}-X_{2,3}X_{4,1}-X_{1,3}X_{3,1}\bigr).
	\end{equation}
	%CLEAN
	
	By construction, $R_{s,w'}$ is homogeneous of degree $2$. Moreover, the polynomial
	$$
	f_1(X_{i,j}):=1-\bigl(X_{1,1}X_{3,3}+X_{2,1}X_{4,3}-X_{2,3}X_{4,1}-X_{1,3}X_{3,1}\bigr)
	$$
	lies in $I(\SP_4)$. Hence \Cref{trivialrelations} gives
	$$
	\iota_{w'}(f_1(Y_G(x(s))))=0.
	$$
	Combining this with \eqref{eq:archrelsplit0} gives
	$$
	\iota_{w'}(R_{s,w'}(Y_G(x(s))))=0.
	$$
	%CLEAN
	
	Finally, define
	\begin{equation}\label{eq:excmarchfin}
		R_{s,w}:=\prod_{w'\mid w}R_{s,w'}.
	\end{equation}
	By construction, $R_{s,w}$ has coefficients in $L_s$ and satisfies
	$$
	\iota_{w'}(R_{s,w}(Y_G(x(s))))=0
	$$
	for every $w'\mid w$. It is homogeneous, and
	$$
	\deg(R_{s,w})=2\cdot\#\{w'\in\Sigma_{L_s,\infty}:w'\mid w\}\leq2[L_s:K(s)]\leq8\cdot36^{16}.
	$$
	Thus we may take
	$$
	c_\infty:=8\cdot36^{16}.
	$$
	
	Finally, $R_{s,w}\notin I(\SP_4)$ by \Cref{app:arch-nontriviality}.
\end{proof}
	%CLEAN
		
	%\part{height bounds and applications}
	\section{Height bounds and applications}\label{section:htboundsgoodred}

In this section we prove the height bound of \Cref{goodreductionmainhtbound} and apply it to obtain large Galois orbit estimates, and hence Zilber--Pink finiteness results, under a bound on the number of supersingular primes of proximity.
%CLEAN

\subsection{Proof of the height bounds}\label{section:pfofhtboundsgood}

For $\xi\in\{\xi_1,\ldots,\xi_l\}$, write $\Sigma_{\ssing}(\CX_\xi)$ for the set of finite places of $K$ at which $\CX_\xi$ has good supersingular reduction. By \Cref{reduction}, this set is independent of the choice of $\xi$. For $s\in S(\bar{\Q})$, using the notion of $v$-adic proximity from \Cref{section:vadicproximity}, set
$$
\begin{gathered}
	\Sigma(s,0):=\{v\in\Sigma_{K(s)}:s\text{ is }v\text{-adically close to }0\},\\
	\Sigma_{K,\ssing}(s,0):=\{w\in\Sigma_{\ssing}(\CX_\xi):\exists v\in\Sigma(s,0),\ v\mid w\}.
\end{gathered}
$$
%CLEAN

\begin{prop}\label{htboundsfin}
	Assume that, for every $\xi\in\{\xi_1,\ldots,\xi_l\}$, the fiber $\CX_\xi$ is an $E\times CM$-point (resp. an $E^2$-point) with everywhere potentially good reduction. Set
	$$
	\Sha_{\spli}(S):=\{s\in S(\bar{\Q}):\CX_s\text{ is an }E\times CM\text{- or }E^2\text{-surface}\}.
	$$
	Then there exist effectively computable constants $c_1,c_2>0$ such that, for every $s\in\Sha_{\spli}(S)$,
	$$
	h(s)\leq c_1\Bigl([K(s):\Q]\bigl(1+|\Sigma_{K,\ssing}(s,0)|\bigr)\Bigr)^{c_2}.
	$$
\end{prop}
%CLEAN

\begin{proof}
	Fix $s\in\Sha_{\spli}(S)$. If $x(s)=0$, then $s\in\{\xi_1,\ldots,\xi_l\}$, so its height is bounded independently of $s$. For the remainder of the proof assume that $x(s)\neq 0$.
	
	Let $L_s/K(s)$ be the extension introduced in \Cref{propisoga2fin}. Set
	$$
	\Sigma(s):=\{v\in\Sigma_{L_s}:s\text{ is }v\text{-adically close to }0\}.
	$$
	
	If $\Sigma(s)=\emptyset$, the proof of \cite[Theorem~1.3]{papasbigboi} gives a bound $h(s)\leq c$ for some constant $c>0$ independent of $s$. We may therefore assume that $\Sigma(s)\neq\emptyset$.
	%CLEAN

	For $\lambda\in\Lambda$, let $\xi_\lambda$ and $\CX_\lambda\rightarrow S$ be as in \Cref{section:settinggfuns}. Fix $v\in\Sigma(s)$.
	%CLEAN
	
	For every finite place $v\in\Sigma(s)$, the fiber $\CX_s$ has good reduction at $v$. Indeed, by \Cref{section:vadicproximity}, its reduction agrees with that of some distinguished fiber $\CX_{\xi_t}$. By \Cref{reduction}, $\CX_{\xi_t}$ has semistable reduction at $v$, while by hypothesis it has potentially good reduction. Hence $\CX_{\xi_t}$, and therefore $\CX_s$, has good reduction at $v$.
	%CLEAN
	
	The proof of \cite[Proposition~4.1]{papaseffbrsieg} gives $t=t(v)\in\{1,\ldots,l\}$ and a representative $\lambda\in\Lambda$ with $t\sim\lambda$ such that, setting $s_t:=\sigma_t^{-1}(s)$,
	$$
	s_t\in\Delta_{v,1}\qquad\text{and}\qquad \CX_{\lambda,s_t}\text{ splits}.
	$$
	Here $\Delta_{v,1}$ denotes the $v$-adic disc centered at $s_0:=\xi_1$ of radius $r_v(\mathcal{Y}):=\min\{1,R_v(\mathcal{Y})\}$.
	%CLEAN
	
	Set
	$$
	\Sigma(s)_\infty:=\{v\in\Sigma(s):v\mid\infty\},\qquad
	\Sigma(s)_f:=\Sigma(s)\setminus\Sigma(s)_\infty,
	$$
	and
	$$
	\begin{gathered}
		\Sigma(s)_{\ssing}:=\{v\in\Sigma(s)_f:v\mid w\text{ for some }w\in\Sigma_{\ssing}(\CX_\xi)\},\\
		\Sigma(s)_{\nssing}:=\Sigma(s)_f\setminus\Sigma(s)_{\ssing}.
	\end{gathered}
	$$
	%CLEAN
	
	Suppose first that $v\in\Sigma(s)_{\nssing}$. Let $R_{s,t,\nssing}$ be the product of the polynomial furnished by \Cref{propisoga2fin} and the polynomial furnished by \Cref{propordinaryexcm} or \Cref{propordinarye2}, according as $\CX_{\xi_t}$ is of $E\times CM$- or $E^2$-type. Since $v$ is not supersingular, either the two elliptic factors of $\CX_{\xi_t}$ are not geometrically isogenous modulo $v$, or they are simultaneously ordinary. Hence
	$$
	\iota_v\bigl(R_{s,t,\nssing}(Y_{G,\lambda}(x(s)))\bigr)=0.
	$$
	For fixed $t$, the polynomial $R_{s,t,\nssing}$ is independent of $v$, homogeneous of degree bounded by an absolute constant, and
	$$
	R_{s,t,\nssing}\notin I(\SP_4),
	$$
	the last assertion following from the primeness of $I(\SP_4)$ and the corresponding non-triviality statements in \Cref{propisoga2fin}, \Cref{propordinaryexcm}, and \Cref{propordinarye2}.
	%CLEAN
	
	If $v\in\Sigma(s)_{\ssing}$, let $R_{s,t,v}$ be the polynomial furnished by \Cref{propsupersingularexcm} or \Cref{propsupersingulare2}, according as $\CX_{\xi_t}$ is of $E\times CM$- or $E^2$-type. If $v\in\Sigma(s)_\infty$, let $R_{s,t,v}$ be the polynomial furnished by \Cref{archrelexcm}. In either case,
	$$
	R_{s,t,v}\in\bar{\Q}[X_{i,j}^{(\lambda)}:1\leq i,j\leq4]\setminus I(\SP_4),
	\qquad
	\iota_v\bigl(R_{s,t,v}(Y_{G,\lambda}(x(s)))\bigr)=0,
	$$
	where $\lambda\in\Lambda$ is the representative with $t\sim\lambda$, and $\deg R_{s,t,v}$ is bounded by an absolute constant.
	%CLEAN
	
	For $\lambda\in\Lambda$, set
	$$
	T_{s,\lambda,\nssing}:=\{t(v):v\in\Sigma(s)_{\nssing},\ t(v)\sim\lambda\}.
	$$
	Define
	$$
	R_{s,\lambda,\infty}:=\prod_{\substack{v\in\Sigma(s)_\infty\\ t(v)\sim\lambda}}R_{s,t(v),v},
	\qquad
	R_{s,\lambda,\nssing}:=\prod_{t\in T_{s,\lambda,\nssing}}R_{s,t,\nssing}.
	$$
	Then
	$$
	\deg R_{s,\lambda,\infty}\leq c_\infty[L_s:\Q],
	\qquad
	\deg R_{s,\lambda,\nssing}\leq 8l.
	$$
	%CLEAN
	
	Finally, for $\lambda\in\Lambda$, define
	$$
	R_{s,\lambda,\ssing}:=\prod_{\substack{v\in\Sigma(s)_{\ssing}\\ t(v)\sim\lambda}}R_{s,t(v),v},
	\qquad
	R_{s,\ssing}:=\prod_{\lambda\in\Lambda}R_{s,\lambda,\ssing}.
	$$
	Then
	$$
	\deg R_{s,\ssing}\leq 4|\Sigma(s)_{\ssing}|.
	$$
	%CLEAN
	
	Set
	$$
	R_{s,f}:=R_{s,\ssing}\prod_{\lambda\in\Lambda}R_{s,\lambda,\nssing},
	\qquad
	R_{s,\infty}:=\prod_{\lambda\in\Lambda}R_{s,\lambda,\infty},
	\qquad
	R_s:=R_{s,f}R_{s,\infty}.
	$$
	By construction, $R_s$ gives a global relation at $s$. Since $[L_s:K(s)]$ is bounded independently of $s$ and
	$$
	|\Sigma(s)_{\ssing}|\leq [L_s:K]\,|\Sigma_{K,\ssing}(s,0)|,
	$$
	the preceding degree bounds give
	$$
	\deg R_s\leq c_0[K(s):\Q]\bigl(1+|\Sigma_{K,\ssing}(s,0)|\bigr)
	$$
	for some constant $c_0>0$ independent of $s$.
	%CLEAN
	
	Since each local factor lies outside $I(\SP_4)$, \Cref{idealprime} and the argument of Step~4 in the proof of \cite[Proposition~4.1]{papaseffbrsieg} give
	$$
	R_s\notin I_\Lambda.
	$$
	%CLEAN
	
	Applying the Andr\'e--Bombieri Hasse principle \cite[Ch.~VII, \S5]{andre1989g} to $R_s$ gives the asserted height bound.
	\end{proof}
	%CLEAN
	\subsection{Applications to unlikely intersections}\label{section:applicationsgood}

We now discuss applications of \Cref{goodreductionmainhtbound} to the Zilber--Pink conjecture, following the work of Daw and Orr \cite{daworr,daworr2}.
%CLEAN

Let $Z\subset\mathcal{A}_2$ be a curve. For $s\in Z(\bar{\Q})$, let $A_s$ denote the corresponding abelian surface. We call $s$ an $E\times CM$-point if
$$
A_s\sim E_s\times_{\bar{\Q}}E_s',
$$
with exactly one of $E_s$ and $E_s'$ CM, and an $E^2$-point if
$$
A_s\sim E_s\times_{\bar{\Q}}E_s',\qquad E_s\sim E_s',
$$
with $E_s$ and $E_s'$ non-CM. Thus points of either type are non-special. The case of special points belongs instead to the Andr\'e--Oort setting, which is known for $\mathcal A_2$ by the work of Pila--Tsimerman \cite{PilaTsimerman2013}.
%CLEAN

For an $E\times CM$- or $E^2$-point $s\in Z(\bar{\Q})$, let $V_s\subset\mathcal{A}_2$ denote the unique special curve containing $s$; see \cite{daworr}. Following Daw--Ren \cite{dawren}, write $\Delta(V)$ for the complexity of a special subvariety $V$. For $E\times CM$-curves and $E^2$-curves, we use the conventions of \cite[\S3]{daworr} and \cite[\S6.3]{daworr2}, respectively.
%CLEAN

Assume now that $Z$ is smooth and irreducible, and let $s_0\in Z(\bar{\Q})$. Fix a number field $K$ over which $Z$ and $s_0$ are defined. After replacing $K$ by a finite extension if necessary, fix a cover $(S,\{\xi_1,\ldots,\xi_l\})$ of $(Z,s_0)$ as in \Cref{section:htboundsred}. For each $s\in Z(\bar{\Q})$, fix a lift $s'\in S(\bar{\Q})$ of $s$, and use the notation $\Sigma(s,0):=\Sigma(s',0)$ and $\Sigma_{K,\ssing}(s,0):=\Sigma_{K,\ssing}(s',0)$ introduced in \Cref{section:pfofhtboundsgood}.
%CLEAN
\begin{prop}\label{lgogood}
	Let $Z\subset\mathcal{A}_2$ be a smooth irreducible curve defined over $\bar{\Q}$ that is not contained in a proper special subvariety of $\mathcal{A}_2$, and fix $N\in\N$. Set
	$$
	\Sha_{ZP\text{-}\spli,N}(Z):=\{s\in Z(\C):s\text{ is an }E\times CM\text{- or }E^2\text{-point and }|\Sigma_{K,\ssing}(s,0)|\leq N\}.
	$$
	Assume that $Z$ contains an $E\times CM$- or $E^2$-point $s_0\in Z(\bar{\Q})$ such that
	$$
	A_{s_0}\sim E_0\times_{\bar{\Q}}E_0',
	$$
	where $E_0$ and $E_0'$ have everywhere potentially good reduction. 
	
	There exist effectively computable constants
	$$
	c_1=c_1(Z,N,s_0)>0,\qquad c_2=c_2(Z,s_0)>0
	$$
	such that, for every $s\in\Sha_{ZP\text{-}\spli,N}(Z)$,
	\begin{equation}\label{eq:lgo}
		|\gal(\bar{\Q}/\Q)\cdot s|\geq c_1\Delta(V_s)^{c_2}.
	\end{equation}
\end{prop}
%CLEAN
\begin{proof}
	The argument combines \Cref{goodreductionmainhtbound} with the Masser--W\"ustholz isogeny estimates \cite{masserwu,mwendoesti}. For $E\times CM$-points, the proof of \cite[Proposition~9.2]{daworr} applies with the height bound there replaced by \Cref{goodreductionmainhtbound}. For $E^2$-points, the same applies to the proof of \cite[Theorem~6.5]{daworr2}.
\end{proof}
%CLEAN
\begin{proof}[Proof of \Cref{goodreductionmainzpapp}]
	Combining \Cref{lgogood} with \cite[Theorem~1.2]{daworr} gives the required finiteness for $E\times CM$-points, while the $E^2$-case follows from \Cref{lgogood} and \cite[Theorem~1.3]{daworr2}.
\end{proof}
%CLEAN

	\subsection{Supersingular primes of proximity}\label{section:supersingularprox}

The height bound in \Cref{goodreductionmainhtbound} depends on $|\Sigma_{K,\ssing}(s,0)|$. For every $s\in S(\bar{\Q})$ with $x(s)\neq 0$, this set is finite, since
$$
\Sigma_{K,\ssing}(s,0)\subseteq\{w\in\Sigma_{K,f}:\exists v\in\Sigma_{K(s),f},\ v\mid w,\ |x(s)|_v<1\}.
$$
The problem is therefore to bound its cardinality in terms of the arithmetic complexity of $s$. The following conjecture would suffice to establish the corresponding same-type Zilber--Pink statement.
%CLEAN
\begin{conj}\label{conjonsupersingularproximity}
	Let $f:\CX\rightarrow S$ be a one-parameter family defined over a number field $K$, and let $s_0\in S(K)$ be an $E^2$-point (respectively an $E\times CM$-point). For every $E^2$-point (respectively $E\times CM$-point) $s\in S(\bar{\Q})$ with $x(s)\neq 0$, let $V_s\subset\mathcal A_2$ be the corresponding special curve.
	
	Then there exist constants $c_1=c_1(f,s_0)>0$ and $c_2=c_2(f,s_0)>0$ such that
	\begin{equation}\label{eq:conjsupsingproxim}
		|\Sigma_{K,\ssing}(s,0)|\leq c_1\bigl([K(s):K]\log\Delta(V_s)\bigr)^{c_2}.
	\end{equation}
\end{conj}
%CLEAN

Combined with the Masser--W\"ustholz isogeny estimates, \Cref{conjonsupersingularproximity} gives the large Galois orbit estimates required in the arguments of Daw--Orr recalled in \Cref{section:applicationsgood}, and hence the corresponding Zilber--Pink finiteness statement.
%CLEAN

For comparison, define
$$
\Sigma_{\geomisog}(\CX_\xi):=\{v\in\Sigma_{K,f}:E_\xi\text{ and }E_\xi'\text{ are geometrically isogenous modulo }v\},
$$
where $\CX_\xi\sim E_\xi\times_K E_\xi'$. We have
$$
\Sigma_{\ssing}(\CX_\xi)\subseteq\Sigma_{\geomisog}(\CX_\xi).
$$
If $\CX_\xi$ is an $E^2$-point, then $\Sigma_{\geomisog}(\CX_\xi)$ is infinite. If $\CX_\xi$ is an $E\times CM$-point, the same conclusion follows from \cite[Theorem~1.1]{charlesisog}. Thus the proximity condition in $\Sigma_{K,\ssing}(s,0)$ is essential to the quantitative problem above. When $\CX_\xi$ is an $E^2$-point, the Lang--Trotter conjecture predicts that $\Sigma_{\ssing}(\CX_\xi)$ itself is infinite; compare also \cite{elkies}.
%CLEAN

\subsection{Splittings in $\mathcal{A}_g$}\label{section:splittingsinag}

The constructions of \Cref{section:relsa2} suggest an analogous problem for splittings in $\mathcal{A}_g$. Let $f:\CX\rightarrow S$ be a one-parameter family of $g$-dimensional principally polarized abelian varieties defined over a number field $K$, and assume that the induced morphism $j:S\rightarrow\mathcal{A}_g$ has Hodge generic image. Suppose that, for some $s_0\in S(K)$,
$$
\CX_{s_0}\sim A_0\times B_0,
$$
where $A_0$ and $B_0$ have positive dimension.
%CLEAN

For $s\in S(\bar{\Q})$, suppose that
$$
\CX_s\sim A_s\times B_s,
\qquad
\dim A_s=\dim A_0,\qquad \dim B_s=\dim B_0.
$$
The problem is to construct relations among the corresponding values of the associated G-functions reflecting the splittings $\CX_{s_0}\sim A_0\times B_0$ and $\CX_s\sim A_s\times B_s$.
%CLEAN

The constructions of \Cref{section:relsa2} suggest that, in this setting, one should obtain a polynomial $R_{s,v}$ satisfying analogues of the conclusions of \Cref{naivelocalrelations}. At finite places, one may further expect the dependence of $R_{s,v}$ on $v$ to be controlled, away from supersingular reduction of $\CX_{s_0}$, by discrete reduction data such as the possible Newton polygons of the associated $p$-divisible groups.
%CLEAN

A separate difficulty is to establish that the resulting relations are non-trivial, a problem that already requires substantial explicit computation in the surface case.
%CLEAN

More generally, the underlying methods are not restricted to loci corresponding to splittings. The existence of additional endomorphisms of the fibers can be exploited to construct algebraic relations among values of $G$-functions, as illustrated by the case of quaternionic multiplication in $\mathcal A_2$ treated in \cite{papaspadicgfuns3}. This suggests analogous constructions for other endomorphism algebras defining special subvarieties in $\mathcal A_g$.
%CLEAN
	
	%\part{Bad reduction}
	\section{Values of G-functions at places of bad reduction} \label{section:badred}

At finite places of good reduction of the central fiber $\CX_0$, the values of the G-functions constructed in \Cref{section:gfunctions} admit an interpretation in terms of crystalline periods. It is natural to expect an analogous interpretation at places of bad semistable reduction, with Hyodo--Kato cohomology replacing crystalline cohomology. In this section we formulate two conjectural comparison statements that would provide such an interpretation and show that, assuming them, the resulting $(\phi,N)$-module structure yields relations among the corresponding $v$-adic values of the G-functions. As an application, these relations extend the height bounds of the previous section beyond the case in which $\CX_0$ has everywhere potentially good reduction.

\subsection{Hyodo-Kato cohomology} \label{section:logcryst}

Let $K/\mathbb{Q}_p$ be a finite extension and let $X/K$ be a $g$-dimensional abelian variety with semistable reduction. Let $k=\mathbb{F}_q$ be the residue field of $\CO_K$, set $W:=W(k)$ and $K_0:=W[\frac{1}{p}]$, write $\sigma$ for the Frobenius on $K_0$, and fix a uniformizer $\varpi\in\CO_K$.

We call a proper log smooth morphism
$$
\mathfrak{X}\longrightarrow \spec(\CO_K)
$$
a semistable model of $X$ if $\mathfrak{X}$ is an fs log scheme with generic fiber $X$, where $\spec(\CO_K)$ is endowed with the log structure induced by
$$
\mathbb{N}\longrightarrow \CO_K,\qquad n\longmapsto \varpi^n.
$$

In \cite{hyodokato}, Hyodo and Kato construct a $W$-lattice, which we denote by $H^1_{HK}(\widetilde{\mathfrak{X}}/W)$. At places of bad semistable reduction, this will replace crystalline cohomology as the cohomological object used to interpret the corresponding $v$-adic periods.

There is a canonical isomorphism of $K$-vector spaces
\begin{equation}\label{eq:hyodokatoisom}
	\rho_{HK}(X):H^1_{dR}(X/K)\longrightarrow H^1_{HK}(\widetilde{\mathfrak{X}}/W)\otimes_W K.
\end{equation}
Unlike the de Rham--crystalline comparison isomorphism of \cite{bertogus}, the isomorphism $\rho_{HK}(X)$ depends on the choice of the uniformizer $\varpi$.\footnote{See Remark $4.4.18$ in \cite{tsujipadicetale}.}

The Hyodo--Kato cohomology also carries a $(\phi,N)$-module structure. More precisely, set
$$
D(X):=H^1_{HK}(\widetilde{\mathfrak{X}}/W)\otimes_W K_0.
$$
Then $D(X)$ is equipped with a bijective $\sigma$-semilinear Frobenius operator
$$
\phi:D(X)\longrightarrow D(X)
$$
and a nilpotent $K_0$-linear monodromy operator $N\in\End(D(X))$ satisfying
$$
N\phi=p\phi N.
$$

The relation constructions below require functoriality of Hyodo--Kato cohomology under isogenies, compatible with both the $(\phi,N)$-module structure and the Hyodo--Kato comparison isomorphism. Motivated by Theorem B of \cite{vonk}, we formulate this as follows.

\begin{conj}\label{conjhyodokato}
	Let $f:X\rightarrow X'$ be an isogeny between abelian varieties with semistable reduction over $K$, and let $\mathfrak{X}$ and $\mathfrak{X}'$ be semistable models of $X$ and $X'$, respectively. Then $f$ induces a canonical pullback
	$$
	f^*_{HK}:H^1_{HK}(\widetilde{\mathfrak{X}'}/W)\longrightarrow H^1_{HK}(\widetilde{\mathfrak{X}}/W)
	$$
	which is a morphism of $(\phi,N)$-modules. Moreover, the diagram
	\begin{center}
		\begin{tikzcd}
			H^1_{dR}(X/K) \arrow[r, "\rho_{HK}(X)"] & H^1_{HK}(\widetilde{\mathfrak{X}}/W)\otimes_W K \\
			H^1_{dR}(X'/K) \arrow[u, "f^*_{dR}"'] \arrow[r, "\rho_{HK}(X')"] & H^1_{HK}(\widetilde{\mathfrak{X}'}/W)\otimes_W K \arrow[u, "f^*_{HK}"']
		\end{tikzcd}
	\end{center}
	commutes, where $f^*_{dR}$ denotes the pullback induced by $f$ on de Rham cohomology.
\end{conj}

	\subsection{$G$-functions and bad reduction} \label{section:gfunsbad}

Let $K$ be a number field and let $f:\CX \rightarrow S$ be a one--parameter family of principally polarized $g$-dimensional abelian varieties as in \Cref{section:backgroundgfuns}. Fix $s_0\in S(K)$ such that the fiber $\CX_0:=\CX_{s_0}$ has everywhere semistable reduction.

For a place $v\in\Sigma_{K,f}$ at which $\CX_0$ has bad semistable reduction, define
$$
H^1_v(\CX_0):=H^1_{HK}(\widetilde{\mathfrak{X}}_{0,v}/W(k_v))\otimes_{W(k_v)}K_{v,0},
$$
where $\mathfrak{X}_{0,v}\rightarrow\spec(\CO_{K_v})$ is a semistable model of $\CX_0$ in the sense of \Cref{section:logcryst}.

At places of bad semistable reduction, we expect the values of the $G$-functions associated to the pair $(f:\CX\rightarrow S,s_0)$ in \Cref{gfuns} to admit a period interpretation analogous to that at places of good reduction, with Hyodo--Kato cohomology replacing crystalline cohomology. Such an interpretation requires a relative Hyodo--Kato comparison isomorphism, analogous to the relative comparison isomorphisms of \cite{bertogus,ogus}. We formulate the required statement as follows.

\begin{conj}\label{conjrelhyodokato}
	Let $K/\Q_p$ be a finite extension and let $f:\CX\rightarrow S$ be a one--parameter family of abelian varieties defined over $K$ and satisfying the properties of \Cref{reduction}. Let $s_0\in S(K)$ be a point such that the fiber $\CX_0:=\CX_{s_0}$ has semistable reduction, and let $\mathfrak{X}_0$ be a semistable model of $\CX_0$ over $\CO_K$.
	
	Then there exists a sufficiently small $p$-adic analytic disc $\Delta\hookrightarrow S^{an}$ centered at $s_0$ and a canonical isomorphism
	\begin{equation}
		H^1_{dR}(\CX/S)\otimes_{\CO_S}\CO_{\Delta}
		\longrightarrow
		\left(H^1_{HK}(\widetilde{\mathfrak{X}}_0/W)\otimes_W K\right)\otimes_K\CO_{\Delta},
	\end{equation}
	which is horizontal with respect to the Gauss--Manin connection on the left and the trivial connection on the right.
	
	Let $\overline K$ be an algebraic closure of $K$, let $\overline k$ be its residue field, and set $\breve K_0:=\operatorname{Frac}W(\overline k)$. For every finite extension $L/K$ contained in $\overline K$ and every $s\in\Delta(L)$, the fiber $\CX_s$ has semistable reduction, and its Hyodo--Kato cohomology is canonically identified, after extension of scalars to $\breve K_0$, as a $(\phi,N)$-module, with that of $\CX_0$. Under this identification, the specialization at $s$ of the preceding relative isomorphism agrees with the Hyodo--Kato comparison isomorphism \eqref{eq:hyodokatoisom}, after extension of scalars to $\C_p$.
\end{conj}

Under the identification of \Cref{conjrelhyodokato}, for $s\in\Delta(\overline K)$ defined over a finite extension $L/K$ with residue field $k_L$, we write
$$
H^1_{HK}(\widetilde{\mathfrak{X}}_s/W(k_L))\otimes_{W(k_L)}\breve K_0
=
H^1_{HK}(\widetilde{\mathfrak{X}}_0/W)\otimes_W\breve K_0,
$$
where $\mathfrak{X}_s$ is a semistable model of $\CX_s$ over $\CO_L$.
	
\begin{remark}\label{remarkconjrelhyodokato}
	In \cite[(1.7)]{hyodokato}, Hyodo--Kato note that the $(\phi,N)$-module $H^1_{HK}(\widetilde{\mathfrak{X}}_0/W)$ depends only on the reduction
	$$
	\mathfrak{X}_0\otimes_{\CO_K}\CO_K/\mathfrak{m}_K^2.
	$$
		
	This suggests that the disc $\Delta$ in \Cref{conjrelhyodokato} should be chosen so that fibers parametrized by points of $\Delta$ have the same reduction data modulo $\mathfrak{m}_K^2$ as $\CX_0$.
\end{remark}

\subsubsection{$v$-adic proximity at places of bad reduction}\label{section:vadicproximitybad}

At places of bad semistable reduction, the comparison in \Cref{conjrelhyodokato} requires a stronger proximity condition than equality of the special fibers. In view of \Cref{remarkconjrelhyodokato}, for a finite place $v\in\Sigma_K$ at which a distinguished fiber $\CX_\xi$ has bad semistable reduction, we require sufficiently $v$-adically close points to agree with $\xi$ modulo the square of the maximal ideal. More precisely, if $s\in S(\bar{\Q})$ and $w$ is a place above $v$ of a number field over which $s$ is defined, we require the sections associated with $s$ and $\xi$ to have the same image in
$$
\mathfrak{S}'(\CO_w/\mathfrak{m}_w^2).
$$

The set of finite places at which one of the distinguished fibers $\CX_{\xi_1},\ld,\CX_{\xi_l}$ has bad reduction is finite. We may therefore decrease, at these places, the constants $\kappa_v$ appearing in \Cref{lemgfunsnew} so that the $v$-adic proximity condition of \Cref{section:vadicproximity} implies the preceding congruence modulo $\mathfrak{m}_w^2$. We make this additional requirement in the construction of the associated family of $G$-functions.

The modification of the Hodge basis in \Cref{section:vadicproximity} is unchanged: the entries of the resulting matrix of $G$-functions are obtained by multiplying the two halves of the Hodge basis by $H(x)$ and $H(x)^{-1}$, respectively. At the places of bad reduction, this additional proximity requirement is achieved by further restricting the choice of $\zeta$, without changing the form of the modification. We continue to denote the resulting matrix of $G$-functions by $Y_G(x)$. In particular, the description of functional relations of \Cref{trivialrelations} still holds.

For a place $v$ of bad semistable reduction of $\CX_\xi$, let $\Delta_v(\xi)$ denote the corresponding $v$-adic disc on which \Cref{conjrelhyodokato} applies. We choose the preceding proximity constants so that $w$-adic proximity to $\xi$ implies membership in $\Delta_v(\xi)$. Hence, throughout this section, if $s$ is $w$-adically close to $\xi$ at a place $w\mid v$ of bad reduction, then $s$ and $\xi$ agree modulo $\mathfrak{m}_w^2$ and the relative Hyodo--Kato comparison of \Cref{conjrelhyodokato} applies.

\subsubsection{$G$-functions and periods}

Let $\Gamma_{dR}(\CX)$ denote the Hodge basis obtained after the modifications of the preceding construction, normalized so that its specialization at $s_0$ agrees with the Hodge basis fixed in \Cref{gfuns}. For each $v\in\Sigma_{K,f}$ at which $\CX_0$ has bad semistable reduction, fix a basis $\Gamma_v(\CX_0)$ of $H^1_v(\CX_0)$. For $s\in\Delta_v(s_0)$, the comparison isomorphism of \Cref{conjrelhyodokato} determines a period matrix, which we denote by $\CP_v(s)$.

Assuming \Cref{conjrelhyodokato}, the argument of \Cref{gfuns} gives the corresponding period interpretation at places of bad semistable reduction.

\begin{prop}\label{gfunsbadprop}
	Let $f:\CX\rightarrow S$ and $v\in\Sigma_{K,f}$ be as above, and assume that \Cref{conjrelhyodokato} holds. Then, for every $s\in\Delta_v(s_0)$,
	$$
	\CP_v(s)=\iota_v\bigl(Y_G(x(s))\bigr)\cdot\CP_v(s_0),
	$$
	where $Y_G(x)$ denotes the matrix of $G$-functions obtained from the preceding construction.
\end{prop}

\subsection{Relations among values of $G$-functions}\label{section:relsbad}

Throughout this subsection, assume that there exists an isogeny
$$
\theta_0:\CX_0\longrightarrow \CX_0':=E_0\times_K E_0',
$$
where $E_0$ and $E_0'$ are elliptic curves.

\begin{lemma}\label{lemrelsbad}
	Let $f:\CX\rightarrow S$ be as in \Cref{section:gfunsbad}, and let $v\in\Sigma_{K,f}$ be a place at which $\CX_0$ has bad semistable reduction. Assume that \Cref{conjhyodokato} and \Cref{conjrelhyodokato} hold.
	
	Let $s\in S(K)$ be such that there exists an isogeny
	$$
	\theta_s:\CX_s\longrightarrow \CX_s':=E_s\times_K E_s',
	$$
	where $E_s$ and $E_s'$ are elliptic curves, and assume that $s\in\Delta_v(s_0)$. Then the identity \eqref{eq:e2fmain}, with the factor $1/\deg(\theta_0)$ absorbed into $\Theta_v$, holds for some $\Theta_v\in\GL_4(\C_v)$.
	
	Moreover, there exist $\psi_{i,j}\in\C_v$ such that:
	\begin{itemize}
		\item if $E_0'$ and $E_s'$ are CM elliptic curves, then
		\begin{equation}\label{eq:badthetacm}
			\Theta_v=
			\begin{pmatrix}
				\psi_{1,1}&0&\psi_{1,3}&\psi_{1,4}\\
				\psi_{2,1}&\psi_{2,2}&\psi_{2,3}&\psi_{2,4}\\
				\psi_{3,1}&0&\psi_{3,3}&\psi_{3,4}\\
				\psi_{4,1}&0&\psi_{4,3}&\psi_{4,4}
			\end{pmatrix};
		\end{equation}
		\item if there exist isogenies $\phi_0:E_0\rightarrow E_0'$ and $\phi_s:E_s\rightarrow E_s'$, then
		\begin{equation}\label{eq:badthetaisog}
			\Theta_v=\begin{pmatrix}
				\psi_{1,1}&0&\psi_{1,3}&0\\
				\psi_{2,1}&\psi_{1,1}&\psi_{2,3}&\psi_{1,3}\\
				\psi_{3,1}&0&\psi_{3,3}&0\\
				\psi_{4,1}&\psi_{3,1}&\psi_{4,3}&\psi_{3,3}
			\end{pmatrix}.
		\end{equation}
	\end{itemize}
\end{lemma}
	\begin{proof}
	Since $\CX_0$ has semistable reduction at $v$ and $\CX_0'$ is isogenous to $\CX_0$, the same holds for $\CX_0'$, and hence for its elliptic factors $E_0$ and $E_0'$. Likewise, since $s\in\Delta_v(s_0)$, the fibers $\CX_s$ and $\CX_s'=E_s\times_K E_s'$, and therefore $E_s$ and $E_s'$, have semistable reduction at $v$.
	
	By \Cref{gfunsbadprop} and \Cref{conjhyodokato}, the argument of \Cref{lemisogfin} applies with Hyodo--Kato cohomology in place of crystalline cohomology. The remaining point is to choose the bases $\Gamma_v(E_P)$ and $\Gamma_v(E_P')$, for $P\in\{0,s\}$, compatibly with the monodromy operators. We distinguish the two cases in the statement.
	
	Assume first that $E_0'$ and $E_s'$ are CM elliptic curves. Since both have semistable reduction at $v$, they have good reduction by \cite{serretate}. In particular, their $v$-adic cohomology is crystalline.
	
	Since $\CX_0$ has bad reduction at $v$ while $\CX_0'$ is isogenous to $E_0\times_K E_0'$, it follows that $E_0$ has bad reduction at $v$. Similarly, $E_s$ has bad reduction at $v$. For $P\in\{0,s\}$, let
	$$
	N_P:H^1_v(E_P)\longrightarrow H^1_v(E_P)
	$$
	denote the monodromy operator. Since $E_P$ has bad semistable reduction, $N_P\neq 0$; as $H^1_v(E_P)$ is two-dimensional and $N_P^2=0$, we have
	$$
	\ker N_P=\im {N_P}.
	$$
	
	For $P\in\{0,s\}$, choose $\gamma_P\in\ker N_P\setminus\{0\}$ and $\delta_P\in H^1_v(E_P)$ such that $N_P(\delta_P)=\gamma_P$. Then $\Gamma_v(E_P):=(\gamma_P,\delta_P)$ is a basis of $H^1_v(E_P)$. Since $E_P'$ has good reduction, choose a symplectic basis $\Gamma_v(E_P'):=(\gamma_P',\delta_P')$ of $H^1_v(E_P')$. We consider the ordered bases
	$$
	\Gamma_v(\CX_P'):=(\gamma_P,\gamma_P',\delta_P,\delta_P'),\qquad
	\beta_v(\CX_P'):=(\gamma_P,\delta_P,\gamma_P',\delta_P')
	$$
	of $H^1_v(\CX_P')$.
	
	By \Cref{gfunsbadprop}, the argument of \Cref{lemisogfin} gives \eqref{eq:lemisogmat2} with
	$$
	\Theta_v=J_{2,3}[\theta_s]_v[\theta_0]_v^{-1}J_{2,3},
	$$
	where, using \Cref{conjhyodokato}, $[\theta_P]_v$ denotes the matrix of the pullback induced by $\theta_P$ on Hyodo--Kato cohomology. As in the proof of \Cref{lemisogfin},
	$$
	[\theta_0]_v^{-1}=\frac{1}{\deg\theta_0}[\theta_0^\vee]_v,
	$$
	where $\theta_0^\vee:\CX_0'\rightarrow\CX_0$ is the dual isogeny.
	
	By \Cref{conjhyodokato}, the composition
	$$
	\psi:H^1_v(\CX_s')
	\xrightarrow{(\theta_s)^*_{HK}}
	H^1_v(\CX_s)=H^1_v(\CX_0)
	\xrightarrow{(\theta_0^\vee)^*_{HK}}
	H^1_v(\CX_0')
	$$
	is a morphism of $(\phi,N)$-modules. Set
	$$
	\Psi:=\frac{1}{\deg\theta_0}\psi.
	$$
	Since $\frac{1}{\deg\theta_0}\in\Q$, the map $\Psi$ is again a morphism of $(\phi,N)$-modules, and
	$$
	\Theta_v=[\Psi]_{\beta_v(\CX_s')}^{\beta_v(\CX_0')}.
	$$
	
	For $P\in\{0,s\}$, let $M_P$ denote the monodromy operator on $H^1_v(\CX_P')=H^1_v(E_P)\oplus H^1_v(E_P')$. Since $E_P'$ has good reduction, we have $M_P=N_P\oplus 0$, and hence $\ker M_P=\Span\{\gamma_P,\gamma_P',\delta_P'\}$. Since $\Psi$ is a morphism of $(\phi,N)$-modules, we have $\Psi M_s=M_0\Psi$. In particular, $\Psi(\gamma_s)$, $\Psi(\gamma_s')$, and $\Psi(\delta_s')$ belong to $\ker M_0=\Span\{\gamma_0,\gamma_0',\delta_0'\}$. Thus, 
	$$
	\Psi(\gamma_s)=\psi_{1,1}\gamma_0+\psi_{1,3}\gamma_0'+\psi_{1,4}\delta_0',
	$$
	and the same restriction applies to $\Psi(\gamma_s')$ and $\Psi(\delta_s')$. With respect to the ordered product bases $\beta_v(\CX_s')$ and $\beta_v(\CX_0')$, this gives \eqref{eq:badthetacm}.
	
	Assume now that there exist isogenies $\phi_0:E_0\rightarrow E_0'$ and $\phi_s:E_s\rightarrow E_s'$. Since $E_0$ and $E_0'$ are isogenous, they have the same reduction type at $v$. As $\CX_0$ has bad reduction at $v$, both $E_0$ and $E_0'$ therefore have bad semistable reduction. The same argument shows that $E_s$ and $E_s'$ have bad semistable reduction at $v$.
	
	For $P\in\{0,s\}$, let $N_P$, $N_P'$, and $M_P$ denote the monodromy operators on $H^1_v(E_P)$, $H^1_v(E_P')$, and $H^1_v(\CX_P')$, respectively. Since both elliptic factors have bad semistable reduction, $N_P$ and $N_P'$ are nonzero nilpotent operators on two-dimensional spaces. Hence
	$$
	\im{N_P}=\ker N_P,\qquad \im{N_P'}=\ker N_P'.
	$$
	Choose $\gamma_P\in\ker N_P\setminus\{0\}$ and $\delta_P$ such that $N_P(\delta_P)=\gamma_P$, and similarly choose $\gamma_P'\in\ker N_P'\setminus\{0\}$ and $\delta_P'$ such that $N_P'(\delta_P')=\gamma_P'$. We consider the ordered bases $\Gamma_v(E_P):=(\gamma_P,\delta_P)$ and $\Gamma_v(E_P'):=(\gamma_P',\delta_P')$, together with
	$$
	\Gamma_v(\CX_P'):=(\gamma_P,\gamma_P',\delta_P,\delta_P'),\qquad
	\beta_v(\CX_P'):=(\gamma_P,\delta_P,\gamma_P',\delta_P').
	$$
	With these choices, the same argument as in the first case, using \Cref{conjhyodokato}, gives \eqref{eq:lemisogmat2}, with
	$$
	\Theta_v=[\Psi]_{\beta_v(\CX_s')}^{\beta_v(\CX_0')},
	$$
	where $\Psi$ is defined by the same composition as above.
	
	Since $\Psi M_s=M_0\Psi$ and $\gamma_s,\gamma_s'\in\ker(M_s)$, we have
	$\Psi(\gamma_s),\Psi(\gamma_s')\in\ker(M_0)=\Span\{\gamma_0,\gamma_0'\}$. Moreover,
	$M_0\Psi(\delta_s)=\Psi(M_s\delta_s)=\Psi(\gamma_s)$, and similarly
	$M_0\Psi(\delta_s')=\Psi(\gamma_s')$. It follows that, for suitable $\psi_{i,j}\in\C_v$,
	$$
	\begin{gathered}
		\Psi(\gamma_s)=\psi_{1,1}\gamma_0+\psi_{1,3}\gamma_0',
		\qquad
		\Psi(\gamma_s')=\psi_{3,1}\gamma_0+\psi_{3,3}\gamma_0',\\
		\Psi(\delta_s)=\psi_{2,1}\gamma_0+\psi_{1,1}\delta_0+\psi_{2,3}\gamma_0'+\psi_{1,3}\delta_0',\\
		\Psi(\delta_s')=\psi_{4,1}\gamma_0+\psi_{3,1}\delta_0+\psi_{4,3}\gamma_0'+\psi_{3,3}\delta_0'.
	\end{gathered}
	$$
	With respect to the ordered bases $\beta_v(\CX_s')$ and $\beta_v(\CX_0')$, this gives \eqref{eq:badthetaisog}.
\end{proof}  
	\begin{prop}\label{excmbad}
	Assume that
	$$
	\CX_0\sim E_0\times_K E_0',
	$$
	where $E_0'$ is a CM elliptic curve. Let $s\in S(\bar{\Q})$ be such that
	$$
	\CX_s\sim E_s\times_{\bar{\Q}}E_s',
	$$
	where $E_s'$ is a CM elliptic curve. Assume that \Cref{conjhyodokato} and \Cref{conjrelhyodokato} hold. Then there exists a polynomial
	$$
	R_{s,\bad}\in\bar{\Q}[X_{i,j}:1\leq i,j\leq 4]
	$$
	satisfying the following:
	\begin{enumerate}
		\item $R_{s,\bad}$ has coefficients in a finite extension $L_s/K(s)$ with $[L_s:K(s)]$ bounded by an absolute constant;
		\item $R_{s,\bad}$ is homogeneous of degree $2$;
		\item for every $w\in\Sigma_{L_s,f}$ at which $\CX_0$ has bad semistable reduction and for which $s$ is $w$-adically close to $s_0$,
		$$
		\iota_w\bigl(R_{s,\bad}(Y_G(x(s)))\bigr)=0;
		$$
		\item $R_{s,\bad}\notin I(\SP_4)$.
	\end{enumerate}
\end{prop}

\begin{proof}
	The proof follows that of \Cref{propisoga2fin}, with the local relation replaced by the one constructed below from \Cref{lemrelsbad}. We retain the notation introduced there and fix $w\in\Sigma_{L_s,f}$ such that $\CX_0$ has bad semistable reduction at $w$ and $s$ is $w$-adically close to $s_0$.
		
	Write $\Theta_w=\begin{pmatrix}\Theta_1&\Theta_2\\\Theta_3&\Theta_4\end{pmatrix}$ with $\Theta_i\in M_2(\C_w)$, and set
	$$
	F_3(x):=\begin{pmatrix}F_{3,1}(x)&F_{3,2}(x)\\F_{4,1}(x)&F_{4,2}(x)\end{pmatrix},
	$$
	where $(F_{i,j}(x))$ is the matrix appearing in the proof of \Cref{propisoga2fin}. By \eqref{eq:badthetacm},
	$\Theta_3=\begin{pmatrix}\psi_{3,1}&0\\\psi_{4,1}&0\end{pmatrix}$, and hence
	\begin{equation}\label{eq:excmbadmain}
		\iota_w\bigl(\det F_3(x(s))\bigr)
		=\det\left(\Pi_w(E_s')\begin{pmatrix}\psi_{3,1}&0\\	\psi_{4,1}&0\end{pmatrix}\Pi_w(E_0)^{-1}\right)=0.
	\end{equation}
	
	Define $R_{s,\bad}\in L_s[X_{i,j}:1\leq i,j\leq 4]$ by
	$$
	R_{s,\bad}(Y_G(x)):=\frac{1}{\deg(\theta_0)^2}\det(F_3(x)).
	$$
	By construction and \eqref{eq:excmbadmain}, this polynomial satisfies properties $(1)$--$(3)$ of the statement. It remains to prove that $R_{s,\bad}\notin I(\SP_4)$. Suppose, to the contrary, that $R_{s,\bad}\in I(\SP_4)$.
	
	Using the notation of the proof of \Cref{propisoga2fin}, the relation in question becomes
	\begin{equation}\label{eq:excmbad2}
		\iota_w\bigl(\widetilde{F}_{2,1}(x(s))\widetilde{F}_{4,3}(x(s))
		-\widetilde{F}_{2,3}(x(s))\widetilde{F}_{4,1}(x(s))\bigr)=0.
	\end{equation}
	Let $R_{i,j}$ denote the polynomial satisfying $R_{i,j}(Y_G(x))=\widetilde{F}_{i,j}(x)$. Write
	$A_0=(a_{i,j})$, $B_0=(b_{i,j})$, $C_0=(c_{i,j})$, $A_s=(d_{i,j})$, $B_s=(f_{i,j})$, and $C_s=(e_{i,j})$. Then
	$$
	\begin{aligned}
		R_{2,1}={}&d_{2,1}(X_{1,1}a_{1,1}+X_{1,2}a_{2,1}+X_{1,3}b_{1,1}+X_{1,4}b_{2,1})\\
		&+d_{2,2}(X_{2,1}a_{1,1}+X_{2,2}a_{2,1}+X_{2,3}b_{1,1}+X_{2,4}b_{2,1}),\\
		R_{2,3}={}&d_{2,1}(X_{1,3}c_{1,1}+X_{1,4}c_{2,1})
		+d_{2,2}(X_{2,3}c_{1,1}+X_{2,4}c_{2,1}),\\
		R_{4,1}={}&f_{2,1}(X_{1,1}a_{1,1}+X_{1,2}a_{2,1})
		+f_{2,2}(X_{2,1}a_{1,1}+X_{2,2}a_{2,1})\\
		&+e_{2,1}(X_{3,1}a_{1,1}+X_{3,2}a_{2,1})
		+e_{2,2}(X_{4,1}a_{1,1}+X_{4,2}a_{2,1})\\
		&+f_{2,1}(X_{1,3}b_{1,1}+X_{1,4}b_{2,1})
		+f_{2,2}(X_{2,3}b_{1,1}+X_{2,4}b_{2,1})\\
		&+e_{2,1}(X_{3,3}b_{1,1}+X_{3,4}b_{2,1})
		+e_{2,2}(X_{4,3}b_{1,1}+X_{4,4}b_{2,1}),\\
		R_{4,3}={}&f_{2,1}(X_{1,3}c_{1,1}+X_{1,4}c_{2,1})
		+f_{2,2}(X_{2,3}c_{1,1}+X_{2,4}c_{2,1})\\
		&+e_{2,1}(X_{3,3}c_{1,1}+X_{3,4}c_{2,1})
		+e_{2,2}(X_{4,3}c_{1,1}+X_{4,4}c_{2,1}).
	\end{aligned}
	$$
	
	Since $R_{s,\bad}\in I(\SP_4)$, we have $R_{s,\bad}(S(p,q,r,n))=0$ for all $p,q,r,n\in\Q$ with $pn-rq\neq 0$, where
	$$
	S(p,q,r,n):=
	\begin{pmatrix}
		p&q&0&0\\
		r&n&0&0\\
		0&0&\frac{n}{pn-rq}&\frac{-r}{pn-rq}\\
		0&0&\frac{-q}{pn-rq}&\frac{p}{pn-rq}
	\end{pmatrix}
	\in\SP_4(\Q).
	$$
	Substitution gives
	$$
	0=
	[d_{2,1}(pa_{1,1}+qa_{2,1})+d_{2,2}(ra_{1,1}+na_{2,1})]
	[e_{2,1}(nc_{1,1}-rc_{2,1})+e_{2,2}(pc_{2,1}-qc_{1,1})].
	$$
	Since this identity holds on the Zariski dense open set $pn-rq\neq 0$, one of the two linear factors vanishes identically. Hence either
	$a_{1,1}d_{2,1}=a_{2,1}d_{2,1}=a_{1,1}d_{2,2}=a_{2,1}d_{2,2}=0$
	or
	$c_{1,1}e_{2,1}=c_{2,1}e_{2,1}=c_{1,1}e_{2,2}=c_{2,1}e_{2,2}=0$.
	The first alternative contradicts the invertibility of $A_0$ and $A_s$, while the second contradicts the invertibility of $C_0$ and $C_s$. Thus $R_{s,\bad}\notin I(\SP_4)$.
\end{proof}
	\begin{prop}\label{e2relbad}
	Assume that
	$$
	\CX_0\sim E_0\times_K E_0',
	$$
	where $E_0$ and $E_0'$ are isogenous elliptic curves, and assume that \Cref{conjhyodokato} and \Cref{conjrelhyodokato} hold.
	
	Let $s\in S(\bar{\Q})$ be such that
	$$
	\CX_s\sim \CX_s':=E_s\times_{\bar{\Q}}E_s',
	$$
	where $E_s$ and $E_s'$ are isogenous elliptic curves. Then there exists
	$$
	R_{s,\bad}\in\bar{\Q}[X_{i,j}:1\leq i,j\leq 4]
	$$
	satisfying properties $(1)$, $(3)$, and $(4)$ of \Cref{excmbad}, and homogeneous with $\deg(R_{s,\bad})=4$.
\end{prop}

\begin{proof}	
	Fix $v\in\Sigma_{L_s,f}$ at which $\CX_0$ has bad semistable reduction and for which $s$ is $v$-adically close to $s_0$. Then $\CX_0$ and $\CX_s$ have bad semistable reduction at $v$. Since $E_P$ and $E_P'$ are isogenous for $P\in\{0,s\}$, all four elliptic curves have bad semistable reduction at $v$.
	
	Let $\phi_P:E_P\rightarrow E_P'$ denote the corresponding isogenies. By \cite[Theorem B]{vonk}, the induced pullback
	$\phi_{P,HK}^*:H^1_v(E_P')\rightarrow H^1_v(E_P)$
	is a morphism of $(\phi,N)$-modules and is compatible, through the Hyodo--Kato comparison isomorphism \eqref{eq:hyodokatoisom}, with the de Rham pullback induced by $\phi_P$. Writing this compatibility with respect to the chosen bases, as in \cite[Lemma 3.1]{papaszpy1}, gives
	\begin{equation}\label{eq:isogpairsperiodsbad}
		[\phi_P]_{dR}\cdot\Pi_v(E_P)=\Pi_v(E_P')\cdot[\phi_P]_v.
	\end{equation}
	Here $[\phi_P]_{dR}$ denotes the matrix of the de Rham pullback with respect to the chosen Hodge bases, while $[\phi_P]_v$ denotes the matrix of $\phi_{P,HK}^*$ with respect to the bases $\Gamma_v(E_P)$ and $\Gamma_v(E_P')$ chosen in the proof of \Cref{lemrelsbad}.
	
	Retaining the notation of \Cref{lemrelsbad} and \Cref{propisoga2fin}, we have
	\begin{equation}\label{eq:isogmainrelgen}
		\iota_v\bigl((F_{i,j}(x(s)))\bigr)=	\begin{pmatrix}	\Pi_v(E_s)&0\\0&\Pi_v(E_s') \end{pmatrix}\Theta_v\begin{pmatrix}\Pi_v(E_0)^{-1}&0\\		0&\Pi_v(E_0')^{-1}\end{pmatrix}.
	\end{equation}
	
	Using \eqref{eq:isogpairsperiodsbad}, we may rewrite \eqref{eq:isogmainrelgen} as
	\begin{multline}\label{eq:e2relsbadmain}
		\iota_v\left(
		\begin{pmatrix}I_2&0\\0&[\phi_s]_{dR}^{-1}\end{pmatrix}
		\cdot(F_{i,j}(x(s)))\cdot
		\begin{pmatrix}I_2&0\\0&[\phi_0]_{dR}\end{pmatrix}
		\right)=\\
		=\begin{pmatrix}\Pi_v(E_s)&0\\0&\Pi_v(E_s)\end{pmatrix}
		\cdot\begin{pmatrix}I_2&0\\0&[\phi_s]_{v}^{-1}\end{pmatrix}
		\cdot\Theta_v\cdot
		\begin{pmatrix}I_2&0\\0&[\phi_0]_{v}\end{pmatrix}
		\cdot\begin{pmatrix}\Pi_v(E_0)^{-1}&0\\0&\Pi_v(E_0)^{-1}\end{pmatrix}.
	\end{multline}
	
	Since $\phi_{P,HK}^*$ is a morphism of $(\phi,N)$-modules, the matrix $[\phi_P]_v$ intertwines the monodromy operators $N_P'$ and $N_P$. With respect to the monodromy-adapted bases $\Gamma_v(E_P)$ and $\Gamma_v(E_P')$ chosen in the proof of \Cref{lemrelsbad}, this gives
	$$
	[\phi_P]_v=	\begin{pmatrix}\xi_P&0\\\zeta_P&\xi_P\end{pmatrix}
	$$
	for some $\xi_P\in\C_v^\times$ and $\zeta_P\in\C_v$.
	
	Write
	$\Theta_v=\begin{pmatrix}\Theta_1&\Theta_2\\ \Theta_3&\Theta_4\end{pmatrix}$,
	with $\Theta_j\in M_2(\C_v)$. Then the right-hand side of \eqref{eq:e2relsbadmain} takes the form
	\begin{equation}\label{eq:thetatildee2bad}
		\begin{pmatrix}
			\Pi_v(E_s)\widetilde{\Theta}_1\Pi_v(E_0)^{-1}&
			\Pi_v(E_s)\widetilde{\Theta}_2\Pi_v(E_0)^{-1}\\
			\Pi_v(E_s)\widetilde{\Theta}_3\Pi_v(E_0)^{-1}&
			\Pi_v(E_s)\widetilde{\Theta}_4\Pi_v(E_0)^{-1}
		\end{pmatrix},
	\end{equation}
	where each $\widetilde{\Theta}_j\in M_2(\C_v)$ has the form $\begin{pmatrix}\alpha_j&0\\ \beta_j&\alpha_j\end{pmatrix}$. Indeed, by \Cref{lemrelsbad} the blocks $\Theta_j$ have this form, and the preceding description of $[\phi_P]_v$ shows that this class of matrices is preserved by the multiplications appearing in \eqref{eq:e2relsbadmain}.
	
	Set $(g_{i,j}(x)):=\begin{pmatrix}I_2&0\\0&[\phi_s]_{dR}^{-1}\end{pmatrix}(F_{i,j}(x))\begin{pmatrix}I_2&0\\0&[\phi_0]_{dR}\end{pmatrix}$ and $G_{i,j}:=\iota_v(g_{i,j}(x(s)))$. The power series $g_{i,j}(x)$ are algebraic linear combinations of the entries of $Y_G(x)$ and are independent of the place $v$, since the matrices involved on the de Rham side are independent of $v$. Write $(G_{i,j})=\begin{pmatrix}G_1&G_2\\G_3&G_4\end{pmatrix}$, with $G_j\in M_2(\C_v)$. Then \eqref{eq:e2relsbadmain} and \eqref{eq:thetatildee2bad} give
	\begin{equation}\label{eq:e2relsbadactualmain}
		\Pi_v(E_s)^{-1}G_j\Pi_v(E_0)=\begin{pmatrix}\alpha_j&0\\ \beta_j&\alpha_j\end{pmatrix},\qquad 1\leq j\leq 4.
	\end{equation}
	
	Write $\Pi_v(E_s)^{-1}=(\pi_{i,j})$ and $\Pi_v(E_0)=(\rho_{i,j})$. For $j=1$, equality of the diagonal entries in \eqref{eq:e2relsbadactualmain} gives
	\begin{equation}\label{eq:firstrow}
		\chi_1G_{1,1}+\chi_2G_{1,2}+\chi_3G_{2,1}+\chi_4G_{2,2}=0,
	\end{equation}
	where $\chi_1=\pi_{1,1}\rho_{1,1}-\pi_{2,1}\rho_{1,2}$, $\chi_2=\pi_{1,1}\rho_{2,1}-\pi_{2,1}\rho_{2,2}$, $\chi_3=\pi_{1,2}\rho_{1,1}-\pi_{2,2}\rho_{1,2}$, and $\chi_4=\pi_{1,2}\rho_{2,1}-\pi_{2,2}\rho_{2,2}$. Applying the same argument for $j=2,3,4$ gives
	$$
	\begin{gathered}
		\chi_1G_{1,3}+\chi_2G_{1,4}+\chi_3G_{2,3}+\chi_4G_{2,4}=0,\\
		\chi_1G_{3,1}+\chi_2G_{3,2}+\chi_3G_{4,1}+\chi_4G_{4,2}=0,\\
		\chi_1G_{3,3}+\chi_2G_{3,4}+\chi_3G_{4,3}+\chi_4G_{4,4}=0.
	\end{gathered}
	$$
	
	Note that $\vec{\chi}\neq 0$. Indeed, if $\chi_1=\chi_2=0$, then
	$$
	\begin{pmatrix}\rho_{1,1}&\rho_{1,2}\\ \rho_{2,1}&\rho_{2,2}\end{pmatrix}
	\begin{pmatrix}\pi_{1,1}\\ -\pi_{2,1}\end{pmatrix}=0.
	$$
	Since $(\rho_{i,j})$ is invertible, this gives $\pi_{1,1}=\pi_{2,1}=0$, contradicting the invertibility of $(\pi_{i,j})$.
	
	Set
	$$
	\widetilde{G}:=
	\begin{pmatrix}
		G_{1,1}&G_{1,2}&G_{2,1}&G_{2,2}\\
		G_{1,3}&G_{1,4}&G_{2,3}&G_{2,4}\\
		G_{3,1}&G_{3,2}&G_{4,1}&G_{4,2}\\
		G_{3,3}&G_{3,4}&G_{4,3}&G_{4,4}
	\end{pmatrix}.
	$$
	The preceding equations give $\widetilde{G}\vec{\chi}=0$. Since $\vec{\chi}\neq 0$, it follows that $\det\widetilde{G}=0$.
	
	Let $\widetilde{g}(x)$ be the $4\times 4$ matrix obtained from $\widetilde{G}$ by replacing each entry $G_{i,j}$ with the corresponding power series $g_{i,j}(x)$. Define $R_{s,\bad}\in\bar{\Q}[X_{i,j}:1\leq i,j\leq 4]$ by
	$$
	R_{s,\bad}(Y_G(x)):=\det\widetilde{g}(x).
	$$
	Then $R_{s,\bad}$ is homogeneous of degree $4$, is independent of the choice of the bad place $v$, and satisfies $\iota_v(R_{s,\bad}(Y_G(x(s))))=0$.
	
	It remains to prove that $R_{s,\bad}\notin I(\SP_4)$. Suppose, to the contrary, that $R_{s,\bad}\in I(\SP_4)$. Recall that the Hodge basis $\Gamma_{dR}(\CX)$ specializes at $s_0$ to the basis induced by the isogeny $\theta_0$. Hence, in the notation used below, $(a_{i,j})=I_2$. The accompanying Mathematica code \cite{mathematicacode} computes, for $R_{s,\bad}$, the remainder upon division by a Gr\"obner basis of $I(\SP_4)$. Under our assumption this remainder vanishes identically. For a monomial $M$ appearing in it, write $c(M)$ for its coefficient.
	
	Set
	$$
	p:=a_sd_{2,1}e_{1,1}-c_sd_{1,1}e_{2,1},
	q:=a_sd_{2,1}e_{1,2}-c_sd_{1,1}e_{2,2},
	r:=a_sd_{2,2}e_{1,1}-c_sd_{1,2}e_{2,1}.
	$$
	The identities $c(X_{1,3}^2X_{3,1}^2)=0$ and $c(X_{1,3}^2X_{3,2}^2)=0$ give
	$$
	c_0^2c_{1,2}^2p^2=a_0^2c_{1,1}^2p^2=0.
	$$
	Since $(c_{i,j})$ is invertible and $a_0c_0\neq0$, it follows that $p=0$.
	
	We also have $q\neq0$. Indeed, if $q=0$, then $(a_sd_{2,1},-c_sd_{1,1})(e_{i,j})=(0,0)$. Since $(e_{i,j})$ is invertible and $a_sc_s\neq0$, this gives $d_{2,1}=d_{1,1}=0$, contradicting the invertibility of $(d_{i,j})$.
	
	The identities
	$$
	c(X_{1,4}^2X_{4,1}X_{4,2})=0,\qquad c(X_{1,4}^2X_{4,2}^2)=0
	$$
	then give $c_{1,2}c_{2,2}=c_{1,1}c_{2,2}=0$. Hence $c_{2,2}=0$, since $(c_{i,j})$ is invertible, and therefore $c_{1,2}c_{2,1}\neq0$.
	
	Finally, $p=0$ implies $r\neq0$. Otherwise
	$$
	(-c_se_{2,1},a_se_{1,1})(d_{i,j})=(0,0),
	$$
	which contradicts the invertibility of $(d_{i,j})$ and $(e_{i,j})$. But then
	$$
	c(X_{2,1}X_{2,3}X_{3,2}X_{3,4})=a_0c_0c_{1,2}c_{2,1}r^2\neq0,
	$$
	contradicting our assumption that all coefficients of the remainder vanish. Hence $R_{s,\bad}\notin I(\SP_4)$.
\end{proof}
	\subsection{Height bounds}\label{section:htboundsbad}

\begin{prop}\label{prophtboundsbad}
	Let $f:\CX\rightarrow S$ be a family of abelian surfaces over a smooth geometrically irreducible curve $S$ defined over a number field $K$, and let $\{\xi_1,\ld,\xi_l\}\subset S(K)$ be as in \Cref{reduction}. Assume that the image of the induced morphism $S\rightarrow\mathcal{A}_2$ is Hodge generic and that \Cref{conjhyodokato} and \Cref{conjrelhyodokato} hold.
	
	Suppose that, for every $\xi\in\{\xi_1,\ld,\xi_l\}$, the fiber $\CX_\xi$ is either an $E^2$-point or an $E\times\mathrm{CM}$-point. Then there exist constants $c_1,c_2>0$ such that, for every $s\in S(\bar{\Q})$ which is either an $E^2$-point or an $E\times\mathrm{CM}$-point, we have
	$$
	h(s)\leq c_1\Bigl([K(s):\Q]\bigl(1+|\Sigma_{K,\ssing}(s,0)|\bigr)\Bigr)^{c_2}.
	$$
\end{prop}

\begin{proof}
	We first note that the two types of split fibers considered above cannot occur in the same bad-reduction proximity disc. Let $v$ be a place of bad semistable reduction of $\CX_\xi$ and let $s\in\Delta_v(\xi)$. By \Cref{conjrelhyodokato}, the $(\phi,N)$-modules attached to $\CX_s$ and $\CX_\xi$ are identified. In particular, the rank of the monodromy operator is the same for the two fibers.
	
	If $\CX_\xi\sim E_0\times E_0'$ with $E_0'$ CM, then $E_0'$ has good reduction, while $E_0$ has bad semistable reduction. Hence the monodromy operator on
	$$
	H^1_v(\CX_\xi)=H^1_v(E_0)\oplus H^1_v(E_0')
	$$
	has the form $N_{E_0}\oplus0$ and has rank $1$.
	
	If instead $\CX_\xi\sim E_0\times E_0'$ with $E_0\sim E_0'$, then the two elliptic factors have the same reduction type; since $\CX_\xi$ has bad reduction, both have bad semistable reduction, and the monodromy operator has rank $2$. Consequently, an $E\times\mathrm{CM}$-point and an $E^2$-point cannot lie in the same bad-reduction proximity disc.
	
	Fix $s\in S(\bar{\Q})$ which is either an $E^2$-point or an $E\times\mathrm{CM}$-point, and let $L_s/K(s)$ be a common finite extension, of degree bounded independently of $s$, over which all the relation polynomials used below are defined.
	
	As in the proof of \Cref{htboundsfin}, set
	$$
	\Sigma(s):=\{v\in\Sigma_{L_s}:s\text{ is }v\text{-adically close to }0\}.
	$$
	If $\Sigma(s)=\emptyset$, the argument in the proof of \Cref{htboundsfin} gives $h(s)\leq c$ for a constant $c>0$ independent of $s$. We may thus assume that $\Sigma(s)\neq\emptyset$.
	
	Let $\Sigma(s)_{\bad}\subset\Sigma(s)_f$ consist of the places at which the distinguished fibers have bad semistable reduction, and set
	$$
	\Sigma(s)_{\good}:=\Sigma(s)\setminus\Sigma(s)_{\bad}.
	$$

	At the places in $\Sigma(s)_{\good}$, use the local relations constructed in the proof of \Cref{htboundsfin}, and let $R_{s,\good}$ denote their product, with $R_{s,\good}:=1$ if $\Sigma(s)_{\good}=\emptyset$. If $\Sigma(s)_{\bad}=\emptyset$, the conclusion follows from the proof of \Cref{htboundsfin}. Therefore, we assume that $\Sigma(s)_{\bad}\neq\emptyset$.
	
	For $v\in\Sigma(s)_{\bad}$, choose $t=t(v)\in\{1,\ld,l\}$ such that $s$ is $v$-adically close to $\xi_t$, and let $\lambda\in\Lambda$ be the representative with $t\sim\lambda$. By the observation above, $\CX_s$ and $\CX_{\xi_t}$ have the same splitting type. Hence the appropriate one of \Cref{excmbad} and \Cref{e2relbad} gives a polynomial $R_{s,t,\bad}$. For fixed $t$, this polynomial is independent of the bad place.
	
	For $\lambda\in\Lambda$, write $T_{s,\lambda,\bad}:=\{t(v):v\in\Sigma(s)_{\bad},\ t(v)\sim\lambda\}$ and set $R_{s,\lambda,\bad}:=\prod_{t\in T_{s,\lambda,\bad}}R_{s,t,\bad}$. Let
	$$
	R_s:=R_{s,\good}\cdot\prod_{\lambda\in\Lambda}R_{s,\lambda,\bad}.
	$$
	By construction, $R_s$ gives a global relation at $s$.
	
	Non-triviality follows as in the proof of \Cref{htboundsfin} by the fact that none of the local factors lies in the ideal
	$$
	I(\SP_4)\subset\bar{\Q}[X_{i,j}^{(\lambda)}:1\leq i,j\leq4].
	$$
	
	By \Cref{excmbad} and \Cref{e2relbad}, we have $\deg R_{s,\lambda,\bad}\leq 4|T_{s,\lambda,\bad}|$ for every $\lambda\in\Lambda$. Since the sets $T_{s,\lambda,\bad}$ are disjoint subsets of $\{1,\ld,l\}$, the bad-reduction factors contribute at most $4l$ to the degree of $R_s$. As $l$ depends only on the fixed auxiliary cover, this bound is independent of $s$. Combining this with the degree estimate in the proof of \Cref{htboundsfin}, we obtain
	$$
	\deg R_s\leq c_0[K(s):\Q]+c_0'[K(s):\Q]|\Sigma_{K,\ssing}(s,0)|
	$$
	for constants $c_0,c_0'>0$ independent of $s$, after enlarging $c_0$ if necessary.
	
	The Andr\'e--Bombieri Hasse principle, applied as in the proof of \Cref{htboundsfin}, now gives
	$$
	h(s)\leq c_1\Bigl([K(s):\Q]\bigl(1+|\Sigma_{K,\ssing}(s,0)|\bigr)\Bigr)^{c_2},
	$$
	for constants $c_1,c_2>0$ independent of $s$.
\end{proof}
	\subsection{Applications to unlikely intersections}\label{section:badapplications}

Retain the notation of \Cref{section:applicationsgood}.

\begin{prop}\label{lgobad}
	Let $Z\subset\mathcal{A}_2$ be a smooth irreducible curve defined over $\bar{\Q}$ that is not contained in a proper special subvariety of $\mathcal{A}_2$, and fix $N\in\N$. Assume that $s_0\in Z(\bar{\Q})$ is either an $E^2$-point or an $E\times\mathrm{CM}$-point, and that \Cref{conjhyodokato} and \Cref{conjrelhyodokato} hold.
	
	Then there exist positive constants
	$$
	c_1=c_1(Z,N,s_0),\qquad c_2=c_2(Z,s_0)
	$$
	such that, for every $s\in\Sha_{ZP\text{-}\spli,N}(Z)$,
	\begin{equation}\label{eq:lgobad}
		|\gal(\bar{\Q}/\Q)\cdot s|\geq c_1\Delta(V_s)^{c_2}.
	\end{equation}
\end{prop}

\begin{proof}
Let $K$ be a number field over which $Z$ and $s_0$ are defined, so that $s_0\in Z(K)$. We use the auxiliary cover $(S,\{\xi_1,\ld,\xi_l\})$ of $(Z,s_0)$ fixed in the proof of \Cref{lgogood}.

By \Cref{prophtboundsbad}, there exist constants $c_1',c_2'>0$ such that
	$$
	h(s)\leq c_1'\Bigl([K(s):\Q]\bigl(1+|\Sigma_{K,\ssing}(s,0)|\bigr)\Bigr)^{c_2'}
	$$
	for every $s\in\Sha_{ZP\text{-}\spli,N}(Z)$.
	
	Since $|\Sigma_{K,\ssing}(s,0)|\leq N$ on this set, the argument of \Cref{lgogood}, using the Masser--W\"ustholz isogeny estimates and the corresponding results of Daw--Orr for $E\times\mathrm{CM}$-points and $E^2$-points, gives constants
	$$
	c_1=c_1(Z,N,s_0)>0,\qquad c_2=c_2(Z,s_0)>0
	$$
	such that
	$$
	|\gal(\bar{\Q}/\Q)\cdot s|\geq c_1\Delta(V_s)^{c_2}.
	$$
\end{proof}

\begin{cor}\label{badredmainzpapp}
	Let $Z\subset\mathcal A_2$ be a smooth irreducible curve defined over $\bar{\Q}$ that is not contained in a proper special subvariety of $\mathcal A_2$. Assume that $Z$ contains an $E^2$- or $E\times CM$-point $s_0\in Z(\bar{\Q})$ and that \Cref{conjhyodokato} and \Cref{conjrelhyodokato} hold. Then, for every $N\in\N$, the set $\Sha_{ZP\text{-}\spli,N}(Z)$ defined in \Cref{goodreductionmainzpapp} is finite.
\end{cor}
\begin{proof}
	The result follows from \Cref{lgobad} and the same results of Daw and Orr used in the proof of \Cref{goodreductionmainzpapp}.
\end{proof}

	\subsection*{Acknowledgments}
	The author thanks Chris Daw for his encouragement and for many enlightening discussions on G-functions. The author thanks Or Shahar for showing him the basics on Wolfram Mathematica. The author also heartily thanks the Hebrew University of Jerusalem, the Weizmann Institute of Science, and the IAS for the excellent working conditions.
	
	Work on this project started when the author was supported by Michael Temkin's ERC Consolidator Grant 770922 - BirNonArchGeom. Throughout the majority of this work, the author received funding by the European Union (ERC, SharpOS, 101087910), and by the Israel Science Foundation (grant No. 2067/23). In the final stages of this work, the author was supported by the Minerva Research Foundation Member Fund while in residence at the Institute for Advanced Study for the academic year $2025-26$. 
	
		\subsection*{Declaration of AI usage.}
	While revising this manuscript from an earlier version available on arXiv, the author used ChatGPT (OpenAI) for editorial and language assistance, principally to improve the exposition. For some of the explicit non-triviality checks, ChatGPT was used to analyze lists of coefficients computed by the author using Mathematica and to help simplify the exposition of the corresponding checks. ChatGPT was also used to integrate previously written and functional Mathematica code into coherent notebooks. The accompanying Mathematica notebooks are available online \cite{mathematicacode}. The mathematical ideas, results, proofs, and arguments are solely due to the author, who has independently verified the relevant computations and takes full responsibility for the manuscript.
	
	%\part{Appendix Mathematica code}
\appendix	

\section{Non-triviality of some relations}\label{section:appendix}

\subsection{Non-triviality check for Case~1 of \Cref{propordinarye2}}\label{app:ordinarye2e2}

We prove the non-triviality assertion in Case 1 of \Cref{propordinarye2}, retaining the notation of its proof.
%CLEAN

Suppose, to the contrary, that $R_{s,\simord}\in I(\SP_4)$. The accompanying Mathematica code \cite{mathematicacode} computes, for $R_{s,\simord}$, denoted by ``Qe2e2'' in the code, the remainder upon division by a Gr\"obner basis of $I(\SP_4)$. Under this assumption, every coefficient of the remainder vanishes.
%CLEAN

Write $C:=C_0=(c_{i,j})$, $D:=A_s=(d_{i,j})$, $E:=C_s=(e_{i,j})$, and
$$
[\varphi_0]_{dR}=\begin{pmatrix}a_0&0\\b_0&c_0\end{pmatrix},\qquad
[\varphi_s]_{dR}^{-1}=\begin{pmatrix}a_s&0\\b_s&c_s\end{pmatrix}.
$$

Since the matrices above are invertible, $a_0c_0a_sc_s\det(C)\det(D)\det(E)\neq0$. Set
$$
\kappa:=a_0c_0\det(C),
u_1:=a_sd_{2,1}e_{1,1}-c_sd_{1,1}e_{2,1},
u_2:=a_sd_{2,2}e_{1,1}-c_sd_{1,2}e_{2,1}.
$$
%CLEAN

For a monomial $M$ appearing in the remainder, write $c(M)$ for its coefficient. The computation in \cite{mathematicacode} gives
$$
\begin{aligned}
	c(X_{1,2}X_{1,3}X_{3,1}X_{4,4})&=-\kappa d_{1,1}e_{1,2}u_1,\\
	c(X_{1,3}X_{2,2}X_{3,1}X_{4,4})&=\kappa\bigl((d_{1,1}e_{1,1}-d_{1,2}e_{1,2})u_1-d_{1,1}e_{1,2}u_2\bigr),\\
	c(X_{2,2}X_{2,3}X_{3,1}X_{4,4})&=\kappa\bigl(d_{1,2}e_{1,1}u_1+(d_{1,1}e_{1,1}-d_{1,2}e_{1,2})u_2\bigr).
\end{aligned}
$$
Since these coefficients vanish and $\kappa\neq0$, the constant, linear, and quadratic coefficients of
$$
H(t):=(d_{1,1}-d_{1,2}t)(e_{1,2}+e_{1,1}t)(u_1-u_2t)
$$
all vanish. Hence $t^3$ divides $H(t)$.
%CLEAN

Each factor of $H(t)$ is a nonzero polynomial of degree at most one. Indeed, the first two are nonzero by the invertibility of $D$ and $E$, while
$$
\begin{pmatrix}u_1\\u_2\end{pmatrix}
=D^{\mathsf T}\begin{pmatrix}0&-c_s\\a_s&0\end{pmatrix}
\begin{pmatrix}e_{1,1}\\e_{2,1}\end{pmatrix}\neq0,
$$
since $a_sc_s\det(D)\det(E)\neq0$. Thus each factor must be divisible by $t$, giving $d_{1,1}=e_{1,2}=u_1=0$. But the invertibility of $D$ and $E$ then gives $d_{2,1}e_{1,1}\neq0$, so $u_1=a_sd_{2,1}e_{1,1}\neq0$, a contradiction. Hence $R_{s,\simord}\notin I(\SP_4)$.

\subsection{Non-triviality check for Case~1 of \Cref{propsupersingularexcm}}\label{section:app-supersingular-nontriviality}

We prove the non-triviality assertion in Case~1 of \Cref{propsupersingularexcm}, retaining the notation of its proof.
%CLEAN

Suppose, to the contrary, that $R_{s,v}\in I(\SP_4)$. The accompanying Mathematica code \cite{mathematicacode} computes, for $R_{s,v}$, denoted by ``Rss1'' in the code, the remainder upon division by a Gr\"obner basis of $I(\SP_4)$. Under this assumption, every coefficient of the remainder vanishes.
%CLEAN

Write $C:=C_0=(c_{i,j})$, $D:=A_s=(d_{i,j})$, and $E:=C_s=(e_{i,j})$, so that $\det(C)\det(D)\det(E)\neq0$. For a monomial $M$ appearing in the remainder, write $c(M)$ for its coefficient. We use the scalar $a$ introduced in Case~1 of the proof of \Cref{propsupersingularexcm}; this scalar is denoted by ``d1'' in the code.
%CLEAN

Set
$$
\begin{gathered}
	\alpha:=c_{1,2}c_{2,1},\quad
	\beta:=c_{1,1}c_{2,2},\quad
	\sigma:=\alpha+\beta,\\
	h:=d_{1,1}d_{2,2},\quad
	k:=d_{1,2}d_{2,1},\quad
	y:=e_{1,2}e_{2,1},\quad
	z:=e_{1,1}e_{2,2}.
\end{gathered}
$$
Thus $\det(C)=\beta-\alpha$, $\det(D)=h-k$, and $\det(E)=z-y$.
%CLEAN

Define $q_0,q_1,q_2$ by
$$
(d_{1,1}-d_{1,2}u)(d_{2,1}-d_{2,2}u)(e_{1,2}+e_{1,1}u)(e_{2,2}+e_{2,1}u)
=q_0+q_1u+q_2u^2+\cdots.
$$
These three coefficients cannot all vanish. Indeed, otherwise the product on the left would vanish to order at least three at $u=0$, whereas the invertibility of $D$ and $E$ implies that at most one factor from each of the two corresponding pairs vanishes there.
%CLEAN

The accompanying computation gives
$$
\begin{gathered}
	c(X_{1,2}^2X_{4,4}^2)=(a-1)\sigma q_0,\qquad
	c(X_{1,2}X_{2,2}X_{4,4}^2)=-(a-1)\sigma q_1,\\
	c(X_{2,2}^2X_{4,4}^2)=(a-1)\sigma q_2,\\
	H:=c(X_{1,3}X_{3,1})=\alpha(hy-akz),\qquad
	T:=c(1)=(k-ah)(\alpha y+\beta z).
\end{gathered}
$$
and, writing $R:=c(X_{2,2}X_{4,4})$,
$$
R+2T
=-(a\beta-\alpha)\det(D)\det(E)
-(a-1)\sigma(d_{1,1}e_{1,1}-d_{1,2}e_{1,2})(d_{2,1}e_{2,1}-d_{2,2}e_{2,2}).
$$
%CLEAN

If $a=1$, then $R=T=0$ gives
$$
0=R+2T=-(\beta-\alpha)\det(D)\det(E)
=-\det(C)\det(D)\det(E),
$$
a contradiction.
%CLEAN

Suppose now that $a\neq1$. Since the first three coefficients above vanish and $q_0,q_1,q_2$ are not all zero, we have $\sigma=0$. Hence $\beta=-\alpha$, and $\det(C)\neq0$ gives $\alpha\neq0$. The identity for $R+2T$ now yields $0=(a+1)\alpha\det(D)\det(E)$, so $a=-1$.
%CLEAN

Using $\beta=-\alpha$ and $a=-1$, the equation $T=0$ becomes $-\alpha(k+h)\det(E)=0$, hence $k=-h$. Since $\det(D)=h-k=2h\neq0$, we have $h\neq0$. Finally,
$$
H=\alpha(hy+kz)=-\alpha h\det(E)\neq0,
$$
contradicting the vanishing of the remainder. Hence $R_{s,v}\notin I(\SP_4)$.
%CLEAN

\subsection{Non-triviality check for Case~1 of \Cref{propsupersingulare2}}\label{section:app-supersingular-e2-nontriviality}

We prove here that at least one of the polynomials $R_{s,v,1}$ and $R_{s,v,2}$ constructed in Case~1 of the proof of \Cref{propsupersingulare2} does not belong to $I(\SP_4)$. Suppose, to the contrary, that both $R_{s,v,1}$ and $R_{s,v,2}$ belong to $I(\SP_4)$. The accompanying Mathematica code \cite{mathematicacode} computes, for each polynomial, the remainder upon division by a Gr\"obner basis of $I(\SP_4)$, using the labels ``Rss5'' and ``Rss6'', respectively. Under this assumption, every coefficient of both remainders vanishes.
%CLEAN

We retain the notation $C:=C_0=(c_{i,j})$, $D:=A_s=(d_{i,j})$, $E:=C_s=(e_{i,j})$, and the scalars $a_0,c_0,a_s,c_s$ from \Cref{app:ordinarye2e2}. In particular, $a_0c_0a_sc_s\det(C)\det(D)\det(E)\neq0$.

Set $t_v:=\tr([\alpha]_v)/m$ and $\delta_v:=\det([\alpha]_v)/m^2$, and define
$$
H_{\tr}:=\begin{pmatrix}-t_v&c_s\\a_s&0\end{pmatrix},
\qquad
H_{\det}:=\begin{pmatrix}\delta_v&0\\0&-a_sc_s\end{pmatrix}.
$$
Set moreover $M_{\tr}:=D^{\mathsf T}H_{\tr}E$ and $M_{\det}:=D^{\mathsf T}H_{\det}E$.
%CLEAN

For the three pairs of monomials
$(X_{1,2}X_{4,3},X_{1,2}X_{4,4})$,
$(X_{2,2}X_{3,3},X_{2,2}X_{3,4})$, and
$(X_{2,2}X_{4,3},X_{2,2}X_{4,4})$, the corresponding coefficient pairs are
$$
\begin{array}{c|c|c}
	& R_{s,v,1} & R_{s,v,2}\\ \hline
	1&
	a_0c_0(c_{1,2},c_{2,2})(M_{\tr})_{1,2}&
	-a_0c_0(c_{1,2},c_{2,2})(M_{\det})_{1,2}\\
	2&
	a_0c_0(c_{1,2},c_{2,2})(M_{\tr})_{2,1}&
	-a_0c_0(c_{1,2},c_{2,2})(M_{\det})_{2,1}\\
	3&
	-a_0c_0(c_{1,2},c_{2,2})\bigl((M_{\tr})_{1,1}-(M_{\tr})_{2,2}\bigr)&
	a_0c_0(c_{1,2},c_{2,2})\bigl((M_{\det})_{1,1}-(M_{\det})_{2,2}\bigr).
\end{array}
$$
%CLEAN

Since $a_0c_0\neq0$ and $C$ is invertible, $(c_{1,2},c_{2,2})\neq(0,0)$. Hence $M_{\tr}=\mu I_2$ and $M_{\det}=\lambda I_2$ for some $\mu,\lambda$. Moreover,
$$
\mu^2=\det(M_{\tr})
=\det(D)\det(H_{\tr})\det(E)
=-a_sc_s\det(D)\det(E)\neq0,
$$
so $\mu\neq0$. Therefore
$$
0=\mu M_{\det}-\lambda M_{\tr}
=D^{\mathsf T}(\mu H_{\det}-\lambda H_{\tr})E.
$$
Since $D$ and $E$ are invertible, $\mu H_{\det}=\lambda H_{\tr}$. Comparing the $(2,2)$-entries gives $-\mu a_sc_s=0$, contradicting $\mu a_sc_s\neq0$. Thus at least one of $R_{s,v,1}$ and $R_{s,v,2}$ does not belong to $I(\SP_4)$.
%CLEAN

\subsection{Non-triviality check for \Cref{archrelexcm}}\label{app:arch-nontriviality}

We prove the non-triviality assertion in \Cref{archrelexcm}, retaining the notation of its proof. Since $I(\SP_4)$ is prime by \Cref{idealprime}, it suffices to prove that $R_{s,w'}\notin I(\SP_4)$ for every $w'\mid w$.
%CLEAN

Fix such a $w'$ and suppose, to the contrary, that $R_{s,w'}\in I(\SP_4)$. For $R_{s,w'}$, denoted by ``Ra'' in the code, the accompanying Mathematica code \cite{mathematicacode} computes the remainder upon division by a Gr\"obner basis of $I(\SP_4)$. For a monomial $M$ appearing in the remainder, write $c(M)$ for its coefficient. Under this assumption, all these coefficients vanish.
%CLEAN

As in the proof of \Cref{propisoga2fin}, the Hodge-basis choices give
$$
[\theta_s]_{dR}=\begin{pmatrix}A_s&0\\B_s&C_s\end{pmatrix},
\qquad
[\theta_0]_{dR}^{-1}=\begin{pmatrix}A_0&0\\B_0&C_0\end{pmatrix},
$$
where $A_P$ and $C_P$ are invertible for $P\in\{s,0\}$. 

Write $C_0=(c_{i,j})$, $A_s=(d_{i,j})$, and $C_s=(e_{i,j})$. By the choice of $\Gamma_{dR}(\CX/S)$ at the beginning of the section, $A_0=I_2$.
%CLEAN

The coefficient identities $c(X_{1,1}X_{4,4})=c(X_{1,2}X_{4,4})=0$ give $c_{1,1}d_{1,1}e_{1,2}=c_{2,1}d_{1,1}e_{1,2}=0$. Since $(c_{i,j})$ is invertible, $(c_{1,1},c_{2,1})\neq(0,0)$, and therefore $d_{1,1}e_{1,2}=0$.
%CLEAN

Assume first that $d_{1,1}=0$. Since $(d_{i,j})$ is invertible, we have $d_{1,2}d_{2,1}\neq0$. The coefficient identities
$$
c(X_{2,1}X_{3,4})=c_{2,1}d_{1,2}e_{1,1}=0,\qquad
c(X_{2,1}X_{4,4})=c_{2,1}d_{1,2}e_{1,2}=0
$$
give $c_{2,1}e_{1,1}=c_{2,1}e_{1,2}=0$. Since $(e_{i,j})$ is invertible, $(e_{1,1},e_{1,2})\neq(0,0)$, and hence $c_{2,1}=0$. Similarly, $c(X_{2,2}X_{3,4})=c(X_{2,2}X_{4,4})=0$ gives $c_{1,1}=0$, contradicting the invertibility of $(c_{i,j})$.
%CLEAN

Hence $d_{1,1}\neq0$, so $e_{1,2}=0$. Since $(e_{i,j})$ is invertible, we have $e_{1,1}e_{2,2}\neq0$. The coefficient identities
$$
c(X_{2,2}X_{4,4})=c_{1,1}d_{1,1}e_{1,1}=0,\qquad
c(X_{2,1}X_{4,4})=-c_{2,1}d_{1,1}e_{1,1}=0
$$
give $c_{1,1}=c_{2,1}=0$, contradicting the invertibility of $(c_{i,j})$.

\bibliographystyle{alpha}\bibliography{biblio}

@book {andre1989g,
	AUTHOR = {Andr\'{e}, Y.},
	TITLE = {{$G$}-functions and geometry},
	SERIES = {Aspects of Mathematics, E13},
	PUBLISHER = {Friedr. Vieweg \& Sohn, Braunschweig},
	YEAR = {1989},
	PAGES = {xii+229},
	ISBN = {3-528-06317-3},
	MRCLASS = {11J82 (11-02 11G10 11J87 12H25)},
	MRNUMBER = {990016},
	MRREVIEWER = {C. L. Stewart},
	DOI = {10.1007/978-3-663-14108-2},
	URL = {https://doi.org/10.1007/978-3-663-14108-2},
}

@article{PilaTsimerman2013,
	author  = {Pila, J. and Tsimerman, J.},
	title   = {The Andr{\'e}--Oort conjecture for the moduli space of abelian surfaces},
	journal = {Compositio Mathematica},
	volume  = {149},
	number  = {2},
	pages   = {204--216},
	year    = {2013},
	doi     = {10.1112/S0010437X12000589}
}

@misc{stacks-project,
	author       = {The Stacks Project Authors},
	title        = {Stacks Project},
	howpublished = {\url{https://stacks.math.columbia.edu}},
}

@book {pilabook,
	AUTHOR = {Pila, J.},
	TITLE = {Point-counting and the {Z}ilber-{P}ink conjecture},
	SERIES = {Cambridge Tracts in Mathematics},
	VOLUME = {228},
	PUBLISHER = {Cambridge University Press, Cambridge},
	YEAR = {2022},
	PAGES = {x+254},
	ISBN = {978-1-009-17032-1},
	MRCLASS = {11U09 (11-02 11D45 11G35)},
	MRNUMBER = {4420059},
	DOI = {10.1017/9781009170314},
	URL = {https://doi.org/10.1017/9781009170314},
}

@book {zilberpinkbook,
	AUTHOR = {Habegger, P. and R\'{e}mond, G. and Scanlon, T. and Ullmo, E. and
	Yafaev, A.},
	TITLE = {Around the {Z}ilber-{P}ink conjecture/{A}utour de la
	conjecture de {Z}ilber-{P}ink},
	SERIES = {Panoramas et Synth\`eses [Panoramas and Syntheses]},
	VOLUME = {52},
	NOTE = {Papers corresponding to courses from the conference ``Etats de
	la Recherche'' held at CIRM, Marseille, May 2011},
	PUBLISHER = {Soci\'{e}t\'{e} Math\'{e}matique de France, Paris},
	YEAR = {2017},
	PAGES = {xiv+284},
	ISBN = {978-2-85629-856-5},
	MRCLASS = {11-06 (03-02 03C60 03C64 14G35)},
	MRNUMBER = {3752657},
}

@article {bakkertsimermanandregroth,
	AUTHOR = {Bakker, B. and Tsimerman, J.},
	TITLE = {Functional transcendence of periods and the geometric
	{A}ndr\'e-{G}rothendieck period conjecture},
	JOURNAL = {Forum Math. Sigma},
	FJOURNAL = {Forum of Mathematics. Sigma},
	VOLUME = {13},
	YEAR = {2025},
	PAGES = {Paper No. e97, 24},
	ISSN = {2050-5094},
	MRCLASS = {14C30 (11J91)},
	MRNUMBER = {4922219},
	DOI = {10.1017/fms.2025.10036},
	URL = {https://doi.org/10.1017/fms.2025.10036},
}

@article {ayoub,
	AUTHOR = {Ayoub, J.},
	TITLE = {Une version relative de la conjecture des p\'eriodes de
	{K}ontsevich-{Z}agier},
	JOURNAL = {Ann. of Math. (2)},
	FJOURNAL = {Annals of Mathematics. Second Series},
	VOLUME = {181},
	YEAR = {2015},
	NUMBER = {3},
	PAGES = {905--992},
	ISSN = {0003-486X,1939-8980},
	MRCLASS = {32G20 (11J81 14F42)},
	MRNUMBER = {3296818},
	MRREVIEWER = {B. Muetzel},
	DOI = {10.4007/annals.2015.181.3.2},
	URL = {https://doi.org/10.4007/annals.2015.181.3.2},
}

@article {elkies,
	AUTHOR = {Elkies, N. D.},
	TITLE = {Supersingular primes for elliptic curves over real number
	fields},
	JOURNAL = {Compositio Math.},
	FJOURNAL = {Compositio Mathematica},
	VOLUME = {72},
	YEAR = {1989},
	NUMBER = {2},
	PAGES = {165--172},
	ISSN = {0010-437X,1570-5846},
	MRCLASS = {11G05},
	MRNUMBER = {1030140},
	MRREVIEWER = {J. H. Silverman},
	URL = {http://www.numdam.org/item?id=CM_1989__72_2_165_0},
}

@article {tsujipadicetale,
	AUTHOR = {Tsuji, T.},
	TITLE = {{$p$}-adic \'etale cohomology and crystalline cohomology in
	the semi-stable reduction case},
	JOURNAL = {Invent. Math.},
	FJOURNAL = {Inventiones Mathematicae},
	VOLUME = {137},
	YEAR = {1999},
	NUMBER = {2},
	PAGES = {233--411},
	ISSN = {0020-9910,1432-1297},
	MRCLASS = {14F30 (14F20)},
	MRNUMBER = {1705837},
	MRREVIEWER = {A. Mokrane},
	DOI = {10.1007/s002220050330},
	URL = {https://doi.org/10.1007/s002220050330},
}

@incollection {delignegabber,
	AUTHOR = {Deligne, P.},
	TITLE = {Le lemme de {G}abber},
	BOOKTITLE = {S\'{e}minaire sur les pinceaux arithm\'{e}tiques : la conjecture de Mordell},
	SERIES = {Ast\'{e}risque},
	NUMBER = {127},
	PAGES = {131--150},
	PUBLISHER = {Soci\'{e}t\'{e} Math\'{e}matique de France},
	YEAR = {1985},
	URL = {https://www.numdam.org/item/AST_1985__127__131_0/},
}

@misc {papaspadicgfuns2,
	AUTHOR = {Papas, G.},
	TITLE = {Local values of {$G$}-functions at {CM} points},
	YEAR = {2025},
	NOTE = {arXiv:2510.11814},
	EPRINT = {2510.11814},
	ARCHIVEPREFIX = {arXiv},
	PRIMARYCLASS = {math.NT},
	URL = {https://arxiv.org/abs/2510.11814},
}

@article{vonk,
	title={Crystalline cohomology of towers of curves},
	author={Vonk, J.},
	journal={International Mathematics Research Notices},
	volume={2020},
	number={21},
	pages={7454--7488},
	year={2020},
	publisher={Oxford University Press}
}

@misc {papaszpy1,
	AUTHOR = {Papas, G.},
	TITLE = {Zilber--{P}ink in ${Y}(1)^n$: Beyond multiplicative degeneration},
	YEAR = {2024},
	NOTE = {arXiv:2402.09487},
	EPRINT = {2402.09487},
	ARCHIVEPREFIX = {arXiv},
	PRIMARYCLASS = {math.NT},
	URL = {https://arxiv.org/abs/2402.09487},
}

@misc {daworrpap,
	AUTHOR = {Daw, C. and Orr, M. and Papas, G.},
	TITLE = {Some new cases of {Z}ilber--{P}ink in ${Y}(1)^3$},
	YEAR = {2025},
	NOTE = {arXiv:2510.09603},
	EPRINT = {2510.09603},
	ARCHIVEPREFIX = {arXiv},
	PRIMARYCLASS = {math.NT},
	URL = {https://arxiv.org/abs/2510.09603},
}

@incollection {hyodokato,
	AUTHOR = {Hyodo, O. and Kato, K.},
	TITLE = {Semi-stable reduction and crystalline cohomology with logarithmic poles},
	BOOKTITLE = {P\'{e}riodes {$p$}-adiques---S\'{e}minaire de Bures, 1988},
	EDITOR = {Fontaine, J.-M.},
	SERIES = {Ast\'{e}risque},
	NUMBER = {223},
	PAGES = {221--268},
	PUBLISHER = {Soci\'{e}t\'{e} Math\'{e}matique de France},
	YEAR = {1994},
	MRNUMBER = {1293974},
	MRREVIEWER = {A. Quir\'{o}s},
	URL = {https://www.numdam.org/item/AST_1994__223__221_0/},
}

@article {charlesisog,
	AUTHOR = {Charles, F.},
	TITLE = {Exceptional isogenies between reductions of pairs of elliptic
	curves},
	JOURNAL = {Duke Math. J.},
	FJOURNAL = {Duke Mathematical Journal},
	VOLUME = {167},
	YEAR = {2018},
	NUMBER = {11},
	PAGES = {2039--2072},
	ISSN = {0012-7094,1547-7398},
	MRCLASS = {11G05 (14G40)},
	MRNUMBER = {3843371},
	MRREVIEWER = {A. Javanpeykar},
	DOI = {10.1215/00127094-2018-0011},
	URL = {https://doi.org/10.1215/00127094-2018-0011},
}

@article {beukers,
	AUTHOR = {Beukers, F.},
	TITLE = {Algebraic values of {$G$}-functions},
	JOURNAL = {J. Reine Angew. Math.},
	FJOURNAL = {Journal f\"ur die Reine und Angewandte Mathematik. [Crelle's
	Journal]},
	VOLUME = {434},
	YEAR = {1993},
	PAGES = {45--65},
	ISSN = {0075-4102,1435-5345},
	MRCLASS = {11J91 (14D05)},
	MRNUMBER = {1195690},
	MRREVIEWER = {Y. Andr\'e},
	DOI = {10.1515/crll.1993.434.45},
	URL = {https://doi.org/10.1515/crll.1993.434.45},
}

@book {demazure,
	AUTHOR = {Demazure, M.},
	TITLE = {Lectures on {$p$}-divisible groups},
	SERIES = {Lecture Notes in Mathematics},
	VOLUME = {302},
	PUBLISHER = {Springer-Verlag, Berlin-New York},
	YEAR = {1972},
	PAGES = {v+98},
	MRCLASS = {14L05},
	MRNUMBER = {344261},
	MRREVIEWER = {L. D. Olson},
}

@book {langellfuns,
	AUTHOR = {Lang, S.},
	TITLE = {Elliptic functions},
	SERIES = {Graduate Texts in Mathematics},
	VOLUME = {112},
	EDITION = {Second},
	NOTE = {With an appendix by J. Tate},
	PUBLISHER = {Springer-Verlag, New York},
	YEAR = {1987},
	PAGES = {xii+326},
	ISBN = {0-387-96508-4},
	MRCLASS = {11F03 (11G05 11G07 11G15 14G25)},
	MRNUMBER = {890960},
	DOI = {10.1007/978-1-4612-4752-4},
	URL = {https://doi.org/10.1007/978-1-4612-4752-4},
}

@incollection {andremots,
	AUTHOR = {Andr\'{e}, Y.},
	TITLE = {Th\'{e}orie des motifs et interpr\'{e}tation g\'{e}om\'{e}trique des valeurs {$p$}-adiques de {$G$}-functions (une introduction)},
	BOOKTITLE = {Number theory (Paris, 1992--1993)},
	EDITOR = {David, S.},
	SERIES = {London Mathematical Society Lecture Note Series},
	VOLUME = {215},
	PAGES = {37--60},
	PUBLISHER = {Cambridge University Press, Cambridge},
	YEAR = {1995},
	DOI = {10.1017/CBO9780511661990.003},
	URL = {https://doi.org/10.1017/CBO9780511661990.003},
}

@article {ogus,
	AUTHOR = {Ogus, A.},
	TITLE = {{$F$}-isocrystals and de {R}ham cohomology. {II}. {C}onvergent
	isocrystals},
	JOURNAL = {Duke Math. J.},
	FJOURNAL = {Duke Mathematical Journal},
	VOLUME = {51},
	YEAR = {1984},
	NUMBER = {4},
	PAGES = {765--850},
	ISSN = {0012-7094,1547-7398},
	MRCLASS = {14F30 (11G10 14F40 14J28 14K99)},
	MRNUMBER = {771383},
	MRREVIEWER = {G. Horrocks},
	DOI = {10.1215/S0012-7094-84-05136-6},
	URL = {https://doi.org/10.1215/S0012-7094-84-05136-6},
}

@article {bertogus,
	AUTHOR = {Berthelot, P. and Ogus, A.},
	TITLE = {{$F$}-isocrystals and de {R}ham cohomology. {I}},
	JOURNAL = {Invent. Math.},
	FJOURNAL = {Inventiones Mathematicae},
	VOLUME = {72},
	YEAR = {1983},
	NUMBER = {2},
	PAGES = {159--199},
	ISSN = {0020-9910,1432-1297},
	MRCLASS = {14F30},
	MRNUMBER = {700767},
	MRREVIEWER = {G. Horrocks},
	DOI = {10.1007/BF01389319},
	URL = {https://doi.org/10.1007/BF01389319},
}

@incollection {bombg,
	AUTHOR = {Bombieri, E.},
	TITLE = {On {$G$}-functions},
	BOOKTITLE = {Recent progress in analytic number theory, {V}ol. 2 ({D}urham,
	1979)},
	PAGES = {1--67},
	PUBLISHER = {Academic Press, London-New York},
	YEAR = {1981},
	MRCLASS = {10F35 (33A70)},
	MRNUMBER = {637359},
	MRREVIEWER = {F. Beukers},
}

@book {berthbreenmessing,
	AUTHOR = {Berthelot, P. and Breen, L. and Messing, W.},
	TITLE = {Th\'eorie de {D}ieudonn\'e{} cristalline. {II}},
	SERIES = {Lecture Notes in Mathematics},
	VOLUME = {930},
	PUBLISHER = {Springer-Verlag, Berlin},
	YEAR = {1982},
	PAGES = {x+261},
	ISBN = {3-540-11556-0},
	MRCLASS = {14L05 (14F30)},
	MRNUMBER = {667344},
	MRREVIEWER = {L. B\'egueri},
	DOI = {10.1007/BFb0093025},
	URL = {https://doi.org/10.1007/BFb0093025},
}

@article {daworr,
	AUTHOR = {Daw, C. and Orr, M.},
	TITLE = {Unlikely intersections with {$E\times{CM}$} curves in
	{$\mathcal{A}_2$}},
	JOURNAL = {Ann. Sc. Norm. Super. Pisa Cl. Sci. (5)},
	FJOURNAL = {Annali della Scuola Normale Superiore di Pisa. Classe di
	Scienze. Serie V},
	VOLUME = {22},
	YEAR = {2021},
	NUMBER = {4},
	PAGES = {1705--1745},
	ISSN = {0391-173X},
	MRCLASS = {11G15 (11G18 14G35)},
	MRNUMBER = {4360602},
}

@article {daworr2,
	AUTHOR = {Daw, C. and Orr, M.},
	TITLE = {Quantitative reduction theory and unlikely intersections},
	JOURNAL = {International Mathematics Research Notices},
	VOLUME = {2022},
	YEAR = {2022},
	NUMBER = {20},
	PAGES = {16138--16195},
	ISSN = {1073-7928},
	DOI = {10.1093/imrn/rnab173},
	URL = {https://doi.org/10.1093/imrn/rnab173},
}

@misc{mathematicacode,
	author = {Papas, G.},
	title = {Accompanying {M}athematica code for
	{{O}n the $v$-adic values of ${G}$-functions}},
	year = {2026},
	doi = {10.5281/zenodo.23453166},
	url = {https://doi.org/10.5281/zenodo.23453166},
	note = {DOI: 10.5281/zenodo.23453166}
}

@article {daworr3,
	AUTHOR = {Daw, C. and Orr, M.},
	TITLE = {Lattices with skew-{H}ermitian forms over division algebras
	and unlikely intersections},
	JOURNAL = {J. \'Ec. polytech. Math.},
	FJOURNAL = {Journal de l'\'Ecole polytechnique. Math\'ematiques},
	VOLUME = {10},
	YEAR = {2023},
	PAGES = {1097--1156},
	ISSN = {2429-7100,2270-518X},
	MRCLASS = {11E39 (11G18 11R52)},
	MRNUMBER = {4600401},
	MRREVIEWER = {G. Fowler},
	DOI = {10.5802/jep.240},
	URL = {https://doi.org/10.5802/jep.240},
}

@article{daworr5,
	title={The large {G}alois orbits conjecture under multiplicative degeneration},
	author={Daw, C. and Orr, M.},
	journal={arXiv preprint arXiv:2306.13463},
	year={2023}
}

@article{dawren,
	title={Applications of the hyperbolic {A}x-{S}chanuel conjecture},
	author={Daw, C. and Ren, J.},
	journal={Compositio Mathematica},
	volume={154},
	number={9},
	pages={1843--1888},
	year={2018},
	publisher={Foundation Compositio Mathematica}
}

@article {masserwu,
	AUTHOR = {Masser, D. and W\"{u}stholz, G.},
	TITLE = {Isogeny estimates for abelian varieties, and finiteness
	theorems},
	JOURNAL = {Ann. of Math. (2)},
	FJOURNAL = {Annals of Mathematics. Second Series},
	VOLUME = {137},
	YEAR = {1993},
	NUMBER = {3},
	PAGES = {459--472},
	ISSN = {0003-486X},
	MRCLASS = {11G10 (14K02)},
	MRNUMBER = {1217345},
	DOI = {10.2307/2946529},
	URL = {https://doi.org/10.2307/2946529},
}

@article {mwendoesti,
	AUTHOR = {Masser, D. W. and W\"{u}stholz, G.},
	TITLE = {Endomorphism estimates for abelian varieties},
	JOURNAL = {Math. Z.},
	FJOURNAL = {Mathematische Zeitschrift},
	VOLUME = {215},
	YEAR = {1994},
	NUMBER = {4},
	PAGES = {641--653},
	ISSN = {0025-5874},
	MRCLASS = {14K15 (11G10 14K02)},
	MRNUMBER = {1269495},
	MRREVIEWER = {S. Abdulali},
	DOI = {10.1007/BF02571735},
	URL = {https://doi.org/10.1007/BF02571735},
}

@article {pilazannier,
	author = {Pila, J. and Zannier, U.},
	TITLE = {Rational points in periodic analytic sets and the
	{M}anin-{M}umford conjecture},
	JOURNAL = {Atti Accad. Naz. Lincei Rend. Lincei Mat. Appl.},
	FJOURNAL = {Atti della Accademia Nazionale dei Lincei. Rendiconti Lincei.
	Matematica e Applicazioni},
	VOLUME = {19},
	YEAR = {2008},
	NUMBER = {2},
	PAGES = {149--162},
	ISSN = {1120-6330},
	MRCLASS = {11J95 (11G10 14G05)},
	MRNUMBER = {2411018},
	MRREVIEWER = {P. Tzermias},
	DOI = {10.4171/RLM/514},
	URL = {https://doi.org/10.4171/RLM/514},
}

@article {serretate,
	AUTHOR = {Serre, J.-P. and Tate, J.},
	TITLE = {Good reduction of abelian varieties},
	JOURNAL = {Ann. of Math. (2)},
	FJOURNAL = {Annals of Mathematics. Second Series},
	VOLUME = {88},
	YEAR = {1968},
	PAGES = {492--517},
	ISSN = {0003-486X},
	MRCLASS = {14.51},
	MRNUMBER = {236190},
	MRREVIEWER = {M. J. Greenberg},
	DOI = {10.2307/1970722},
	URL = {https://doi.org/10.2307/1970722},
}

@incollection {siegel,
	AUTHOR = {Siegel, C. L.},
	TITLE = {\"{U}ber einige {A}nwendungen diophantischer {A}pproximationen
	[reprint of {A}bhandlungen der {P}reu\ss ischen {A}kademie der
	{W}issenschaften. {P}hysikalisch-mathematische {K}lasse 1929,
	{N}r. 1]},
	BOOKTITLE = {On some applications of {D}iophantine approximations},
	SERIES = {Quad./Monogr.},
	VOLUME = {2},
	PAGES = {81--138},
	PUBLISHER = {Ed. Norm., Pisa},
	YEAR = {2014},
	MRCLASS = {11G30 (01A75 11Jxx)},
	MRNUMBER = {3330350},
}

@article {silverberg,
	AUTHOR = {Silverberg, A.},
	TITLE = {Fields of definition for homomorphisms of abelian varieties},
	JOURNAL = {J. Pure Appl. Algebra},
	FJOURNAL = {Journal of Pure and Applied Algebra},
	VOLUME = {77},
	YEAR = {1992},
	NUMBER = {3},
	PAGES = {253--262},
	ISSN = {0022-4049},
	MRCLASS = {14K02 (11G10)},
	MRNUMBER = {1154704},
	MRREVIEWER = {D. Grant},
	DOI = {10.1016/0022-4049(92)90141-2},
	URL = {https://doi.org/10.1016/0022-4049(92)90141-2},
}

@book {voight,
	AUTHOR = {Voight, J.},
	TITLE = {Quaternion algebras},
	SERIES = {Graduate Texts in Mathematics},
	VOLUME = {288},
	PUBLISHER = {Springer, Cham},
	YEAR = {2021},
	PAGES = {xxiii+885},
	ISBN = {978-3-030-56692-0; 978-3-030-56694-4},
	MRCLASS = {11R52 (11-02 11E12 11F06 16H05 16U60 20H10)},
	MRNUMBER = {4279905},
	MRREVIEWER = {J. Brzezi\'{n}ski},
	DOI = {10.1007/978-3-030-56694-4},
	URL = {https://doi.org/10.1007/978-3-030-56694-4},
}

@book {hhorne,
	AUTHOR = {Hartshorne, R.},
	TITLE = {Algebraic geometry},
	SERIES = {Graduate Texts in Mathematics},
	VOLUME = {52},
	PUBLISHER = {Springer-Verlag, New York-Heidelberg},
	YEAR = {1977},
	PAGES = {xvi+496},
	ISBN = {0-387-90244-9},
	MRCLASS = {14-01},
	MRNUMBER = {0463157},
	MRREVIEWER = {R. Speiser},
}

@article{papasbigboi,
	title={Unlikely intersections in the {T}orelli locus and the {G}-functions method},
	author={Papas, G.},
    journal={arXiv preprint arXiv:2201.11240},
    year={2022}
}

@article {papaseffbrsieg,
	AUTHOR = {Papas, G.},
	TITLE = {Effective {B}rauer--{S}iegel on some curves in ${Y}(1)^n$},
	JOURNAL = {Math. Ann.},
	FJOURNAL = {Mathematische Annalen},
	VOLUME = {394},
	YEAR = {2026},
	NUMBER = {4},
	PAGES = {Paper No. 84},
	DOI = {10.1007/s00208-026-03414-8},
	URL = {https://doi.org/10.1007/s00208-026-03414-8},
}

@article{papaspadicgfuns3,
	author = {Papas, G.},
	title = {{G}-functions and quaternionic multiplication in {$\mathcal A_2$}},
	journal={arXiv preprint arXiv:2510.19683},
	year = {2025}
}

@article {daworr4,
	AUTHOR = {Daw, C. and Orr, M.},
	TITLE = {Zilber-{P}ink in a product of modular curves assuming multiplicative degeneration},
	JOURNAL = {Duke Math. J.},
	FJOURNAL = {Duke Mathematical Journal},
	VOLUME = {174},
	YEAR = {2025},
	NUMBER = {13},
	PAGES = {2877--2926},
	DOI = {10.1215/00127094-2025-0011},
	URL = {https://doi.org/10.1215/00127094-2025-0011},
}

@book {silvermanell,
	AUTHOR = {Silverman, J. H.},
	TITLE = {The arithmetic of elliptic curves},
	SERIES = {Graduate Texts in Mathematics},
	VOLUME = {106},
	PUBLISHER = {Springer-Verlag, New York},
	YEAR = {1986},
	PAGES = {xii+400},
	ISBN = {0-387-96203-4},
	MRCLASS = {11G05 (14Gxx 14K07 14K15)},
	MRNUMBER = {817210},
	MRREVIEWER = {R. S. Rumely},
	DOI = {10.1007/978-1-4757-1920-8},
	URL = {https://doi.org/10.1007/978-1-4757-1920-8},
}
\end{document}